\documentclass[11pt]{amsart}

\usepackage{amsmath,amsthm,amssymb,mathrsfs,mathtools,color}
\usepackage{bm} \usepackage{stmaryrd}

\usepackage[top=1in, bottom=1in, left=1in, right=1in]{geometry}
\usepackage[skip=4pt, indent=0pt]{parskip}

\usepackage{float}
\usepackage{graphicx}
\usepackage{etoolbox}

\usepackage[utf8]{inputenc}
\usepackage[T1]{fontenc}

\usepackage{microtype} 

\usepackage{enumitem}
\setlist[enumerate]{itemsep=2pt,parsep=2pt,before={\parskip=2pt}}
\setlist[itemize]{itemsep=2pt,parsep=2pt,before={\parskip=2pt}}
\usepackage{longtable}

\usepackage[colorlinks=true,hyperindex, linkcolor=red!60, pagebackref=false, citecolor=cyan, pdfpagelabels, hypertexnames=false, linktoc=all]{hyperref}
\usepackage[capitalize]{cleveref}
\crefname{equation}{}{}
\crefformat{enumi}{#2\textup{\textcolor{black}{(}#1\textcolor{black}{)}}#3}

\usepackage{appendix}
\AtBeginEnvironment{appendices}{\crefalias{section}{appendix}}

\usepackage{url}

\usepackage{tikz}
\usetikzlibrary{calc}
\usetikzlibrary{math}
\usetikzlibrary{positioning}

\usepackage{pdfpages}

\usepackage{thmtools}
\usepackage{thm-restate}

\usepackage{empheq}

\makeatletter
\def\thm@space@setup{\thm@preskip=3.5pt \thm@postskip=3.5pt }
\makeatother

\newtheorem{theorem}{Theorem}[section]
\newtheorem*{theorem*}{Theorem}
\newtheorem*{definition*}{Definition}
\newtheorem{proposition}[theorem]{Proposition}
\newtheorem{lemma}[theorem]{Lemma}
\newtheorem{corollary}[theorem]{Corollary}
\newtheorem{conjecture}[theorem]{Conjecture}

\theoremstyle{definition}
\newtheorem{definition}[theorem]{Definition}

\newtheorem{remark}[theorem]{Remark}
\newtheorem{example}[theorem]{Example}
\newtheorem{notation}[theorem]{Notation}

\renewcommand{\le}{\leqslant}
\renewcommand{\leq}{\leqslant}
\renewcommand{\ge}{\geqslant}
\renewcommand{\geq}{\geqslant}
\renewcommand{\Re}{\mathrm{Re}}
\renewcommand{\Im}{\mathrm{Im}}

\newcommand{\norm}[1]{\left\|#1\right\|}

\newcommand{\mnorm}[1]{\big\|#1\big\|}

\newcommand{\abs}[1]{\left|#1\right|}
\newcommand{\ind}[1]{\mathbf{1}_{#1}}

\newcommand{\eps}{\varepsilon}
\renewcommand{\phi}{\varphi}

\newcommand{\R}{\mathbb{R}}
\newcommand{\C}{\mathbb{C}}
\newcommand{\Z}{\mathbb{Z}}
\newcommand{\N}{\mathbb{N}}

\renewcommand{\Pr}[1]{{\,\mathbb{P}\!\left(#1\right)}}
\newcommand{\E}[1]{\,\mathbb{E}\!\left[#1\right]}
\newcommand{\EE}{\mathbb{E}}
\newcommand{\spmod}[1]{\,(\mathrm{mod}{\, #1})}
\newcommand{\rad}{\mathrm{rad}}

\newcommand{\Q}{Q}
\newcommand{\A}{\mathcal{A}}
\newcommand{\B}{\mathcal{B}}
\renewcommand{\P}{\mathcal{P}}
\newcommand{\T}{\mathcal{T}}
\newcommand{\U}{\mathcal{U}}
\newcommand{\V}{\mathcal{V}}
\newcommand{\CC}{\mathcal{C}}
\newcommand{\RR}{\mathcal{R}}
\newcommand{\NN}{\mathcal{N}}

\newcommand{\dist}[2]{d_{#1}(#2)}

\newcommand{\odist}{\mathrm{dist}}

\newcommand{\None}{N_{\tinyparallelogram}}
\newcommand{\Ntwo}{N_{\tinygraphicon}}

\newcommand{\z}{\mathbf{z}}
\newcommand{\Clift}{C_0}
\newcommand{\Cclusters}{C_1}

\makeatletter

\renewcommand{\tocsection}[3]{\indentlabel{\@ifnotempty{#2}{\makebox[1.75em][l]{#2}}}#3}

\newcommand\@dotsep{4.5}
\def\@tocline#1#2#3#4#5#6#7{\relax
  \ifnum #1>\c@tocdepth \else
    \par \addpenalty\@secpenalty\addvspace{#2}\begingroup \hyphenpenalty\@M
    \@ifempty{#4}{\@tempdima\csname r@tocindent\number#1\endcsname\relax
    }{\@tempdima#4\relax
    }\parindent\z@ \leftskip#3\relax \advance\leftskip\@tempdima\relax
    \rightskip\@pnumwidth plus1em \parfillskip-\@pnumwidth
    #5\leavevmode\hskip-\@tempdima{#6}\nobreak
    \leaders\hbox{$\m@th\mkern \@dotsep mu\hbox{.}\mkern \@dotsep mu$}\hfill
    \nobreak
    \hbox to\@pnumwidth{\@tocpagenum{#7}}\par
    \nobreak
    \endgroup
  \fi}

\def\l@subsection{\@tocline{2}{0pt}{2.5pc}{5pc}{}}
\makeatother

\renewcommand{\thepart}{\Roman{part}} 

\makeatletter
\renewcommand{\part}[1]{\refstepcounter{part}\phantomsection 
  
  \addtocontents{toc}{\protect\vspace{-12pt}} 
  \addcontentsline{toc}{part}{\partname\ \thepart.\ #1}

  \vspace*{0.5cm}\begin{center}\Large\bfseries \partname\ \thepart.\ #1\end{center}\vspace{0.5cm}

  \@afterindentfalse\@afterheading
}
\makeatother

\begin{document}

\title[Improved bounds for the Fourier uniformity conjecture]{Improved bounds for the Fourier uniformity conjecture}

\author{C\'edric Pilatte}
\address{Mathematical Institute, University of Oxford.}
\email{cedric.pilatte@maths.ox.ac.uk}

\begin{abstract}
    Let $\lambda$ denote the Liouville function. We prove that
    $$\sum_{X \leq x < 2X} \sup_{\alpha \in \R/\Z} \bigg\lvert\!\sum_{x \leq n < x+H} \lambda(n) e(n\alpha)\bigg\rvert = o(HX)$$
    as $X\to \infty$, in the regime $H = H(X) \geq \exp({(\log X)^{2/5+\eps}})$. This improves upon a result of Walsh towards the Fourier uniformity conjecture.
\end{abstract}

\maketitle

{
    \hypersetup{linkcolor=black}
    \setlength{\parskip}{0pt} \tableofcontents
}

\part{Introduction}
\label{part:intro}
\section{Introduction}
\subsection{The Fourier uniformity conjecture}
Let $\lambda : \N \to \{+1, -1\}$ be the Liouville function, defined by $\lambda(n) := (-1)^{\Omega(n)}$, where $\Omega(n)$ is the number of prime factors of $n$ counted with multiplicity. The pseudorandom behaviour of $\lambda$ is intimately connected to the distribution of prime numbers. For example, the conjectured square-root cancellation estimate $\sum_{n\leq X} \lambda(n) \ll X^{1/2+\eps}$ is equivalent to the Riemann Hypothesis.

The~celebrated Matomäki-Radziwi{\l\l} theorem~\cite{MR} shows that $\lambda$ exhibits cancellation in almost all short intervals. More precisely, for any function $H = H(X)$ tending to infinity with $X$, we have
\begin{equation*}
    \sum_{X \leq x < 2X} \bigg\lvert\!\sum_{x \leq n < x+H} \lambda(n)\bigg\rvert = o(HX)
\end{equation*}
as $X\to \infty$. Remarkably, this estimate holds for arbitrarily slowly growing $H$.

For applications to Chowla's conjecture and related problems (see \cref{sec:chowla}), a stronger form of pseudorandomness is required: the \emph{Fourier uniformity} of the Liouville function in short intervals.

\begin{conjecture}[Fourier uniformity conjecture]
    \label{conj:FUC}
    For any function $H = H(X)$ tending to infinity with $X$, we have
    \begin{equation}
        \label{eq:FUCstatement}
        \sum_{X \leq x < 2X} \sup_{\alpha \in \R/\Z} \bigg\lvert\!\sum_{x \leq n < x+H} \lambda(n) e(n\alpha)\bigg\rvert = o(HX)
    \end{equation}
    as $X\to \infty$.
\end{conjecture}

Informally, \cref{conj:FUC} asserts that $\lambda$ has negligible correlations with linear phases in almost all short intervals.

A crucial feature of \cref{conj:FUC} is that the frequency $\alpha$ is permitted to depend on the interval $[x, x+H)$ in an arbitrary way. A weaker version of \cref{eq:FUCstatement}, in which the supremum over $\alpha$ is taken outside the sum over $x$, was obtained by Matomäki, Radziwi{\l\l}, and Tao~\cite{MRT} shortly after the original Matomäki-Radziwi{\l\l} theorem.

As one might expect, \cref{conj:FUC} becomes more difficult to establish for shorter interval lengths~$H$. The conjecture was initially proved for $H \geq X^{\eps}$ by Matomäki, Radziwi{\l\l}, and Tao~\cite{MRTinventiones}, with a simpler proof later given by Walsh~\cite{walsh1}. This threshold was lowered to ${H \geq \exp((\log X)^{5/8+\eps})}$ by Matomäki, Radziwi{\l\l}, Tao, Ter\"av\"ainen, and Ziegler~\cite{MRTTZ}, and further improved to $H \geq \exp((\log X)^{1/2+\eps})$ by Walsh~\cite{walsh2}.

As explained in \cite{MRTTZ} and \cite{walsh2}, a severe bottleneck emerges at the scale
\begin{equation}
    \label{eq:barrier}
    H\approx \exp((\log X)^{1/2}).
\end{equation}
Indeed, all known approaches require an iterative argument consisting of $k$ steps, where $k \asymp \frac{\log X}{\log H}$ is the diameter of a specific underlying graph. The most natural ways to perform such an iteration incur $k!$-type losses, which become overwhelming once $H$ drops below the threshold \cref{eq:barrier}.

Assuming the Generalised Riemann Hypothesis, Walsh~\cite{walsh3} was able to bypass this obstacle, proving \cref{conj:FUC} for intervals of length $H \geq (\log X)^{\psi(X)}$, where $\psi(X)$ tends to infinity arbitrarily slowly.

In this paper, we overcome this barrier unconditionally. Our main result is the following.

\begin{theorem}
    \label{thm:main}
    Let $\eps>0$. For any function $H = H(X)$ satisfying $H \geq \exp((\log X)^{2/5+\eps})$, we have
    \begin{equation*}
        \sum_{X \leq x < 2X} \sup_{\alpha \in \R/\Z} \bigg\lvert\!\sum_{x \leq n < x+H} \lambda(n) e(n\alpha)\bigg\rvert = o(HX)
    \end{equation*}
    as $X\to \infty$.
\end{theorem}

In fact, we prove a more general version of this result for arbitrary non-pretentious multiplicative functions, with explicit quantitative savings on the order of $\sqrt{\log \log X}$; see \cref{thm:maingeneral} for the full technical statement.

\subsection{Relation to Chowla's conjecture}
\label{sec:chowla}
The primary motivation for studying the Fourier uniformity conjecture is its close connection to Chowla's conjecture~\cite{chowla}, which states that for any fixed $k\geq 1$ and distinct positive integers $h_1, \dots, h_k$, we have
\begin{equation}
    \label{eq:chowla}
    \sum_{n \leq X} \lambda(n+h_1) \cdots \lambda(n+h_k) = o(X)
\end{equation}
as $X\to \infty$.

Chowla's conjecture remains wide open for $k\geq 2$. However, significant progress has been made on weaker variants of the conjecture, and these partial results have been essential in resolving problems of independent interest (see e.g.~\cite{discrepancy,TTapplications}). Most notably, the \emph{logarithmic Chowla conjecture}, which asserts that
\begin{equation}
    \label{eq:logchowla}
    \sum_{n \leq X} \frac{\lambda(n+h_1) \cdots \lambda(n+h_k)}{n} = o(\log X)
\end{equation}
as $X\to \infty$, was established for $k=2$ by Tao~\cite{T} and for all \emph{odd} $k\geq 3$ by Tao and Ter\"av\"ainen~\cite{TTduke}.

The proof of \cref{eq:logchowla} for two-point correlations fundamentally relies on the Matomäki-Radziwi{\l\l}-Tao theorem~\cite{MRT} on short exponential sums of the Liouville function. As it turns out, the odd order cases are amenable to a parity trick that avoids such deep number-theoretic inputs~\cite{TTodd,TTduke}.

In~\cite{taoequiv}, Tao proved the equivalence of three different statements:
\begin{enumerate}[label=(\roman*),ref=\roman*]
    \item \label{item:conj1} the logarithmic Chowla conjecture (for all $k\geq 1$);
    \item \label{item:conj2} the logarithmic Sarnak conjecture;
    \item \label{item:conj3} the logarithmic higher-order Fourier uniformity conjecture.
\end{enumerate}
Here, the logarithmic Sarnak conjecture is the assertion that $\lambda$ has negligible logarithmic correlation with any deterministic sequence.\footnote{That is, any sequence arising from a topological dynamical system of zero entropy; see \cite{sarnak} or \cite{taoequiv} for the relevant definitions.} The \emph{higher-order} Fourier uniformity conjecture is a natural generalisation of \cref{conj:FUC} in which the linear phases $e(n\alpha)$ are replaced by nilsequences of a given step, and~\cref{item:conj3} is a logarithmically weighted version of that statement. The deduction of~\cref{item:conj1} from \cref{item:conj3} relies on the \emph{entropy decrement argument}, first introduced by Tao in \cite{T}.

Following this equivalence, a potential path towards the general logarithmic Chowla conjecture would be to prove the higher-order Fourier uniformity conjecture. Crucially, it is not necessary to consider arbitrarily slowly growing $H$: it was shown in \cite[Proposition~1.7]{MRTTZ} that establishing the conjecture for $H \geq (\log X)^{\eps}$ for every $\eps>0$ is sufficient to deduce the logarithmic Chowla conjecture. Walsh's result under GRH~\cite{walsh3} comes remarkably close to this threshold, as it proves the Fourier uniformity conjecture for $H \geq (\log X)^{\psi(X)}$ with $\psi(X)$ tending to infinity arbitrarily slowly. Lowering this bound further, even conditionally, remains a significant open problem.

An alternative approach towards the logarithmic Chowla conjecture proceeds via expander graph techniques. Introduced by Helfgott and Radziwi{\l\l}~\cite{HR} and further developed by the author~\cite{pilatte}, this method has been successfully applied to the $k=2$ case of \cref{eq:logchowla}, yielding substantially stronger quantitative bounds than the entropy method. Furthermore, the required number-theoretic input is less demanding: the proof only uses the Matomäki-Radziwi{\l\l}-Tao theorem for intervals of length ${H = \exp((\log X)^{1/2-c})}$, for an arbitrarily small constant $c>0$.

The expander graph approach relies on trace methods that currently lack suitable generalisations to hypergraphs. Nevertheless, if a higher-order analogue of the two-point decoupling estimate could be established,\footnote{See~\cite[Section~9.5]{HR} for a more extensive discussion.} such a result would reduce the logarithmic Chowla conjecture to the higher-order Fourier uniformity conjecture at the $H = \exp((\log X)^{1/2-c})$ threshold. Prior to the present paper, the best unconditional bounds for the Fourier uniformity conjecture, due to Walsh~\cite{walsh2}, were limited to the range $H \geq \exp((\log X)^{1/2+\eps})$. By lowering this exponent to $2/5$, \cref{thm:main} would now provide the necessary number-theoretic ingredient, at least for linear phases.

It would then remain to generalise \cref{thm:main} to arbitrary nilsequences. However, the main conceptual difficulties are generally expected to manifest already in the linear setting. Although technically intricate, the task of transferring Fourier uniformity results from linear phases to nilsequences has been successfully accomplished in previous works. In particular, Matomäki, Radziwi{\l\l}, Tao, Ter\"av\"ainen, and Ziegler~\cite{MRTTZ} have proved the higher-order Fourier uniformity conjecture in the range $H \geq X^{\eps}$ for any $\eps>0$. See also \cite{higheralmostall,higherall} for further results on higher-order uniformity of arithmetic functions in short intervals.

\subsection{Overview of the proof and comparison with previous work}
The general framework of our proof follows that of all previous approaches to the Fourier uniformity conjecture. Assuming that \cref{thm:main} fails, there are many short intervals $[x, x+H)$ on which $\lambda$ has a large correlation with some linear phase $e(n\alpha_x)$, where $\alpha_x$ depends on $x$ in an uncontrolled manner. Still, by exploiting the multiplicativity of~$\lambda$ at small primes, these frequencies $\alpha_x$ can be shown to exhibit certain weak dependencies. The heart of the proof is to show that, together, these dependencies are rigid enough to force a global structure on the $\alpha_x$. Once this is established, the Matomäki-Radziwi{\l\l} theorem can be applied to obtain a contradiction.

The central task is to analyse the following setup. Let $A\subset [Y, 2Y]$ be a well-spaced set of size $|A| \asymp Y$, where $Y := X/H$, with a frequency $\alpha_x$ assigned to each $x\in A$. Let $\P$ be a set of primes $p\asymp P$, where the scale $P$ is a small power of $H$. In this setting, the weak dependencies amount to the existence of $\gg Y |\P|^2$ quadruples $(x,y,p,q) \in A^2 \times \P^2$ satisfying $px \approx qy$ and ${q\alpha_x \approx p\alpha_y \pmod 1}$.

The graph formed by interpreting these dependencies as edges on the vertex set $A$ is highly sparse, with~$\asymp Y$ vertices and an average degree of $\asymp |\P|^2$. Propagating this local information to extract a global structure therefore requires an iterative argument with $\asymp \frac{\log Y}{\log P}$ steps. As in~\cite{walsh1,walsh2,walsh3}, our iterative procedure involves the construction of auxiliary frequencies that are `lifts' of the original~$\alpha_x$ and satisfy a similar system of relations. In previous unconditional approaches, the density of relations satisfied by the auxiliary frequencies decays by at least a constant factor at every step; such accumulated losses become fatal below the $H \approx \exp((\log X)^{1/2})$ barrier.

We take the conditional approach of Walsh~\cite{walsh3} as our starting point. Working under GRH, Walsh introduced a `single-prime lifting' argument in which the auxiliary frequencies retain almost all of the relations satisfied by the original $\alpha_x$. This is achieved via a \emph{local structure theorem} that partitions~$A$ into a bounded number of components, all but one of which can be successfully lifted. To control this partition, he employs an expansion lemma based on RH, which ensures that no two subsets of $A$ share significantly more edges than the random model predicts. This expansion property forces the exceptional component to be negligibly small. After iterating this procedure $\asymp \frac{\log Y}{\log P}$ times, GRH is applied once more to extract the required global structure from the relations between the constructed frequencies.

Our main new ingredient is a stronger local structure theorem (\cref{thm:structure}) that is robust enough to allow for an unconditional iterative argument. It differs from Walsh's local structure theorem in two key respects:
\begin{itemize}
    \item It is \emph{quantitative}, yielding a decomposition of $A$ where the number of parts can be chosen to be almost as large as $\log P$, rather than a fixed constant.
    \item It is \emph{relative}, requiring only that the number of relations be large relative to the size of $A$ (rather than an absolute density), which allows the theorem to apply even to very sparse sets.
\end{itemize}
Inspired by Walsh's methods, the proof of \cref{thm:structure} adopts a different viewpoint centred on the quantitative analysis of substructures called `clusters', and requires new tools such as a Vitali-type covering lemma for these clusters.

Iterating this local structure theorem allows us to construct the auxiliary frequencies while working with potentially very sparse sets throughout. To analyse the resulting construction, we substitute GRH with an unconditional expansion estimate derived from Vinogradov-Korobov type bounds for Dirichlet $L$-functions. This expansion estimate is significantly weaker than its conditional counterpart, yielding non-trivial information only for very sparse sets. Nonetheless, our combinatorial framework is robust enough that this input suffices to establish the required global structure.

This conclusion is the content of our \emph{global structure theorem} (\cref{thm:globalstructuretheorem}), which roughly states that, on a large subset $A'\subset A$, the frequencies $\alpha_x$ satisfy an approximate formula of the form $\alpha_x \approx a_0/q_0 + T_0/x \pmod 1$ for some fixed $a_0, q_0$ and $T_0$ of controlled size. However, a strong version of the Matomäki-Radziwiłł theorem shows that the Liouville function can only correlate with such linear phases on a set of density at most $H^{-c}$ for some small $c>0$. By contrast, our Vinogradov-Korobov expansion estimate guarantees that $A'$ has density at least $\exp(-(\log X)^{1-3\theta/2+o(1)})$ when $H = \exp((\log X)^{\theta})$. This yields a contradiction for any $\theta > 2/5$.

The above argument requires a variant of the Matomäki-Radziwiłł theorem that allows for large values of $q_0$ while retaining a power-saving bound for the exceptional set. As the precise quantitative statement we need does not appear to be recorded in the literature, we provide a proof in \cref{sec:appendix_mr}, drawing entirely on ideas from \cite{MR,MRII,MRTTZ}.

These methods naturally extend to arbitrary non-pretentious multiplicative functions, yielding the following general result.

\begin{theorem}[Fourier uniformity estimate]
    \label{thm:maingeneral}
    Let $f : \N \to \C$ be a $1$-bounded multiplicative function.

    Let $0 < \eps < 1/5$. Let $X$ be sufficiently large in terms of $\eps$ and let $H := \exp((\log X)^{2/5+\eps})$.

    Assume that
    \begin{equation*}
        \sum_{X \leq x < 2X} \sup_{\alpha \in \R/\Z} \bigg\lvert\!\sum_{x \leq n < x+H} f(n) e(n\alpha)\bigg\rvert \geq \delta H X,
    \end{equation*}
    where ${\delta \geq C_0(\log \log X)^{-1/2}}$ for some sufficiently large absolute constant $C_0 > 0$.

    Then, $f$ is pretentious in the sense that
    \begin{equation*}
        \sum_{p \leq X} \frac{1-\Re(f(p){\chi(p)}p^{it})}{p} \ll \log \delta^{-1}
    \end{equation*}
    for some Dirichlet character $\chi$ of conductor $O(\delta^{-O(1)})$ and some real number $t$ with ${|t| \leq X^2/H^{2-o(1)}}$.
\end{theorem}

For simplicity, we have stated \cref{thm:maingeneral} for a specific choice of $H$ near the threshold $\exp((\log X)^{2/5})$; the regime of larger $H$ is treated in \cite{walsh2}. The deduction of \cref{thm:main} from \cref{thm:maingeneral} follows from standard Vinogradov-Korobov type non-pretentiousness estimates for the Liouville function.

We note that, unlike previous results, \cref{thm:maingeneral} is not restricted to the dense regime $\delta \asymp 1$, as it allows for quantitative savings of order $\sqrt{\log \log X}$. This rate of decay is the best possible for an approach based on Elliott's inequality, which serves as the starting point for our approach and all earlier work.

The quantitative methods of this paper allow us to recover a version of Walsh's Fourier uniformity estimate under GRH \cite{walsh3} with explicit savings; we do so in \cref{sec:GRHcase}. The present paper is self-contained, and no prior familiarity with \cite{walsh3} is needed to follow the exposition.

Finally, we remark that the exponent $2/5$ in \cref{thm:maingeneral} represents the natural limit of our unconditional methods. For $H$ below this threshold, we do not know how to rule out the possibility that all connected components in the graph of relations are extremely sparse (say, of density less than~$1/H$). In such a scenario, one cannot hope to establish a global structure without extracting more information from the Liouville correlations than the current Elliott inequality argument provides.

\subsection{Notation}
We write $f = O(g)$ or $f \ll g$ if $|f| \leq C g$ for some absolute constant $C > 0$, and $f\asymp g$ if $f \ll g \ll f$. A set $S\subset \R$ is \emph{$t$-separated} if $|x-y|\geq t$ for all distinct $x,y\in S$. For a finite set $S$, we let $\EE_{x\in S} f(x)$ denote the average $\frac{1}{|S|}\sum_{x\in S} f(x)$. Finally, for $\alpha\in \R$ and $X\subset \R$, we write $\odist(\alpha, X) := \inf_{x\in X}|\alpha-x|$ and $\norm{\alpha} := \odist(\alpha, \Z)$.

\addtocontents{toc}{\protect\setcounter{tocdepth}{-1}}
\section*{Acknowledgements}
\addcontentsline{toc}{section}{Acknowledgements}
\addtocontents{toc}{\protect\setcounter{tocdepth}{1}}The author is supported by the Oxford Mathematical Institute, a Saven European Scholarship, a Jane Street Graduate Research Fellowship and the Vocatio foundation. I am very grateful to my advisors, Ben Green and James Maynard, for their continued support and guidance. \section{Definitions and reductions}
\label{sec:definitions}

In this section, we introduce the main definitions and outline the overall strategy of the proof. We state our \emph{local} and \emph{global} structure theorems, and deduce our main Fourier uniformity estimate from the latter.

\subsection{Configurations and lifts}

We fix the following global parameters for the rest of the article.

\begin{definition}[Global parameters]
    \label{def:globalparam}
    Let $P\geq 10^3$, $Y\geq P^3$ and let $\P$ be the set of primes in $[P, 2P]$.
\end{definition}

The basic objects of study in this paper are discrete subsets $A\subset [Y, 2Y)$, where each point $x\in A$ carries a frequency $\alpha_x\in \R/\Z$. We package this data into the following definition.

\begin{definition}[Configuration]
    \label{def:configuration}
    A \emph{configuration} is a triple $(A, \alpha_{\bullet}, H)$ where:
    \begin{itemize}
        \item $A\subset [Y, 2Y)$ is a non-empty $1$-separated set;\footnote{Contrary to the convention in \cite{walsh1,walsh3}, our setup is organised so that the ambient interval $[Y,2Y)$ of any configuration remains the same throughout the proof, while the new parameter $H$ is permitted to vary.}
        \item $\alpha_{\bullet} := (\alpha_x)_{x\in A}$ is a family of elements in $\R/\Z$ indexed by $A$; and
        \item $H$ is a real number satisfying $H \geq P^3$.
    \end{itemize}
\end{definition}

We refer to the elements $x\in A$ as the \emph{points} of the configuration, and the associated $\alpha_x \in \R/\Z$ as the \emph{frequencies}. The parameter $H$ in the definition of a configuration will be used in \cref{sec:localrelations} to quantify the strength of certain \emph{local relations} between the frequencies $\alpha_x$.

\begin{definition}[Sub-configuration]
    Let $\A = (A, \alpha_{\bullet}, H)$ be a configuration. For any non-empty subset $B\subset A$, we define the restriction
    \begin{equation*}
        \A|_{B} := (B, (\alpha_x)_{x\in B}, H).
    \end{equation*}
    Any configuration of this form is called a \emph{sub-configuration} of $\A$.
    \label{def:subconfig}
\end{definition}

A key part of the proof is to construct, given a configuration $\A$ with many local relations, a new configuration $\B$ that retains many of these relations, and is a `lift' of $\A$ in the following sense.
\begin{definition}[Lift]
    \label{def:lift}
    Let $\A = (A, \alpha_{\bullet}, H)$ and $\B = (B, \beta_{\bullet}, H')$ be configurations. For any $p\in \P$, we define
    \begin{equation*}
        p\B := (B, (p \beta_x)_{x\in B}, H'/p).
    \end{equation*}
    We say $\B$ is a \emph{lift} of $\A$ if there exists a prime $p\in \P$ such that $p \B$ is a sub-configuration of $\A$; that is, $B\subset A$, $H'=Hp$ and $p\beta_x = \alpha_x \pmod 1$ for all $x\in B$. We call $p$ the \emph{lifting prime}.
\end{definition}

The proof relies on constructing a sequence of such lifts, which we call a tower.
\begin{definition}[Tower]
    \label{def:tower}
    A \emph{tower of configurations} of \emph{height} $k$ is a sequence $(\A_i)_{0\leq i\leq k}$ of configurations such that $\A_i$ is a lift of $\A_{i-1}$ for all $1\leq i\leq k$.
\end{definition}

\subsection{Local relations}
\label{sec:localrelations}
Roughly speaking, two points $x,y\in A$ satisfy a local relation if there exist primes $p,q\in \P$ such that $px \approx qy$ and $q\alpha_x \approx p\alpha_y\pmod 1$. We measure the error terms in these approximations as follows.

\begin{definition}[Strength of a local relation]
    \label{def:locrelations}
    Let $\A = (A, \alpha_{\bullet}, H)$ be a configuration. For $x,y\in A$ and $p,q\in \P$, we define
    \begin{equation}
        \label{eq:defdist}
        \dist{p,q;\A}{x, y} := \max \Big(\frac{1}{P}\abs{px-qy},\,  H \norm{q\alpha_x - p\alpha_y}\Big).
    \end{equation}
    When $\A$ is clear from the context, we simply write $\dist{p,q}{x, y}$.
\end{definition}

We collect the quadruples $(x,y,p,q)$ satisfying strong local relations in the following set.

\begin{definition}[Strong local relations]
    \label{def:quadruples}
    We define
    \begin{equation}
        \label{eq:defQ}
        \Q(\A) := \big\{ (x,y,p,q) \in A^2\times \P^2\, :\, \dist{p,q;\A}{x,y} \leq \tfrac{1}{10}\big\}.
    \end{equation}
\end{definition}

Thus, a quadruple $(x,y,p,q)$ lies in $\Q(\A)$ if and only if $\abs{px - qy} \leq \tfrac{P}{10}$ and $\norm{q\alpha_x - p\alpha_y} \leq \tfrac{1}{10H}$.\footnote{The constant $1/10$ in \cref{def:quadruples} is not important, but this choice ensures that $y$ is uniquely determined by~$x,p,q$ whenever $(x,y,p,q)\in \Q(\A)$, since $A$ is $1$-separated.}

\subsection{Local structure theorem}
The central tool of our iterative argument is the following local structure theorem for configurations with many local relations.

\begin{restatable}[Local structure theorem]{theorem}{structhm}
    \label{thm:structure}
    There exists an absolute constant $\Clift\geq 1$ such that the following holds. Let~$Y, P,\P$~be as in \cref{def:globalparam}. Let $\A = (A, \alpha_{\bullet}, H)$ be a configuration such that
    \begin{equation*}
        \abs{\Q(\A)} \geq \delta |A| |\P|^2
    \end{equation*}
    for some $0 < \delta <\tfrac{1}{2}$. Let $L\in \N$ and $0<\eps <\tfrac12$ be parameters such that $\eps^{-\Clift L} \leq P$ and $\eps \delta^{-5} L \leq 1$.

    Then, there exists a decomposition $A = A_0 \sqcup \ldots \sqcup A_L \sqcup A_{L+1}$ and a lift $\B = (A\setminus A_{L+1}, \,\beta_{\bullet}, Hp^*)$ of~$\A$ with the following properties.
    \begin{enumerate}[label=(\Roman*),ref=\Roman*]
        \item \label{item:struct1} $|A_0| \gg \delta^{5}|A|$.
        \item \label{item:struct2} For all but $\ll \eps |A| |\P|^2$ quadruples $(x,y,p,q)\in \Q(\A)$,
              \begin{enumerate}
                  \item \label{item:struct2a} if $x\in A_\ell$ then $y\in A_{\ell+e}$ for some $e\in \{-1,0,1\}$, and
                  \item \label{item:struct2b} if $x,y\notin A_{L+1}$, then $(x,y,p,q)\in \Q(\B)$.
              \end{enumerate}
    \end{enumerate}
\end{restatable}

The local structure theorem provides a partition of $A$ into (possibly empty) sets $A_0, A_1, \ldots, A_L$, along with an exceptional set $A_{L+1}$. On the non-exceptional part $A\setminus A_{L+1}$, the frequencies can be lifted to form a configuration $\B$ that preserves almost all local relations not involving $A_{L+1}$. Moreover, this decomposition heavily restricts the combinatorial structure of $\A$: apart from a negligible proportion of exceptions, the local relations in $\A$ must either be internal to a single set $A_{\ell}$, or connect adjacent sets $A_{\ell}$ and $A_{\ell+1}$.

In order to reach the main consequences on Fourier uniformity more quickly, we defer the rather technical proof of \cref{thm:structure} to \cref{part:local} (\cref{sec:preliminarylemmas,sec:covering,sec:structure}).

\subsection{Global structure theorem}
The ultimate goal of our analysis is to establish a global formula for the frequencies of a configuration $\A$ with many local relations.

\begin{restatable}[Global structure theorem]{theorem}{globalthm}
    \label{thm:globalstructuretheorem}
    Let $\eps_0>0$ and let $Y$ be sufficiently large in terms of $\eps_0$. Let $P := \exp((\log Y)^{2/5+\eps})$ for some $\eps_0<\eps<\tfrac{1}{5}$.

    Let $\A = (A, \alpha_{\bullet}, H)$ be a configuration such that $|\Q(\A)| \geq \delta |A| |\P|^2$, where $\delta \geq (\log \log Y)^{-10}$.

    Let $\kappa := \exp((\log Y)^{2/5-\eps})$. Then, there is a subset $A'\subset A$ of size
    \begin{equation*}
        |A'| \geq  \kappa^{-1} |A|,
    \end{equation*}
    such that, for all $x\in A'$, we have the approximate formula
    \begin{equation*}
        \alpha_x = \frac{a_0}{q_0} + \frac{T_0}{x} + O\bigg(\frac{\kappa}{HP}\bigg) \pmod 1,
    \end{equation*}
    where $a_0, q_0$ are coprime integers with $1\leq q_0 \leq \kappa$ and $T_0\in \R$ satisfies $|T_0| \leq \kappa \frac{Y^2}{HP}$.
\end{restatable}

We prove \cref{thm:globalstructuretheorem} in \cref{part:global} (\cref{sec:tower,sec:largeH,sec:expansion}), using the local structure theorem (\cref{thm:structure}) as a black box. The argument proceeds in three main stages.

In \cref{sec:tower}, we iteratively apply \cref{thm:structure} to construct a tower of configurations above $\A$. If the exceptional set $A_{L+1}$ has negligible size, \cref{thm:structure} provides a lift $\B$ that is almost an exact copy of $\A$; unfortunately, this cannot be guaranteed in general. We circumvent this by exploiting two key features of \cref{thm:structure}. First, its quantitative strength allows for a logarithmically large chain length $L$, imposing more rigid constraints on the local relations than \cite[Proposition~4.1]{walsh3}. Second, its validity on arbitrarily sparse sets allows the iteration to proceed despite severe losses in the cardinality of $A$.

In \cref{sec:largeH}, we analyse the resulting construction. The lack of GRH-type expansion prevents us from adopting the strategy of \cite{walsh3}, where one fixes a single point $x_0$ in the top configuration and traces its `descendants' down the tower to deduce a global formula on a positive density subset of $A$. Instead, we work solely within the top configuration $\A'$. The local relations in $\A'$ have exceptionally small error terms on the frequency side (these errors having contracted by a factor $\asymp P$ at each step of the iteration), making them amenable to more direct graph-theoretic arguments.

Finally, in \cref{sec:expansion}, we combine the preceding analysis with an unconditional expansion estimate based on Vinogradov-Korobov type bounds for Dirichlet $L$-functions to establish the global structure theorem.

\subsection{Deduction of the Fourier uniformity estimate}

We now prove our main result on Fourier uniformity, \cref{thm:maingeneral}. The main ingredients are Elliott's inequality (via \cref{prop:setup}), our global structure theorem (\cref{thm:globalstructuretheorem}), and a variant of the Matomäki-Radziwi{\l\l} theorem with a power-saving upper bound for the number of exceptional intervals (\cref{thm:MR}).

\begin{proof}[Proof of \cref{thm:maingeneral}]
    The assumption in the statement implies the existence of an $H$-separated subset $S \subseteq [X, 2X)$ of size $\gg \delta X/H$, such that for each $x \in S$, there exists $\theta_x \in \R/\Z$ satisfying
    \begin{equation*}
        \bigg\lvert\!\sum_{x \leq n < x+H} f(n) e(n\theta_x)\bigg\rvert \gg \delta H.
    \end{equation*}

    By \cref{prop:setup}, there exists a scale $P$ with $H^{e^{-O(\delta^{-2})}} \leq P \leq H^{1/10}$ such that, writing $\P$ for the set of primes in $[P, 2P)$, there are
    \begin{equation*}
        \gg \delta^{7} |S| |\P|^2
    \end{equation*}
    quadruples $(x, y, p, q) \in S^2 \times \P^2$ satisfying $|px-qy|\leq \tfrac{1}{10} PH$ and $\|q \theta_x - p \theta_y\| \leq \tfrac{P}{10H}$. Defining $Y := X/H$, $A := \frac{1}{H} S\subset [Y, 2Y)$ and $\alpha_x := \theta_{Hx}$ for $x\in A$, we see that $\A := (A, \alpha_{\bullet}, H/P)$ is a configuration with
    \begin{equation*}
        \abs{\Q(\A)} \gg \delta^{7} |A| |\P|^2.
    \end{equation*}

    Let $\kappa := \exp((\log Y)^{2/5-\eps/2})$. By our global structure theorem, \cref{thm:globalstructuretheorem}, there exists a subset $A' \subseteq A$ of size $|A'| \geq \kappa^{-1} |A|$ such that, for each $x \in A'$, we have the approximate formula
    \begin{equation*}
        \alpha_x = \frac{a_0}{q_0} + \frac{T_0}{x} + O\bigg(\frac{\kappa}{H}\bigg) \pmod 1
    \end{equation*}
    for some $a_0, q_0, T_0$ independent of $x$, where $a_0$ and $q_0$ are coprime integers with $1 \leq q_0 \leq \kappa$ and $T_0$ is a real number with $|T_0| \leq \kappa Y^2/H$.

    To conclude, we apply \cref{thm:MR} to the dilated set $S' := H\cdot A'\subset [X, 2X)$, the frequencies $(\theta_x)_{x\in S'}$, and $T := HT_0$. The required bound \cref{eq:paramassumpt} simplifies to
    \begin{equation*}
        \exp((\log X)^{2/5-\eps/2})^{\delta^{-C}} \leq H \leq X^{\delta^{C}}
    \end{equation*}
    for some absolute constant $C>0$, which holds by our choices of $\delta$ and $H$ provided that $X$ is sufficiently large in terms of $\eps$. This yields the desired pretentiousness estimate for $f$.
\end{proof}
\part{Global structure theorem}
\label{part:global}

\section{Construction of a tower of configurations}
\label{sec:tower}
In this section, we use our local structure theorem, \cref{thm:structure}, to construct a tower of configurations above a given configuration $\A$ with many local relations, ensuring that each intermediate lift maintains a large relative density of local relations. We show that this construction succeeds unless $\A$ contains a very sparse, highly connected sub-configuration, a scenario we rule out in \cref{sec:expansion}.

To accomplish this, we exploit the full strength of \cref{thm:structure} to decompose $\A$ into $L$ parts, where $L$ can be chosen to be almost as large as $\log P$. We will use the following simple lemma to analyse the resulting partition.

\begin{lemma}
    \label{lem:realanalysis}
    Let $L\in \N$ and $1\leq C \leq L$. Let $a_1, \ldots, a_{L}\geq 0$ be such that, for all $i\in \{1, \ldots, L\}$,
    \begin{equation*}
        a_i \geq \frac{C}{L} \sum_{j>i} a_j.
    \end{equation*}
    Then,
    \begin{equation*}
        \min_{1\leq i\leq L} a_i \ll \frac{e^{-C/3}}{L} \sum_{i=1}^{L} a_i.
    \end{equation*}
\end{lemma}

\begin{proof}
    We first consider a continuous analogue of the problem. Let $f:[0,1]\to \R_{\geq 0}$ be a non-negative integrable function such that, for some $\gamma\geq 1$ and all $x\in [0,1]$,
    \begin{equation}
        \label{eq:realanalysiscondcont}
        f(x) \geq \gamma \int_{x}^{1} f(y) \, dy.
    \end{equation}
    We claim that
    \begin{equation}
        \label{eq:realanalysiscont}
        \min_{x\in [0,1]} f(x) \leq \gamma e^{1-\gamma} \int_0^1 f(y) \, dy.
    \end{equation}
    To see this, let $F(x) := \int_x^1 f(y) \, dy$. Then, $F$ is absolutely continuous and $F'(x) = -f(x)$ almost everywhere. The assumption \cref{eq:realanalysiscondcont} translates to $F'+\gamma F \leq 0$, i.e.~$(e^{\gamma x}F(x))' \leq 0$, almost everywhere. Therefore, the function $e^{\gamma x}F(x)$ (which is absolutely continuous) is non-increasing, so that
    \begin{equation*}
        (1-x) \min_{y\in [0,1]} f(y) \leq F(x) \leq e^{-\gamma x} F(0)
    \end{equation*}
    for all $x\in [0,1]$. Setting $x = 1 - \frac{1}{\gamma}$, we obtain the claim \cref{eq:realanalysiscont}.

    To deduce the lemma from this continuous version, we define a piecewise constant function $f$ on~$[0, 1]$ by setting $f(x) := a_i$ for $x\in \big(\frac{i-1}{L}, \frac{i}{L}\big]$, for $1\leq i\leq L$. Then, for all $x\in \big(\frac{i-1}{L}, \frac{i}{L}\big]$, we have
    \begin{equation*}
        \int_{x}^{1} f(y) \, dy \leq \frac{1}{L} \sum_{j\geq i} a_j \leq \left(\frac{1}{L}+\frac{1}{C}\right)a_i \leq \frac{2}{C} f(x),
    \end{equation*}
    so that $f$ satisfies \cref{eq:realanalysiscondcont} with $\gamma := \max(\tfrac{C}{2}, 1)$. Applying \cref{eq:realanalysiscont} concludes the proof.
\end{proof}

We can now prove a dichotomy: if $\A$ has many local relations, then either it admits a suitable lift, or it contains a small sub-configuration whose relative density of local relations only decreases by a factor of the shape $1-\frac{1}{\log P}$.

\begin{lemma}
    \label{lem:liftorsplit}
    Let $\A = (A, \alpha_{\bullet}, H)$ be a configuration such that $|\Q(\A)| \geq \delta |A| |\P|^2$, where ${\delta \geq 1/\log Y}$. Then, one of the following holds:
    \begin{enumerate}[label=(\roman*),ref=\roman*]
        \item \label{item:yeslift} There exists a lift $\B = (B, \beta_{\bullet}, Hp^*)$ of $\A$ such that
              \begin{equation*}
                  |\Q(\B)| \geq \bigg(1 - O\bigg( \frac{1}{(\log Y)^{10}} \bigg) \bigg) \delta |B| |\P|^2.
              \end{equation*}
        \item \label{item:yessplit} There is a subset $B\subset A$ with $|B| \leq \tfrac12 |A|$, such that
              \begin{equation*}
                  |\Q(\A|_{B})| \geq \bigg(1 - O\bigg(\frac{\delta^{-6}(\log \log Y)^2}{\log P}\bigg) \bigg) \delta |B||\P|^2.
              \end{equation*}
    \end{enumerate}
\end{lemma}

\begin{proof}
    Without loss of generality, we may assume that every proper subset $A' \subset A$ satisfies
    \begin{equation}
        \label{eq:maxasspt}
        |\Q(\A|_{A'})| \leq \delta |A'| |\P|^2,
    \end{equation}
    or else we can start over with $\A|_{A'}$ in place of $\A$.

    We apply \cref{thm:structure} with $\eps := (\log Y)^{-20}$ and $L := \lfloor (\log P) / (20\Clift\log \log Y) \rfloor$. Note that the conditions on $\eps$ and $L$ in \cref{thm:structure} are satisfied if $Y$ is sufficiently large and $L\geq 1$, which we may assume or else the lemma holds for trivial reasons.

    Write $A = A_0 \sqcup A_1 \sqcup \ldots \sqcup A_{L+1}$ and ${\B = (A\setminus A_{L+1}, \beta_{\bullet}, Hp^*)}$ for the decomposition and the lift given by \cref{thm:structure}.

    We claim that one of the following holds:
    \begin{enumerate}[label=(\alph*),ref=\alph*]
        \item \label{item:ALplusone} $|A_{L+1}| \ll |A|/(\log Y)^{15}$,
        \item \label{item:ALatmost} $|A_{L+1}| \geq |A|/(\log Y)^{15}$ and there exists $1\leq i \leq L$ such that
              \begin{equation*}
                  |A_i| \ll \frac{\log \log Y}{L} \sum_{j=i+1}^{L+1} |A_j|.
              \end{equation*}
    \end{enumerate}
    Indeed, suppose that $|A_{L+1}| \geq |A|/(\log Y)^{15}$ and that $|A_i| \geq \frac{60 \log \log Y}{L} \sum_{j=i+1}^{L+1} |A_j|$ for all $1\leq i\leq L$. Then, by \cref{lem:realanalysis}, we have $\min_{1\leq i\leq L+1} |A_i| \ll |A|/(\log Y)^{20}$. If this minimum is attained at $i=L+1$, we are in case \cref{item:ALplusone}; otherwise, we are in case \cref{item:ALatmost}.

    We now treat the two cases separately.

    Suppose first that case \cref{item:ALplusone} holds. By part \cref{item:struct2b} of \cref{thm:structure}, we have
    \begin{equation*}
        |\Q(\B)| \geq \delta |A| |\P|^2 - O\! \left(|A_{L+1}||\P|^2 + \eps |A||\P|^2\right).
    \end{equation*}
    Recalling that $B = A \setminus A_{L+1}$ and inserting our bound on $|A_{L+1}|$, we obtain
    \begin{equation*}
        |\Q(\B)| \geq \delta |A| |\P|^2 - O\! \left(\frac{|A||\P|^2}{(\log Y)^{15}}\right),
    \end{equation*}
    which yields conclusion \cref{item:yeslift} of the lemma.

    Now, suppose that case \cref{item:ALatmost} holds. Define $B_1 := \bigsqcup_{j=0}^{i-1} A_j$ and $B_2 := \bigsqcup_{j=i+1}^{L+1} A_j$. By property \cref{item:struct2a} of \cref{thm:structure}, we have
    \begin{equation*}
        |\Q(\A)| \leq |\Q(\A|_{B_1})| + |\Q(\A|_{B_2})| + O\! \left(|A_i||\P|^2 + \eps |A||\P|^2\right).
    \end{equation*}
    To simplify the error term, we note that $|A_i| \ll \frac{\log \log Y}{L} |B_2|$ and $|A| \ll (\log Y)^{15}|B_2|$ by assumption. Thus,
    \begin{equation}
        \label{eq:bound1QBi}
        |\Q(\A)| \leq |\Q(\A|_{B_1})| + |\Q(\A|_{B_2})| + O\! \left(\frac{\log \log Y}{L}|B_2||\P|^2\right).
    \end{equation}

    Since
    \begin{equation*}
        |\Q(\A)| \geq \delta |A||\P|^2 \geq \delta |B_1||\P|^2 + \delta |B_2| |\P|^2
    \end{equation*}
    and $|\Q(\A|_{B_k})| \leq \delta |B_k| |\P|^2$ for $k=1,2$ by \cref{eq:maxasspt}, the estimate \cref{eq:bound1QBi} implies that
    \begin{equation}
        \label{eq:bound2QBi}
        |\Q(\A|_{B_k})| \geq \delta |B_k| |\P|^2 - O\! \left(\frac{\log \log Y}{L}|B_2||\P|^2\right)
    \end{equation}
    for $k=1,2$. Moreover, by part \cref{item:struct1} of \cref{thm:structure}, we have
    \begin{equation*}
        |B_2| \leq |A| \ll \delta^{-5} |A_0| \leq \delta^{-5}|B_1|,
    \end{equation*}
    so that \cref{eq:bound2QBi} also holds with $\delta^{-5}|B_1|$ in place of $|B_2|$ in the error term. Hence, choosing $B$ to be the smaller of the two sets $B_1$ and $B_2$, we obtain conclusion \cref{item:yessplit} of the lemma.
\end{proof}

Iterating \cref{lem:liftorsplit}, we either obtain the desired tower of configurations, or we can isolate a very sparse sub-configuration with a large relative density of local relations. Crucially, in the latter case, this density parameter only decreases by a constant factor, while the size of the underlying set drops by almost a power of $P$.

\begin{proposition}
    \label{prop:tower}
    Let $\A = (A, \alpha_{\bullet}, H)$ be a configuration with $|\Q(\A)| = \delta |A| |\P|^2$, where $\delta \geq 2/\log Y$.
    Then, one of the following holds:
    \begin{enumerate}
        \item \label{item:towercase} There exists a tower of configurations $(\A_i)_{0\leq i\leq k}$ of height $k\gg (\log Y)^{10}$, such that $\A_0 = \A$ and ${|\Q(\A_k)| \gg \delta |A_k||\P|^2}$ (where~$A_k$ is the set of points of the configuration $\A_k$).
        \item \label{item:smallcase} There exists $B \subset A$ such that ${|\Q(\A|_{B})| \gg \delta |B||\P|^2}$ and $|B| \ll P^{-c_1\delta^6/ (\log \log Y)^2} |A|$, where $c_1>0$ is an absolute constant.
    \end{enumerate}
\end{proposition}

\begin{proof}
    We first dispose of the trivial case where $\delta^6 \log P \leq C_0 (\log \log Y)^2$ for a sufficiently large absolute constant $C_0>0$. In this regime, conclusion \cref{item:smallcase} is trivially satisfied for $B=A$. Hence, we may assume that $\delta^6 \log P > C_0 (\log \log Y)^2$.

    By induction, we construct a sequence of triples $(\B_i, B_i, \delta_i)$ where $\B_i$ is a configuration with point set $B_i$ satisfying $|\Q(\B_i)| = \delta_i |B_i| |\P|^2$. This sequence is initialised at $(\B_0, B_0, \delta_0) := (\A, A, \delta)$.

    Suppose that $\B_0, \ldots, \B_{i-1}$ have been constructed. If $\delta_{i-1} < \delta/2$, we stop the procedure. Otherwise, we apply \cref{lem:liftorsplit} to $\B_{i-1}$:
    \begin{itemize}
        \item \textit{Lifting step.} If case \cref{item:yeslift} of \cref{lem:liftorsplit} holds, define $\B_{i}$ to be the lift of~$\B_{i-1}$ produced by the lemma; the corresponding parameter $\delta_i$ satisfies
              \begin{equation*}
                  \delta_{i} \geq \bigg(1 - O\bigg( \frac{1}{(\log Y)^{10}} \bigg) \bigg)  \delta_{i-1}.
              \end{equation*}
        \item \textit{Halving step.} If, instead, case \cref{item:yessplit} of \cref{lem:liftorsplit} holds, there is a subset $B_{i} \subset B_{i-1}$ with $|B_{i}| \leq \tfrac12 |B_{i-1}|$ such that, defining $\B_i := \B_{i-1}|_{B_i}$, we have
              \begin{equation*}
                  \delta_{i} \geq \bigg(1 - O\bigg(\frac{\delta^{-6}(\log \log Y)^2}{\log P}\bigg) \bigg) \delta_{i-1}.
              \end{equation*}
    \end{itemize}

    Our initial assumption ensures that $\delta_i \geq \frac{99}{100} \delta_{i-1}$ in either case. For any $k,m\in \N$, if the first $k+m$ steps of the procedure consist of $k$ lifting steps and $m$ halving steps (in any order), we have
    \begin{equation*}
        \delta_{k+m} \geq \exp\!\left(-O\!\left(k \cdot \frac{1}{(\log Y)^{10}} + m \cdot \frac{\delta^{-6}(\log \log Y)^2}{\log P}\right)\right) \delta.
    \end{equation*}
    Hence, the stopping condition $\delta_{i}<\delta/2$ is never triggered before either $\gg (\log Y)^{10}$ lifting steps or $\gg {\delta^{6} \log P}/{(\log \log Y)^2}$ halving steps have been performed.

    If there are $k \gg (\log Y)^{10}$ lifting steps, defining $\A_0 := \A$ and letting $\A_{i}$ be the configuration obtained after the $i$-th lifting step, we obtain a tower of configurations as required for conclusion \cref{item:towercase} of the proposition.

    Otherwise, there are $m\gg {\delta^{6} \log P}/{(\log \log Y)^2}$ halving steps. In this case, the configuration $\B = (B, \beta_{\bullet}, H')$ obtained after these $m$ halving steps satisfies $|B| \leq 2^{-m} |A|$ and $|\Q(\B)| \geq \frac{\delta}{2} |B||\P|^2$. Since $\Q(\B) \subset \Q(\A|_{B})$ by definition of a lift, we obtain conclusion \cref{item:smallcase} of the proposition.
\end{proof} \section{Analysis of the top configuration}
\label{sec:largeH}

Having constructed high towers of configurations in the previous section, we now study the configurations that lie at the top of these towers. For such a configuration $\A = (A, \alpha_{\bullet}, H)$, the parameter $H$ is exceptionally large, meaning that the local relations in $\A$ hold with much smaller error terms on the frequency side. This rigidity allows us to employ graph-theoretic exploration arguments to extract an initial approximate formula for the frequencies $\alpha_x$ on a highly connected subset of $A$.

\subsection{Paths in the graph of local relations}

\begin{definition}[Graph of a configuration]
    \label{def:graph}
    Let $\A = (A, \alpha_{\bullet}, H)$ be a configuration. We define $G(\A)$ to be the graph with vertex set $A$, and with an edge between two distinct points $x,y\in A$ whenever $(x,y,p,q)\in \Q(\A)$ for some $p,q\in \P$.
\end{definition}

\begin{remark}
    \label{rem:graphobservations}
    Given a vertex $x\in A$ and $p,q\in \P$, there is at most one neighbour $y$ of $x$ such that $(x,y,p,q)\in \Q(\A)$, by \cref{def:locrelations} and the $1$-separation of $A$.

    Similarly, if two distinct points $x,y\in A$ are neighbours in $G(\A)$, then the pair $(p,q)\in \P^2$ with $(x,y,p,q)\in \Q(\A)$ is unique. Indeed, if $(x,y,p,q),(x,y,p',q')\in \Q(\A)$, then $\big|\frac{p}{q} - \frac{p'}{q'}\big| \leq \frac{1}{5Y}$, and thus $(p,q)=(p',q')$, as the left-hand side is either zero or at least $\frac{1}{qq'}$.
\end{remark}

To analyse the graph $G(\A)$, we introduce the following terminology, which will only be used in the present section.

\begin{definition}[Prime sequences and prime products of a path]
    \label{def:paths}
    A \emph{path of length $k$} in $G(\A)$ is a sequence $x_{\bullet} = (x_i)_{i=0}^k$ of vertices of $G(\A)$ such that $\{x_{i-1}, x_i\}$ is an edge of $G(\A)$ for all $1\leq i\leq k$.

    The \emph{prime sequences} of the path $x_{\bullet}$ are the unique sequences $(p_i)_{i=1}^k, (q_i)_{i=1}^k$ of primes in $\P$ such that ${(x_{i-1}, x_i, p_i, q_i)\in \Q(\A)}$ for all $1\leq i\leq k$. The \emph{prime products} of the path $x_{\bullet}$ are defined to be the numbers $\prod_{i=1}^k p_i$ and $\prod_{i=1}^k q_i$.
\end{definition}

We have the following simple estimate relating the endpoints of a path and the corresponding frequencies, which is similar to \cite[Lemma~6.3]{walsh3}.

\begin{lemma}
    \label{lem:paths}
    Let $\A = (A, \alpha_{\bullet}, H)$ be a configuration. Let $x_{\bullet} = (x_i)_{i=0}^k$ be a path of length $k\leq \log Y$ in $G(\A)$, with prime products $R$ and $S$. Then
    \begin{equation}
        \label{eq:pathabsval}
        \abs{\frac{R}{S} x_0 - x_k} \leq k
    \end{equation}
    and
    \begin{equation}
        \label{eq:pathnorm}
        \norm{S \alpha_{x_0} - R \alpha_{x_k}} \leq \frac{kR}{HP}.
    \end{equation}
\end{lemma}

\begin{proof}
    We introduce the notation $R_{i\to j} := \prod_{i<\ell \leq j} p_{\ell}$ and $S_{i\to j} := \prod_{i<\ell \leq j} q_{\ell}$ (and use the convention $R_{i\to i}=S_{i\to i}= 1$). For each $1\leq \ell \leq k$, we have $(x_{\ell-1}, x_{\ell}, p_{\ell}, q_{\ell}) \in \Q(\A)$, so that
    \begin{equation*}
        \abs{x_{\ell} - \frac{p_{\ell}}{q_{\ell}} x_{\ell-1}} \leq \frac{1}{10}.
    \end{equation*}
    By the triangle inequality, this implies that, for all $i<j$,
    \begin{equation}
        \label{eq:intermediatepath}
        \abs{x_j - \frac{R_{i\to j}}{S_{i\to j}}x_i} \leq \frac{1}{10} \sum_{i<\ell\leq j} \frac{R_{\ell\to j}}{S_{\ell \to j}}.
    \end{equation}
    In particular, using that $\P \subset [P, 2P]$, we obtain the crude bound
    \begin{equation*}
        \abs{x_j - \frac{R_{i\to j}}{S_{i\to j}} x_i} \leq \frac{1}{10} \sum_{i<\ell\leq j} 2^{j-\ell} \leq \frac{2^{j-i}}{10}.
    \end{equation*}
    Since $x_i, x_j\in [Y, 2Y]$ and $k\leq \log Y$, this implies that
    \begin{equation}
        \label{eq:useflbdprimes}
        \frac{R_{i\to j}}{S_{i\to j}} \in \left[ \frac{1}{2} -  \frac{2^{j-i}}{10Y} ,\, 2 + \frac{2^{j-i}}{10Y} \right] \subset \left[\frac13, \, 3\right]
    \end{equation}
    for all $i<j$. Plugging this estimate back into \cref{eq:intermediatepath} yields
    \begin{equation*}
        \abs{x_k -  \frac{R_{0\to k}}{S_{0\to k}} x_0} \leq \frac{3k}{10},
    \end{equation*}
    which proves \cref{eq:pathabsval}.

    The proof of \cref{eq:pathnorm} is very similar. For every $1\leq j\leq k$, we have
    \begin{equation*}
        \norm{R_{0\to j}S_{j\to k} \alpha_{x_j} - R_{0\to j-1}S_{j-1\to k} \alpha_{x_{j-1}}} \leq  R_{0\to j-1}S_{j\to k} \norm{p_j \alpha_{x_j} - q_j \alpha_{x_{j-1}}} \leq \frac{R_{0\to j}S_{j\to k}}{10HP}.
    \end{equation*}
    By the triangle inequality, we obtain
    \begin{equation*}
        \norm{R_{0\to k}\alpha_{x_k} - S_{0\to k} \alpha_{x_{0}}} \leq \frac{1}{10HP} \sum_{j=1}^k R_{0\to j}S_{j\to k}.
    \end{equation*}
    Using the bound $S_{j\to k} \leq 3 R_{j\to k}$ given in \cref{eq:useflbdprimes}, the second conclusion \cref{eq:pathnorm} follows.
\end{proof}

\begin{lemma}
    \label{lem:nonequivpaths}
    Let $\A = (A, \alpha_{\bullet}, H)$ be a configuration whose graph $G(\A)$ has minimum degree at least~$\delta |\P|^2$. Suppose that $|\P|\geq (\log Y)^4$ and $\delta \geq 4/\log |\P|$.

    Then, for any $x\in A$, there exists a path of length
    \begin{equation*}
        \frac{\log Y}{\log |\P|} + O\!\left( \frac{(\log Y) \log \log Y}{(\log |\P|)^2} + 1\right)
    \end{equation*}
    starting and ending at $x$, whose prime products $R$ and $S$ are distinct.
\end{lemma}

\begin{proof}
    Call a path $(x_i)_{i=0}^k$ in $G(\A)$ \emph{typical} if its prime products $R,S$ are coprime.

    Let $x\in A$. By induction, the number of typical paths of length $k$ with starting vertex $x$ is at least $(\delta |\P|^2/2)^k$, for all $1\leq k\leq \delta|\P|/4$. Indeed, a typical path $(x_i)_{i=0}^{\ell}$ of length $\ell < k$ starting at $x$, with prime sequences $(p_i)_{i=1}^{\ell}$ and $(q_i)_{i=1}^{\ell}$, can be extended to a typical path of length $\ell+1$ by choosing an oriented edge $(x_{\ell},y)$ whose corresponding primes $p,q\in \P$ satisfy $p\notin \{q_1, \ldots, q_{\ell}\}$ and $q\notin \{p_1, \ldots, p_{\ell}\}$, and there are at least
    \begin{equation*}
        \geq \delta |\P|^2 - 2\ell |\P| \geq \tfrac{1}{2}\delta |\P|^2
    \end{equation*}
    ways to do so.

    A path of length $k$ starting at $x$ is uniquely determined by its prime sequences, and each prime sequence is a permutation of the $k$ prime factors of the corresponding prime product. Hence, given any pair $(R,S)$ of integers, there are $\leq (k!)^2$ paths of length $k$ starting at $x$ with prime products $R$ and $S$.

    Therefore, for any positive integer $k$ such that
    \begin{equation}
        \label{eq:conditionsfork}
        k \leq \delta |\P|/4 \quad\text{and}\quad\bigg(\frac{\delta |\P|^2}{2} \bigg)^k > (k!)^2 |A|,
    \end{equation}
    there exists some $y\in A$ and two typical paths $(x_i)_{i=0}^k$ and $(y_i)_{i=0}^k$ from $x$ to $y$ not having the same pair of prime products. Let $R_1,S_1$ and $R_2,S_2$ be the prime products of $(x_i)_{i=0}^k$ and $(y_i)_{i=0}^k$, respectively. Since these two paths are typical, we have $\gcd(R_1,S_1) = \gcd(R_2, S_2) = 1$, and thus
    \begin{equation}
        \label{eq:distinctRiSi}
        \frac{R_1}{S_1} \neq \frac{R_2}{S_2}.
    \end{equation}
    Hence, the path obtained by concatenating $(x_i)_{i=0}^k$ with the reverse of $(y_i)_{i=0}^k$ is a path of length $2k$ starting and ending at $x$. The prime products of this closed path are $R_1S_2$ and $R_2S_1$, which are distinct by \cref{eq:distinctRiSi}.

    It only remains to show that the size conditions \cref{eq:conditionsfork} are verified for some $k\in \N$ of the form
    \begin{equation}
        \label{eq:choiceofk}
        k = \frac{\log Y}{2\log |\P|} + O\!\left(\frac{ (\log Y) \log \log Y}{(\log |\P|)^2}+1\right).
    \end{equation}

    By the intermediate value theorem, there exists a real number $t\in \big[\frac{\log Y}{2\log |\P|}, \frac{\log Y}{\log |\P|} \big]$ such that
    \begin{equation*}
        t - \frac{\log Y}{2\log |\P|} - \frac{4t \log(4\delta^{-1}t)}{2\log |\P|} = 0.
    \end{equation*}
    Indeed, the left-hand side is clearly negative at $t = \frac{\log Y}{2\log |\P|}$, and the non-negativity at $t = \frac{\log Y}{\log |\P|}$ is equivalent to
    \begin{equation*}
        \log |\P| \geq 4 \log (4 \delta^{-1} \log Y/\log |\P|),
    \end{equation*}
    which holds since $\delta \geq 4/\log |\P|$ and $|\P|\geq (\log Y)^4$.

    Define $k = \lceil t \rceil$. Recalling that $Y \geq P^3$ (see \cref{def:globalparam}), we have $t\geq 3/2$, so that $k < 2t$. The bound $t\leq \log Y/\log |\P|$ and the inequalities $\delta \geq 4/\log |\P|$ and $|\P|\geq (\log Y)^4$ readily imply \cref{eq:choiceofk} and the first condition in \cref{eq:conditionsfork}. For the second condition in \cref{eq:conditionsfork}, we observe that
    \begin{equation*}
        2^k \delta^{-k} (k!)^2 \leq (2\delta^{-1}k)^{2k} < (4\delta^{-1}t)^{4t} = \frac{|\P|^{2t}}{Y} \leq \frac{|\P|^{2k}}{|A|}.
    \end{equation*}
    Hence, this choice of $k$ satisfies all the required conditions.
\end{proof}

\subsection{Extraction of a global formula}
We now consider a configuration $\A = (A, \alpha_{\bullet}, H)$ with many local relations and a very large parameter $H$. Using the preceding lemmas, we show that many frequencies $\alpha_x$ can be approximated by rationals $a_x/q_0$ with a common denominator $q_0$. Our control over the size of $q_0$ is inevitably quite weak at this stage; we will improve it in \cref{sec:expansion} using a `modular' expansion estimate.

\begin{lemma}
    \label{lem:graphimproving}
    Let $d,n,\delta>0$. Let $G$ be a graph on $n$ vertices, with maximum degree $d$ and $\geq \delta d n$ edges. Then, there exists a subgraph $G'$ of $G$ with minimum degree $\geq \delta d/2$ and $\geq \delta dn/2$ edges.
\end{lemma}
\begin{proof}
    Let $G'$ be the graph obtained by iteratively removing all vertices of degree $< \delta d/2$. This procedure removes at most $\delta dn/2$ edges in total. Hence, $G'$ has the desired properties.
\end{proof}

\begin{lemma}
    \label{lem:getfirstformula}
    Let $\A = (A, \alpha_{\bullet}, H)$ be a configuration such that $|\Q(\A)| \geq \delta |A| |\P|^2$, where ${H \geq Y^{C_1}}$ and $\delta \geq C_1/\log |\P|$ for some sufficiently large absolute constant $C_1>0$. Assume that $|\P| \geq (\log Y)^4$.

    Then, there exists a subset $A_2 \subset A$ and an integer
    \begin{equation*}
        1\leq q_0 \leq \exp\!\left( O\!\left(\frac{(\log Y)\log \log Y}{\log |\P|} + \log |\P|\right) \right)
    \end{equation*}
    with the following properties.
    \begin{enumerate}
        \item \label{item:getform1} For each $x\in A_2$, there are $a_x\in \Z$ and $\beta_x \in \R$ such that $(a_x, q_0)=1$, $|\beta_x| \leq Y^{O(1)}/H$ and
              \begin{equation*}
                  \alpha_x = \frac{a_x}{q_0} + \beta_x \pmod 1.
              \end{equation*}
        \item \label{item:getform2} For each $x\in A_2$, there are $\gg \delta|\P|^2$ quadruples $(x, y, p, q)\in \Q(\A|_{A_2})$ such that ${(pq, q_0)=1}$, $p\neq q$,
              \begin{equation*}
                  q a_x \equiv p a_y \pmod {q_0} \quad \text{and} \quad |q \beta_x - p\beta_y| \leq \frac{1}{10H}.
              \end{equation*}
    \end{enumerate}
\end{lemma}

\begin{proof}
    Since $|\Q(\A)| \geq \delta |A| |\P|^2$, the graph $G(\A)$ has $\geq \tfrac12 (\delta Y |\P|^2 - Y|\P|) \geq \tfrac{1}{3} \delta Y |\P|^2$ edges (as every edge corresponds to exactly two quadruples $(x,y,p,q)\in \Q(\A)$ with $p\neq q$). By \cref{lem:graphimproving}, there is a subset $A_1\subset A$ such that $G(\A|_{A_1})$ has minimum degree $\geq \tfrac{1}{6}\delta|\P|^2$.

    Let $x\in A_1$. Let $\kappa$ denote the quantity $\kappa := {(\log Y)\log \log Y}/{\log |\P|}$. By \cref{lem:nonequivpaths}, there exists a path of length $\ell$ with
    \begin{equation}
        \label{eq:thelengthell}
        \ell = \frac{\log Y}{\log |\P|} + O\!\left( \frac{\kappa}{\log |\P|} +1 \right),
    \end{equation}
    starting and ending at $x$, whose prime products $R_x$ and $S_x$ are distinct. Applying \cref{lem:paths} to this path, we get
    \begin{equation*}
        \abs{\frac{R_x}{S_x} x - x} \leq \ell \quad \text{and} \quad \norm{S_x \alpha_{x} - R_x \alpha_{x}} \leq \frac{\ell R_x}{HP}.
    \end{equation*}
    Since $R_x, S_x \leq (2P)^{\ell}$, we may rewrite these bounds as
    \begin{equation*}
        \abs{R_x - S_x} \leq \frac{\ell 2^{\ell} P^{\ell}}{Y} \quad \text{and} \quad \norm{(R_x - S_x) \alpha_{x}} \leq \frac{\ell 2^{\ell} P^{\ell-1}}{H}.
    \end{equation*}
    The last inequality implies that
    \begin{equation}
        \label{eq:qxversion}
        \alpha_x = \frac{a_x}{q_x} + \beta_x \pmod 1
    \end{equation}
    for some reduced fraction $a_x/q_x$ with denominator $q_x\mid R_x-S_x$ (with $a_x,q_x\in \N$), and some $\beta_x\in \R$ satisfying $|\beta_x| \leq \ell 2^{\ell} P^{\ell-1}/H$. In particular, $q_x \leq |R_x-S_x| \leq \ell 2^{\ell} P^{\ell}/Y$. By Chebyshev's estimate and our estimate \cref{eq:thelengthell} for $\ell$, we have
    \begin{equation*}
        \frac{P^{\ell}}{Y} \leq \frac{|\P|^{\ell}}{Y} \cdot (O(\log Y))^{\ell} \leq e^{O(\kappa+\log |\P|)}.
    \end{equation*}
    Thus, $q_x \leq e^{O(\kappa + \log |\P|)}$ and $|\beta_x| \leq Y^{O(1)}/H$.

    The next step is to pass to a suitable subset $A_2\subset A_1$ on which all $q_x$ are equal. Let
    \begin{equation*}
        q_0 := \max \{q_x : x\in A_1\}.
    \end{equation*}

    Let $G_0$ be the subgraph of $G(\A|_{A_1})$ obtained by discarding all quadruples $(x,y,p,q) \in \Q(\A|_{A_1})$ with $(pq, q_0)>1$. Recall that $G(\A|_{A_1})$ has minimum degree $\geq \tfrac{1}{6}\delta|\P|^2$. Therefore, the subgraph $G_0$ has minimum degree
    \begin{equation*}
        \geq \tfrac{1}{6}\delta|\P|^2 - 2|\P|\sum_{p\in \P}\ind{p\mid q_0} \geq \tfrac{1}{6}\delta|\P|^2 - 2|\P| {\log q_0} \gg \delta |\P|^2,
    \end{equation*}
    where we used that $\log q_0\ll \kappa + \log |\P| \ll \log Y$ and the bound $\delta \geq C_1/\log |\P| \geq C_1 \log Y/|\P|$ for~$C_1$ a sufficiently large constant.

    Let $x_0\in A_1$ be any point such that $q_{x_0} = q_0$. Let $A_2$ be the connected component of $x_0$ in~$G_0$.

    Suppose that $x,y\in A_2$ are connected by an edge in $G_0$, meaning that $(x,y,p,q)\in \Q(\A|_{A_2})$ for some $p,q\in \P$ not dividing $q_0$. Then $\norm{q\alpha_x-p\alpha_y} \leq (10H)^{-1}$, which by \cref{eq:qxversion} implies that
    \begin{equation}
        \label{eq:relationaftersubstitution}
        \norm{\frac{qq_ya_x-pq_xa_y}{q_x q_y} + q\beta_x - p\beta_y} \leq \frac{1}{10H}.
    \end{equation}
    Since $|q\beta_x - p\beta_y| \leq 2P(|\beta_x|+|\beta_y|) \leq Y^{O(1)}/H$, we get
    \begin{equation*}
        \norm{\frac{qq_ya_x-pq_xa_y}{q_x q_y}} \leq \frac{Y^{O(1)}}{H}.
    \end{equation*}
    The left-hand side is either zero or a fraction with denominator at most $q_xq_y \leq Y^{O(1)}$. The latter case is impossible by our assumption that $H \geq Y^{C_1}$ for some large enough constant $C_1$. Therefore, it must be the case that
    \begin{equation}
        \label{eq:qqy}
        qq_ya_x \equiv pq_x a_y \pmod{q_xq_y}.
    \end{equation}
    Plugging this information back into \cref{eq:relationaftersubstitution} yields $\norm{q\beta_x - p\beta_y} \leq (10H)^{-1}$. This modulo $1$ bound can be upgraded to $|q\beta_x - p\beta_y| \leq (10H)^{-1}$, using that ${|q\beta_x - p\beta_y|\leq Y^{O(1)}/H \leq 1/2}$ (provided $C_1$ is sufficiently large).

    Now, suppose that $x,y\in A_2$ are neighbours in $G_0$ such that $q_x = q_0$. By \cref{eq:qqy}, we have
    \begin{equation*}
        q_0 \mid qq_y a_x.
    \end{equation*}
    Since $(a_x, q_x)=1$ (by definition of $a_x,q_x$) and $(q_0, q)=1$ (by definition of $G_0$), we deduce that $q_0 \mid q_y$. By maximality of $q_0$, we conclude that $q_y = q_0$. Thus, the property of having $q_x = q_0$ propagates through the edges of $G_0$. Since $A_2$ is connected in $G_0$, and $x_0\in A_2$ satisfies $q_{x_0}=q_0$, we obtain that~$q_x = q_0$ for all $x\in A_2$. In particular, \cref{eq:qqy} simplifies to $q a_x \equiv p a_y \pmod{q_0}$. This concludes the proof of \cref{lem:getfirstformula}.
\end{proof}

The next technical lemma will be used to derive an approximate formula for the frequencies $\beta_x$ in the conclusion of \cref{lem:getfirstformula}.

\begin{lemma}
    \label{lem:betax}
    Let $\B = (B, \beta_{\bullet}, H)$ be a configuration whose frequencies satisfy $|\beta_x| \leq P^{-2\log Y}$ for all $x\in B$.\footnote{Here we slightly abuse notation by identifying each $\beta_x\in \R/\Z$ with its unique representative in $[-1/2, 1/2)$.} Suppose that the graph $G(\B)$ has minimum degree $\geq \delta |\P|^2 \geq 1$.

    Then, there exists $B'\subset B$ with $|\Q(\B|_{B'})| \gg \delta |B'||\P|^2$ and $x_0\in B'$ such that, for all $x\in B'$,
    \begin{equation*}
        \beta_x = \frac{x_0\beta_{x_0}}{x}+ O \!\left(\frac{ |\beta_{x_0}| \log Y}{Y} + \frac{\log Y}{HP}\right).
    \end{equation*}
    Furthermore, there is a multiset $V \subset \big[\tfrac13Y, 6Y\big]$ with $O(\log Y)$ elements in any unit interval, such that~$\gg \delta |V||\P|^2$ quadruples $(v_1,v_2,p,q)\in V^2 \times \P^2$ satisfy
    \begin{equation*}
        |\beta_{x_0}| \abs{pv_1-qv_2} \ll \frac{Y\log Y}{H}.
    \end{equation*}
\end{lemma}

\begin{proof}
    We first pass to a subgraph of $G(\B)$ of reasonable diameter, while maintaining a large edge concentration.

    Let $x_0\in B$. For $i\geq 0$, let $B_i$ be the set of all vertices $x\in B$ at distance at most $i$ from $x_0$ in $G(\B)$. Since~$|B|\leq Y$, there exists an integer $1\leq r\leq \log Y$ such that $|B_{r}|\leq 3|B_{r-1}|$. Then, by the minimum degree assumption on $G(\B)$, we have
    \begin{equation}
        \label{eq:lowerboundQBr}
        |\Q(\B|_{B_r})| \geq |B_{r-1}| \delta |\P|^2 \geq \tfrac{1}{3}\delta |B_{r}||\P|^2.
    \end{equation}

    We set $B' := B_r$. For any $x\in B'$, there exists a path of length $\leq \log Y$ from $x_0$ to $x$ in $G(\B)$. Hence, by \cref{lem:paths}, there are $R_x,S_x \leq (2P)^{\log Y}$ such that
    \begin{equation}
        \label{eq:applylempathsbeta}
        \abs{\frac{R_x}{S_x} x_0 - x} \leq \log Y \quad \text{and} \quad \norm{S_x \beta_{x_0} - R_x \beta_{x}} \leq \frac{(\log Y )R_x}{HP}.
    \end{equation}
    In particular, $R_x/S_x \in \big[\frac{1}{3}, 3\big]$. Hence, we may write the first part of \cref{eq:applylempathsbeta} as
    \begin{equation}
        \label{eq:boundRxSx}
        \abs{\frac{S_x}{R_x} - \frac{x_0}{x}} \leq \frac{3\log Y}{Y}.
    \end{equation}
    The second part of \cref{eq:applylempathsbeta} can be simplified using the assumption on the size of the frequencies of $\B$. Indeed, since $|\beta_{x_0}|, |\beta_x| \leq P^{-2\log Y}$ and $R_x, S_x \leq (2P)^{\log Y}$, we have $ \abs{S_x \beta_{x_0} - R_x \beta_{x}} \leq 1/2$, which implies that $\norm{S_x \beta_{x_0} - R_x \beta_{x}} = \abs{S_x \beta_{x_0} - R_x \beta_{x}}$. Therefore,
    \begin{equation}
        \label{eq:betaxbeta}
        \abs{\beta_x - \frac{S_x}{R_x} \beta_{x_0}} \leq \frac{\log Y}{HP}.
    \end{equation}
    Combining the estimates \cref{eq:betaxbeta,eq:boundRxSx}, we see that every $x\in B'$ satisfies the approximate formula
    \begin{equation*}
        \beta_x = \frac{x_0\beta_{x_0}}{x}+ O \!\left(\frac{ |\beta_{x_0}| \log Y}{Y} + \frac{\log Y}{HP}\right).
    \end{equation*}
    This proves the first part of the lemma.

    For the second part, we define the multiset $V$ consisting of the values $v_x := \frac{S_x}{R_x}Y$ for $x\in B'$. Since $S_x/R_x\in \big[\frac{1}{3}, 3\big]$, it is immediate that $V \subset \big[\frac{1}{3}Y, 6Y\big]$. Moreover, for any $x,y\in B'$ we have $|v_x - v_y| \geq \tfrac14 |x-y| - 6\log Y$ by \cref{eq:boundRxSx}, so that $V$ has $O(\log Y)$ elements in any unit interval.

    Finally, for any quadruple $(x,y,p,q)\in \Q(\B|_{B'})$, we have
    \begin{equation*}
        \frac{|\beta_{x_0}|}{Y} \abs{q v_x - p v_y} = \abs{q \frac{S_x}{R_x}\beta_{x_0} - p \frac{S_y}{R_y}\beta_{x_0}} \leq |q\beta_x - p\beta_y| + O \!\left(\frac{\log Y}{H}\right)
    \end{equation*}
    by \cref{eq:betaxbeta}. Since $|q\beta_x - p\beta_y| \ll 1/H$ by definition of $\Q(\B|_{B'})$, the desired bound follows.
\end{proof}

\section{Proof of the global structure theorem}
\label{sec:expansion}
We start this section by proving an unconditional expansion estimate, \cref{prop:GRHsubstitute}. Combining this with the results of the previous sections, we then establish the global structure theorem, \cref{thm:globalstructuretheorem}, from which we deduce our main Fourier uniformity result, \cref{thm:maingeneral}.

\subsection{Vinogradov-Korobov expansion estimate}
To circumvent the need for GRH, we employ a large value estimate for Dirichlet polynomials over primes. Specifically, we bound the number of pairs $(\chi, t)$ for which $\big\lvert\!\sum_{p \in \mathcal{P}} \chi(p) p^{it}\big\rvert$ is exceptionally large.

\begin{lemma}
    \label{lem:largevalues}
    Let $T\geq 3$ and $q\in \N$. Let $\T$ be a set of pairs $(\chi, t)$ where $\chi$ is a Dirichlet character modulo $q$ and $t\in [-T, T]$. Suppose that $\abs{t_1 - t_2}\geq 1$ for distinct pairs $(\chi, t_1),(\chi,t_2)\in \T$.

    Let $P\geq 3$ and $k\in \N$ be such that $P^k\leq T^2$. Let $(c_p)$ be a sequence of $1$-bounded complex numbers supported on the primes $p\in [P, 2P]$.

    Then, uniformly for $0< \eta <1/2$,
    \begin{equation*}
        \sum_{(\chi, t)\in \T} \bigg\lvert\sum_{p} \chi(p)c_p p^{it}\bigg\rvert^{2k}
        \leq k^{O(k)} P^{2k} \Big(1 + |\T| q^{\eta}P^{-\eta k} \big(T^{5\eta^{3/2}} \log T+ \eta^{-1}\big)\Big).
    \end{equation*}
\end{lemma}

\Cref{lem:largevalues} is a straightforward adjustment of \cite[Lemma~6.6]{MRTTZ}, which treats the case $k=1$ (and achieves a more refined bound). For completeness,\footnote{The proof of \cite[Lemma~6.6]{MRTTZ} is not explicitly provided in~\cite{MRTTZ}, but instead suggested as a modification of \cite[Lemma~4.4]{MRII}.} we provide a detailed proof of \cref{lem:largevalues} below.

The key ingredient is the following Vinogradov-Korobov type upper bound for Dirichlet $L$-functions.

\begin{lemma}
    \label{lem:fordVK}
    Let $q\geq 1$ and let $\chi$ be a Dirichlet character modulo $q$. For $1/2 \leq \sigma < 1$ and $|t|\geq 3$, we have
    \begin{equation*}
        \abs{L(\sigma+it, \chi)} \ll q^{1-\sigma} |t|^{5(1-\sigma)^{3/2}} (\log{|t|})^{2/3} + \frac{q^{1-\sigma}}{1-\sigma}.
    \end{equation*}
\end{lemma}
\begin{proof}
    As explained in \cite[p.92]{MRTTZ}, this is a direct consequence of the bound \cite[Theorem~1]{ford} (see also~\cite{richert}) on the Hurwitz zeta function $\zeta(s, u)$, using the identity $L(s, \chi) = q^{-s} \sum_{m=1}^q \chi(m) \zeta(s, m/q)$.
\end{proof}

\begin{proof}[Proof of \cref{lem:largevalues}]
    By multiplicativity, we have
    \begin{equation*}
        \bigg(\sum_{p} \chi(p) c_p p^{it}\bigg)^k = \sum_{n} a_n \chi(n) n^{it}
    \end{equation*}
    for some complex coefficients $|a_n| \leq k!$ supported on $n\in [P^k, (2P)^k]$. Let $W:\R\to \R$ be a smooth function such that $\ind{[1,2]} \leq W \leq \ind{[0.5,2.5]}$ pointwise. By repeated integration by parts, its Mellin transform satisfies $|\widetilde{W}(\sigma + i t)| \ll_A (1+|t|)^{-A}$ for $|\sigma| \leq 2$ and $t\in \R$, for every $A>0$. Fix some scale $P^k\leq N\leq (2P)^k$. By duality (see~\cite[Section~7.1]{IK}) and summing over dyadic scales, it suffices to prove the estimate
    \begin{equation}
        \label{eq:dual}
        \sum_{n} W\Big(\frac{n}{N}\Big) \bigg\lvert \sum_{(\chi, t)\in \T} b_{\chi,t\,} \chi(n) n^{it} \bigg\rvert^2 \ll N \Big(1 + |\T| q^{\eta}N^{-\eta } \big(T^{5\eta^{3/2}} \log T+ \eta^{-1}\big)\Big) \sum_{(\chi, t)\in \T} |b_{\chi,t}|^2
    \end{equation}
    for arbitrary coefficients $b_{\chi,t}\in \C$.

    Expanding the square, the left-hand side of \cref{eq:dual} becomes
    \begin{equation}
        \label{eq:dualexpanded}
        \sum_{(\chi_1, t_1),(\chi_2, t_2) \in \T} b_{\chi_1,t_1} \overline{b_{\chi_2,t_2}}\sum_{n} W\Big(\frac{n}{N}\Big) \chi_1\overline{\chi_2}(n) n^{i(t_1-t_2)}.
    \end{equation}
    By Mellin inversion,
    \begin{equation*}
        \sum_{n} W\Big(\frac{n}{N}\Big) \chi_1\overline{\chi_2}(n) n^{i(t_1-t_2)} = \frac{1}{2\pi i} \int_{2-i\infty}^{2+i\infty} N^s \widetilde{W}(s)L\big(s-i(t_1-t_2), \chi_1\overline{\chi_2}\big) ds.
    \end{equation*}
    We can restrict the range of integration to $\lvert\Im(s)\rvert\leq 3T$, at the cost of a negligible error term $O_A(T^{-A})$, using that $N \leq (2P)^k\leq T^{O(1)}$. We then shift the contour to $\Re(s) = 1-\eta$, obtaining a contribution
    \begin{equation*}
        \ind{\chi_1 = \chi_2} N^{1-i(t_1-t_2)} \widetilde{W}(1-i(t_1-t_2)) \ll_A \ind{\chi_1 = \chi_2} N (1+|t_1-t_2|)^{-A}
    \end{equation*}
    from the possible pole at $s = 1-i(t_1-t_2)$. Appealing to \cref{lem:fordVK}, the contribution of the remaining integrals is bounded by
    \begin{equation*}
        \ll_A N^{1-\eta} q^{\eta} \left( T^{5\eta^{3/2}} \log{T} + \eta^{-1}\right) + T^{-A}.
    \end{equation*}
    Combining these bounds, we get that \cref{eq:dualexpanded} is
    \begin{equation*}
        \ll \sum_{(\chi_1, t_1),(\chi_2, t_2) \in \T} \abs{b_{\chi_1,t_1}}  \abs{b_{\chi_2,t_2}} \left( \frac{\ind{\chi_1 = \chi_2} N}{ (1+|t_1-t_2|)^{2}}+ N^{1-\eta} q^{\eta} \left( T^{5\eta^{3/2}} \log{T} + \eta^{-1}\right) \right).
    \end{equation*}
    By the inequality $2|b_1| |b_2| \leq |b_1|^2 + |b_2|^2$ and the separation assumption on $\T$, we obtain \cref{eq:dual}, which concludes the proof of \cref{lem:largevalues}.
\end{proof}

\begin{lemma}
    \label{lem:boundT}
    Let $T$ be sufficiently large. Let $P = \exp((\log T)^{\theta})$ for some $\frac{1}{1000}\leq \theta \leq \frac{2}{3}-\frac{1}{1000}$ and let~$\P$ be the set of all primes in~$[P, 2P]$. Let $1\leq q \leq \exp((\log T)^{1-\theta/2})$ be an integer.

    Let $\T$ be a set of pairs $(\chi, t)$ such that
    \begin{equation*}
        \bigg\lvert \sum_{p\in \P} \chi(p) p^{it} \bigg\rvert \geq \frac{ |\P| }{(\log T)^{100}},
    \end{equation*}
    where $\chi$ is a Dirichlet character modulo $q$ and $t\in [-T, T]$. Suppose that $\abs{t_1 - t_2}\geq 1$ for distinct pairs $(\chi, t_1),(\chi,t_2)\in \T$.

    Then
    \begin{equation*}
        \abs{\T} \leq \exp\!\big(O\big((\log T)^{1-\frac{3\theta}{2}}(\log \log T)^3\big)\big).
    \end{equation*}
\end{lemma}
\begin{proof}
    Applying \cref{lem:largevalues} with
    \begin{empheq}[left=\Bigg\{]{align*}\eta &= (\log T)^{-\theta} (\log \log T)^2\\
        k &= \lceil (\log T)^{1-\frac{3\theta}{2}}(\log\log T)^2 \rceil,
    \end{empheq}
    and all coefficients $c_p = 1$ (the condition $P^k\leq T^2$ is satisfied for large $T$), we obtain
    \begin{equation*}
        |\T| \left( \frac{ |\P| }{(\log T)^{100}}\right)^{2k} \leq k^{O(k)} P^{2k} \Big(1 + |\T| q^{\eta}P^{-\eta k} \big(T^{5\eta^{3/2}} \log T+ \eta^{-1}\big)\Big).
    \end{equation*}
    By our choices of parameters, this simplifies to
    \begin{equation*}
        |\T| \leq (\log T)^{O(k)} \Big(1 + |\T| q^{\eta} T^{5\eta^{3/2}} P^{-\eta k}\Big).
    \end{equation*}
    Thus, either
    \begin{equation*}
        |\T| \leq (\log T)^{O(k)} \leq \exp\!\big(O\big((\log T)^{1-\frac{3\theta}{2}}(\log \log T)^3\big)\big)
    \end{equation*}
    and we are done, or
    \begin{equation}
        \label{eq:imposscase}
        P^{\eta k} \leq (\log T)^{O(k)} q^{\eta} T^{5\eta^{3/2}}.
    \end{equation}
    However, in view of the estimates
    \begin{empheq}[left=\empheqlbrace]{align*}
        q^\eta &\leq \exp\!\big((\log T)^{1-3\theta/2}(\log \log T)^2\big)\\
        T^{\eta^{3/2}} &= \exp\!\big((\log T)^{1-3\theta/2}(\log \log T)^3\big)\\
        P^{\eta k} &\geq \exp\!\big((\log T)^{1-3\theta/2}(\log \log T)^4\big),
    \end{empheq}
    the inequality \cref{eq:imposscase} cannot hold for large enough $T$.
\end{proof}

We can now establish our expansion estimate, \cref{prop:GRHsubstitute}, by passing to the Fourier side and applying the preceding large value estimate.

\begin{proposition}[Expansion Estimate]
    \label{prop:GRHsubstitute}
    Let $Y$ be sufficiently large. Let $P = \exp((\log Y)^{\theta})$ for some~$\frac{1}{100}\leq \theta \leq \frac{2}{3}-\frac{1}{100}$ and let $\P$ be the set of primes in~$[P, 2P]$.

    Let $A\subset \big[\tfrac{1}{10}Y, 10Y\big]$ be a multiset with $O(\log Y)$ elements in any unit interval.

    Let $1\leq q_0 \leq \exp((\log Y)^{1-\theta/2})$ be an integer, let $\P' \subset \P$ consist of all primes not dividing $q_0$, and let~$(a_x)_{x\in A}$ be a sequence of integers coprime to $q_0$.

    Fix $1/Y\leq \eps \leq 1$. Suppose that there are $\geq (\log Y)^{-100} |A| |\P|^2$ quadruples $(x,y,p,q)\in A^2\times (\P')^2$ such that
    \begin{equation}
        \label{eq:conditionsquadruple}
        \abs{px - qy} \leq \eps P\quad \text{and}\quad q a_x \equiv p a_y \pmod{q_0}.
    \end{equation}
    Then
    \begin{equation*}
        \eps^{-1} q_0 \frac{Y}{|A|}  \leq \exp\!\big(O\big((\log Y)^{1-\frac{3\theta}{2}}(\log \log Y)^3\big)\big).
    \end{equation*}
\end{proposition}

Taking $q_0 = \eps = 1$ in \cref{prop:GRHsubstitute}, we immediately deduce the following.

\begin{corollary}
    \label{cor:expansion}
    Suppose that $P = \exp\!\big((\log Y)^{\theta}\big)$ for some $\frac{1}{100}\leq \theta \leq \frac{2}{3}-\frac{1}{100}$.

    Let $\A = (A, \alpha_{\bullet}, H)$ be a configuration such that $|\Q(\A)| \geq \delta |A||\P|^2$ for some $\delta \geq (\log Y)^{-100}$. Then,
    \begin{equation*}
        \frac{Y}{|A|}  \leq \exp\!\big(O\big((\log Y)^{1-\frac{3\theta}{2}}(\log \log Y)^3\big)\big).
    \end{equation*}
\end{corollary}

\begin{proof}[Proof of \cref{prop:GRHsubstitute}]
    By standard rounding and multiplicity arguments, we may assume that $A\subset \big[\tfrac{1}{10}Y, 10Y\big]$ is a set of integer multiples of $\eps$ (rather than a multiset of arbitrary real numbers), such that there are $\gg (\log Y)^{-102} |A| |\P|^2$ quadruples $(x,y,p,q)\in A^2\times (\P')^2$ satisfying
    \begin{equation}
        \label{eq:newconditionquadruple}
        \abs{px - qy} \leq 5\eps P\quad \text{and}\quad q a_x \equiv p a_y \pmod{q_0}.
    \end{equation}

    Let~${\Phi:\R\to \R}$ be a Schwartz function such that $\Phi \geq \ind{[-50,\, 50]}$ and whose Fourier transform is supported on $[-1, 1]$. If $|px-qy|\leq 5\eps P$, then $\abs{\log px - \log qy} \leq 50\eps/Y$ by the mean value theorem. Thus, the number of quadruples $(x,y,p,q)\in A^2\times (\P')^2$ satisfying \cref{eq:newconditionquadruple} is
    \begin{equation*}
        \leq \sum_{x,y\in A} \sum_{p,q \in \P'} \Phi\bigg(\frac{Y}{\eps} \log \frac{px}{qy}\bigg) \ind{qa_x\equiv pa_y \spmod{q_0}}.
    \end{equation*}
    By orthogonality of characters and Fourier inversion, this expression can be rewritten as
    \begin{equation}
        \label{eq:fourierexpr}
        \frac{1}{\phi(q_0)} \sum_{\chi \spmod{q_0}} \frac{\eps}{Y} \int_{-Y/\eps}^{Y/\eps} \widehat{\Phi}\!\left(\frac{\eps t}{Y}\right) \abs{\sum_{x\in A} \chi(a_x) x^{it}}^2 \Bigg\lvert{\sum_{p\in \P} \overline{\chi(p)} p^{it}}\Bigg\rvert^2 dt.
    \end{equation}

    We bound \cref{eq:fourierexpr} by separating the contribution of those $(\chi, t)$ for which the prime Dirichlet polynomial is small from those for which it is large. By the mean-value theorem for Dirichlet polynomials \cite[Theorem~9.1]{IK}, recalling that $\eps^{-1}A\subset \N$, we have
    \begin{equation*}
        \frac{\eps}{Y} \int_{-Y/\eps}^{Y/\eps} \abs{\sum_{x\in A} \chi(a_x) x^{it}}^2 dt = \frac{\eps}{Y} \int_{-Y/\eps}^{Y/\eps} \abs{\sum_{x\in A} \chi(a_x) (\eps^{-1}x)^{it}}^2 dt
        \ll |A|.
    \end{equation*}
    Thus, for any $w>0$, \cref{eq:fourierexpr} is bounded by
    \begin{equation*}
        \ll (w|\P|)^2 |A| + \frac{\eps |A|^2|\P|^2}{Y \phi(q_0)} \sum_{\chi \spmod{q_0}} \int_{-Y/\eps}^{Y/\eps} \ind{\{\lvert\sum_{p\in \P} \overline{\chi(p)} p^{it} \rvert > w |\P|\}} dt.
    \end{equation*}
    Choose $w := (\log (Y/\eps))^{-100}$. Discretising and applying \cref{lem:boundT} with $T = Y/\eps \in [Y, Y^2]$, we have
    \begin{equation*}
        \sum_{\chi \spmod{q_0}} \int_{-Y/\eps}^{Y/\eps} \ind{\{\lvert\sum_{p\in \P} \overline{\chi(p)} p^{it} \rvert > w |\P|\}} dt \leq \exp\!\big(O\big((\log Y)^{1-\frac{3\theta}{2}}(\log \log Y)^3\big)\big).
    \end{equation*}

    Combining the above estimates, we obtain
    \begin{equation*}
        \frac{|A||\P|^2}{(\log Y)^{102}} \ll \frac{|A||\P|^2}{(\log Y)^{200}} +  \frac{\eps|A|^2|\P|^2}{Y \phi(q_0)} \exp\!\big(O\big((\log Y)^{1-\frac{3\theta}{2}}(\log \log Y)^3\big)\big).
    \end{equation*}
    Rearranging and using that ${\phi(n) \gg n^{1-o(1)}}$, the result follows for $Y$ sufficiently large.
\end{proof}

\subsection{Proof of the global structure theorem}
Combining this expansion estimate with the results of \cref{sec:largeH} yields the following approximate formula for the frequencies of configurations having a very large parameter $H$.

\begin{lemma}
    \label{lem:largeHcase}
    Suppose that $Y$ is sufficiently large and $P = \exp((\log Y)^{\theta})$ for some $\frac{1}{100}\leq \theta \leq \frac{2}{3}-\frac{1}{100}$.

    Let $\A = (A, \alpha_{\bullet}, H)$ be a configuration such that $|\Q(\A)| \geq \delta |A| |\P|^2$, where $H \geq P^{3\log Y}$ and $\delta \geq C_1/\log |\P|$ for some sufficiently large absolute constant $C_1>0$.

    Then, writing $f(Y) := (\log Y)^{1-\frac{3\theta}{2}} (\log \log Y)^{3}$, there exist
    \begin{itemize}
        \item a subset $A'\subset A$ with $|\Q(\A|_{A'})| \gg \delta |A'| |\P|^2$,
        \item an integer $q_0 \leq e^{O(f(Y))}$, and
        \item a real number $t$ with $|t| \leq e^{O(f(Y))} \frac{Y^2}{HP}$,
    \end{itemize}
    such that, for every $x\in A'$, there is an integer $a_x$ coprime to $q_0$ such that
    \begin{equation}
        \label{eq:formulalargeH}
        \alpha_x = \frac{a_x}{q_0} + \frac{t}{x} + O\bigg(\frac{e^{O(f(Y))}}{HP}\bigg) \pmod 1.
    \end{equation}
\end{lemma}

\begin{proof}
    Applying \cref{lem:getfirstformula}, we obtain a subset $A_2\subset A$, an integer $q_0$ and sequences $(a_x)_{x\in A_2}$ and $(\beta_x)_{x\in A_2}$ satisfying the properties given in the lemma. In particular, $q_0$ satisfies
    \begin{equation}
        \label{eq:initialbdq0}
        q_0 \leq \exp\!\big( O\big((\log Y)^{1-\theta}\log \log Y + (\log Y)^{\theta}\big) \big).
    \end{equation}
    Given the range of $\theta$, this implies that $q_0 \leq \exp((\log Y)^{1-\theta/2})$ for $Y$ sufficiently large, which is the condition needed to apply the expansion estimate, \cref{prop:GRHsubstitute}.

    We apply it to the set $A_2$; all the other assumptions are satisfied since, by property \cref{item:getform2} of \cref{lem:getfirstformula}, there are $\gg \delta |A_2| |\P|^2$ quadruples $(x,y,p,q)\in (A_2)^2\times \P^2$ such that $(pq, q_0)=1$, $|px-qy| \leq P/10$ and $q a_x \equiv p a_y \pmod{q_0}$. Hence, \cref{prop:GRHsubstitute} gives $q_0 \leq e^{O(f(Y))}$, thus sharpening our initial bound \cref{eq:initialbdq0}.

    Let $\B := (A_2, \beta_\bullet, H)$. This configuration satisfies all the assumptions of \cref{lem:betax}. Indeed, $G(\B)$ has minimum degree $\gg \delta |\P|^2$ by property \cref{item:getform2} of \cref{lem:getfirstformula}, and the frequencies $(\beta_x)_{x\in A_2}$ satisfy $|\beta_x| \leq P^{-2\log Y}$ for $Y$ large enough. \Cref{lem:betax} then provides a subset $B' \subset A_2$ with $|\Q(\B|_{B'})| \gg \delta |B'| |\P|^2$ and an element $x_0 \in B'$ such that, for every $x\in B'$, we have
    \begin{equation*}
        \beta_x = \frac{x_0\beta_{x_0}}{x}+ O \!\left(\frac{ |\beta_{x_0}| \log Y}{Y} + \frac{\log Y}{HP}\right).
    \end{equation*}

    Combining this with the formula for $\alpha_x$ from \cref{lem:getfirstformula}, we obtain that, for every $x\in B'$,
    \begin{equation}
        \label{eq:alphaxbetaxzero}
        \alpha_x = \frac{a_x}{q_0} + \frac{x_0\beta_{x_0}}{x} + O \!\left(\frac{ |\beta_{x_0}| \log Y}{Y} + \frac{\log Y}{HP}\right) \pmod 1.
    \end{equation}

    It remains to estimate $|\beta_{x_0}|$. To do so, we use the final property in \cref{lem:betax}, which states that there is a multiset $V \subset \big[\tfrac13 Y, 6Y \big]$ with $O(\log Y)$ elements in any unit interval, such that $\gg \delta |V| |\P|^2$ quadruples $(v_1, v_2, p,q)\in V^2 \times \P^2$ satisfy
    \begin{equation*}
        |\beta_{x_0}| \abs{pv_1-qv_2} \ll \frac{Y\log Y}{H}.
    \end{equation*}
    If $|\beta_{x_0}| > e^{C f(Y)} \frac{Y}{HP}$ for some large enough constant $C>0$, this contradicts the expansion estimate, \cref{prop:GRHsubstitute}, applied to the set $V$ with $q_0=1$. Thus, we must have $|\beta_{x_0}| \leq e^{O(f(Y))} \frac{Y}{HP}$ which, when plugged back into the formula \cref{eq:alphaxbetaxzero} for the frequencies, yields the desired conclusion.
\end{proof}

We now have all the necessary ingredients to prove the global structure theorem, which we recall for convenience.

\globalthm*

\begin{proof}[Proof of \cref{thm:globalstructuretheorem}]
    We start by applying \cref{prop:tower} to construct a tower of configurations above $\A$.

    Suppose first that we are in case \cref{item:smallcase} of \cref{prop:tower}. Then, there is a subset $B\subset A$ such that $|\Q(\A|_B)| \gg \delta |B| |\P|^2$ and, writing $\theta := 2/5 + \eps$, we have
    \begin{equation*}
        |B| \ll \exp\!\bigg(-\frac{(\log Y)^{\theta}}{(\log \log Y)^{O(1)}} \bigg) |A|.
    \end{equation*}
    On the other hand, by \cref{cor:expansion}, we have
    \begin{equation*}
        |B| \geq \exp\!\Big(-(\log Y)^{1-\frac{3\theta}{2}}(\log \log Y)^{O(1)} \Big) |A|.
    \end{equation*}
    Since $1-\frac{3\theta}{2} < \theta - \eps$, this is a contradiction for $Y$ large enough in terms of $\eps_0$.

    Therefore, we must be in case \cref{item:towercase} of \cref{prop:tower}: there is a tower $(\A_i)_{0\leq i\leq k}$ of height $k \gg (\log Y)^{10}$, such that $\A_0 = \A$ and, if $\A_k = (A_k, \widetilde{\alpha}_{\bullet}, H_k)$, we have $|\Q(\A_k)| \gg \delta |A_k| |\P|^2$.

    By definition of a lift, we have $H_k \geq H P^k$, which is comfortably larger than $P^{3\log Y}$ for $Y$ large enough. We can thus apply \cref{lem:largeHcase} to $\A_k$, obtaining a subset $A' \subset A_k$ with $|\Q(\A_k|_{A'})| \gg \delta |A'| |\P|^2$, an integer $q_0 \leq e^{O(f(Y))}$ and a real number $|t| \leq e^{O(f(Y))} \frac{Y^2}{H_k P}$, such that, for every $x\in A'$,
    \begin{equation*}
        \widetilde{\alpha}_x = \frac{a_x}{q_0} + \frac{t}{x} + O\bigg(\frac{e^{O(f(Y))}}{H_k P}\bigg) \pmod 1
    \end{equation*}
    for some integer $a_x$ coprime to $q_0$. Here, we recall that $f(Y) := (\log Y)^{1-\frac{3\theta}{2}} (\log \log Y)^{3}$.

    Applying our expansion estimate, \cref{cor:expansion}, to the configuration $\A_k|_{A'}$, we obtain the lower bound $|A'| \geq e^{-O(f(Y))} Y$. By pigeonholing on the values of $a_x \spmod{q_0}$, we can find a further subset $A''\subset A'$ with $|A''| \geq e^{-O(f(Y))} |A|$ such that $a_x = a_0$ is constant for $x\in A''$.

    Let $p_1^*, \ldots, p_k^*$ be the lifting primes for the tower $(\A_i)_{0\leq i\leq k}$. By the definition of a lift, we have $H_k = H p_1^* \cdots p_k^*$ and $\alpha_x = p_1^* \cdots p_k^* \widetilde{\alpha}_x \pmod 1$ for every $x\in A_k$. Thus, for every $x\in A''$, we have
    \begin{equation*}
        \alpha_x = \frac{p_1^* \cdots p_k^* a_0}{q_0} + \frac{p_1^* \cdots p_k^* t}{x} + O\bigg(\frac{e^{O(f(Y))}}{H P}\bigg) \pmod 1.
    \end{equation*}
    Relabelling the parameters gives the desired conclusion.
\end{proof}
\part{Local structure theorem}
\label{part:local}

Throughout \cref{part:local}, we fix an arbitrary configuration $\A = (A, \alpha_{\bullet}, H)$. All results and definitions in these sections implicitly refer to this configuration. Unless specified otherwise, the parameters $Y$, $P$, and $H$ are only assumed to satisfy the default bounds in \cref{def:globalparam,def:configuration}.

\section{Preliminary lemmas}
\label{sec:preliminarylemmas}

We begin by establishing some basic properties of the function $\dist{p,q}{x,y}$ from \cref{def:locrelations}. These elementary facts appear in previous works under different notation; we include short proofs here to keep our exposition self-contained.

\subsection{Triangle-like inequalities} Importantly, the function $\dist{p,q}{x,y}$ does \emph{not} satisfy a triangle-like inequality of the form
\begin{equation*}
    {\dist{p,r}{x,z} \leq \dist{p,q}{x,y} + \dist{q,r}{y,z}},
\end{equation*}
due to the second argument of the maximum in \cref{eq:defdist}. Yet, a useful substitute is available: we prove it in \cref{lem:triangle} below.

\begin{lemma}
    \label{lem:normcalculus}
    Let $m,n\geq 1$ be coprime integers, $\alpha\in \R$ and $0< \eps \leq 1$. Suppose that $\norm{m\alpha} < \frac{\eps}{2n}$ and $\norm{n\alpha} < \frac{1}{2m}$. Then $\norm{\alpha} < \frac{\eps}{2mn}$.
\end{lemma}

\begin{proof}
    The assumption $\norm{n\alpha} < \frac{1}{2m}$ means that $\alpha = \frac{k}{n} + \frac{\theta}{2mn} \pmod{1}$ for some $|\theta| < 1$ and $k\in \Z$. Thus
    \begin{equation*}
        m\alpha = \frac{km}{n} + \frac{\theta}{2n} \pmod{1},
    \end{equation*}
    but since $\norm{m\alpha} < \frac{\eps}{2n} \leq \frac{1}{2n}$ and $|\theta| < 1$, it must be that $n\mid km$, i.e.~$n\mid k$ as $(m,n) = 1$. Therefore,~$\alpha = \frac{\theta}{2mn} \pmod{1}$. Recalling that $\norm{m\alpha} < \frac{\eps}{2n}$, we see that $|\theta| < \eps$, so that $\norm{\alpha} < \frac{\eps}{2mn}$.
\end{proof}

\begin{minipage}[t]{0.74\textwidth}
    \begin{lemma}[{\cite[Lemma~2.3]{walsh1}}]
        \label{lem:triangle}
        Let $0< C \leq \tfrac{H}{30P^2}$. Let $x,y_1,y_2,z\in A$ and ${p,q_1,q_2,r\in \P}$. Suppose that
        \begin{equation*}
            \dist{p,q_i}{x,y_i}, \dist{q_i,r}{y_i,z} \leq C
        \end{equation*}
        for $i=1,2$. If $q_1\neq q_2$, then
        \begin{equation*}
            \dist{p,r}{x,z} \leq 4C.
        \end{equation*}
    \end{lemma}
\end{minipage}
\hfill
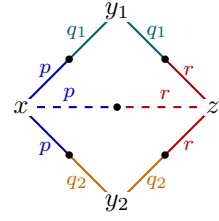
\begin{minipage}[t]{0.25\textwidth}
    \vspace{-3pt} \centering
    \raisebox{0cm}{ \begin{tikzpicture}
            \begin{scope}[scale=0.85, every node/.style={font=\small}]
                \colorlet{cp}{blue!70!black}      \colorlet{cr}{red!70!black}       \colorlet{cqone}{teal!80!black}   \colorlet{cqtwo}{orange!80!black}

                \coordinate (x)  at (0, 0);
                \coordinate (y1) at (1.5, 1.5);
                \coordinate (z)  at (3, 0);
                \coordinate (y2) at (1.5, -1.5);

                \draw[thick, cp]    (x)  -- ($(x)!0.5!(y1)$);
                \draw[thick, cqone] (y1) -- ($(x)!0.5!(y1)$);

                \draw[thick, cqone] (y1) -- ($(y1)!0.5!(z)$);
                \draw[thick, cr]    (z)  -- ($(y1)!0.5!(z)$);

                \draw[thick, cr]    (z)  -- ($(z)!0.5!(y2)$);
                \draw[thick, cqtwo] (y2) -- ($(z)!0.5!(y2)$);

                \draw[thick, cqtwo] (y2) -- ($(y2)!0.5!(x)$);
                \draw[thick, cp]    (x)  -- ($(y2)!0.5!(x)$);

                \draw[dashed, thick, cp] (x) -- ($(x)!0.5!(z)$);
                \draw[dashed, thick, cr] (z) -- ($(x)!0.5!(z)$);

                \foreach \A/\B in {x/y1, y1/z, z/y2, y2/x, x/z} {
                        \path let \p1 = (\A), \p2 = (\B) in
                        node[fill=black, circle, inner sep=1pt] at ($(\p1)!0.5!(\p2)$) {};
                    }

                \node[fill=white, inner sep=1pt] at (x)  {\(x\)};
                \node[fill=white, inner sep=1pt] at (y1) {\(y_1\)};
                \node[fill=white, inner sep=1pt] at (z)  {\(z\)};
                \node[fill=white, inner sep=1pt] at (y2) {\(y_2\)};

                \node[text=cp, font=\scriptsize] at ($ (x)!0.25!(y1) + (0, 0.20) $) {\(p\)};
                \node[text=cp, font=\scriptsize] at ($ (x)!0.25!(y2) + (0, -0.27) $) {\(p\)};
                \node[text=cp, font=\scriptsize] at ($ (x)!0.25!(z)  + (0, 0.20) $) {\(p\)};

                \node[text=cr, font=\scriptsize] at ($ (z)!0.25!(y1) + (0, 0.20) $) {\(r\)};
                \node[text=cr, font=\scriptsize] at ($ (z)!0.25!(y2) + (0, -0.22) $) {\(r\)};
                \node[text=cr, font=\scriptsize] at ($ (z)!0.25!(x)  + (0, 0.20) $) {\(r\)};

                \node[text=cqone, font=\scriptsize] at ($ (y1)!0.25!(x) + (-0.25, 0.05) $) {\(q_1\)};
                \node[text=cqone, font=\scriptsize] at ($ (y1)!0.25!(z) + (0.25, 0.05) $) {\(q_1\)};

                \node[text=cqtwo, font=\scriptsize] at ($ (y2)!0.25!(x) + (-0.25, -0.05) $) {\(q_2\)};
                \node[text=cqtwo, font=\scriptsize] at ($ (y2)!0.25!(z) + (0.25, -0.05) $) {\(q_2\)};

            \end{scope}
        \end{tikzpicture}
    }
\end{minipage}

\begin{proof}
    From the assumptions $\norm{p\alpha_{y_i} - q_i \alpha_x}\leq C/H$ and $\norm{q_i \alpha_z - r \alpha_{y_i}} \leq C/H$, we have
    \begin{equation*}
        \mnorm{q_i(p\alpha_z - r\alpha_x)} \leq \mnorm{r(p\alpha_{y_i} - q_i\alpha_x)} + \mnorm{p(q_i \alpha_z - r\alpha_{y_i})} \leq \frac{4CP}{H}
    \end{equation*}
    for each $i=1,2$. Since $q_1\neq q_2$, we can apply \cref{lem:normcalculus} to deduce that $\norm{p\alpha_z - r\alpha_x} \leq 4C/H$. It remains to prove that $\abs{px - rz}\leq 4CP$, but the stronger bound $\abs{px-rz} \leq 2CP$ follows directly from the standard triangle inequality.
\end{proof}

We also require a version of \cref{lem:triangle} for `parallelogram patterns'.

\begin{minipage}[t]{0.69\textwidth}
    \begin{lemma}[{\cite[Lemma~3.5]{walsh3}}]
        \label{lem:parallelogram}
        Let $0<C \leq \tfrac{H}{30P^2}$.

        Let $x_1, x_1', x_2, x_2', y, y'\in A$ and~$p_1,p_2,q_1,q_2,r,s\in \P$. Suppose that
        \begin{equation*}
            \dist{p_i,q_i}{x_i, y}, \dist{p_i,q_i}{x_i', y'}, \dist{r,s}{x_i, x_i'} \leq C
        \end{equation*}
        for $i=1,2$. If $p_1\neq p_2$, then
        \begin{equation*}
            \dist{r,s}{y,y'} \leq 6C.
        \end{equation*}
    \end{lemma}
\end{minipage}\hfill
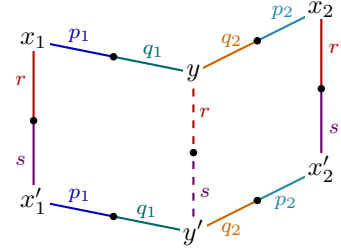
\begin{minipage}[t]{0.31\textwidth}
    \vspace{-10pt} \centering
    \begin{tikzpicture}
        \begin{scope}[scale=0.85, every node/.style={font=\small}]
            \colorlet{cpone}{blue!70!black}     \colorlet{cptwo}{cyan!60!black}     \colorlet{cqone}{teal!80!black}     \colorlet{cqtwo}{orange!80!black}   \colorlet{cr}{red!70!black}         \colorlet{cs}{violet!80!black}

            \coordinate (x1)   at (-1, 0.5);
            \coordinate (y)    at (1.5, 0);
            \coordinate (x1p)  at (-1, -2);
            \coordinate (yp)   at (1.5, -2.5);
            \coordinate (x2)   at (3.5, 1);
            \coordinate (x2pb) at (3.5, -1.5);

            \draw[thick, cpone] (x1) -- ($(x1)!0.5!(y)$);
            \draw[thick, cqone] (y)  -- ($(x1)!0.5!(y)$);

            \draw[thick, cpone] (x1p) -- ($(x1p)!0.5!(yp)$);
            \draw[thick, cqone] (yp)  -- ($(x1p)!0.5!(yp)$);

            \draw[thick, cr] (x1)  -- ($(x1)!0.5!(x1p)$);
            \draw[thick, cs] (x1p) -- ($(x1)!0.5!(x1p)$);

            \draw[dashed, thick, cr] (y)  -- ($(y)!0.5!(yp)$);
            \draw[dashed, thick, cs] (yp) -- ($(y)!0.5!(yp)$);

            \draw[thick, cptwo] (x2) -- ($(x2)!0.5!(y)$);
            \draw[thick, cqtwo] (y)  -- ($(x2)!0.5!(y)$);

            \draw[thick, cptwo] (x2pb) -- ($(x2pb)!0.5!(yp)$);
            \draw[thick, cqtwo] (yp)   -- ($(x2pb)!0.5!(yp)$);

            \draw[thick, cr] (x2)   -- ($(x2)!0.5!(x2pb)$);
            \draw[thick, cs] (x2pb) -- ($(x2)!0.5!(x2pb)$);

            \foreach \A/\B in {x1/y, x1/x1p, x1p/yp, y/yp, x2/y, x2pb/yp, x2/x2pb} {
                    \path let \p1=(\A), \p2=(\B) in
                    node[fill=black, circle, inner sep=1pt] at ($(\p1)!0.5!(\p2)$) {};
                }

            \node[fill=white, inner sep=1pt] at (x1)   {\(x_1\)};
            \node[fill=white, inner sep=1pt] at (y)    {\(y\)};
            \node[fill=white, inner sep=1pt] at (x1p)  {\(x_1'\)};
            \node[fill=white, inner sep=1pt] at (yp)   {\(y'\)};
            \node[fill=white, inner sep=1pt] at (x2)   {\(x_2\)};
            \node[fill=white, inner sep=1pt] at (x2pb) {\(x_2'\)};

            \node[text=cpone, font=\scriptsize] at ($ (x1)!0.25!(y) + (0.1, 0.2) $) {\(p_1\)};
            \node[text=cqone, font=\scriptsize] at ($ (y)!0.25!(x1) + (0, 0.2) $) {\(q_1\)};

            \node[text=cpone, font=\scriptsize] at ($ (x1p)!0.25!(yp) + (0.1, 0.2) $) {\(p_1\)};
            \node[text=cqone, font=\scriptsize] at ($ (yp)!0.25!(x1p) + (-0.1, 0.2) $) {\(q_1\)};

            \node[text=cr, font=\scriptsize] at ($ (x1)!0.25!(x1p) + (-0.2, 0) $) {\(r\)};
            \node[text=cs, font=\scriptsize] at ($ (x1p)!0.25!(x1) + (-0.2, 0) $) {\(s\)};

            \node[text=cr, font=\scriptsize] at ($ (y)!0.25!(yp) + (0.18, 0) $) {\(r\)};
            \node[text=cs, font=\scriptsize] at ($ (yp)!0.25!(y) + (0.18, 0) $) {\(s\)};

            \node[text=cptwo, font=\scriptsize] at ($ (x2)!0.25!(y) + (-0.1, 0.2) $) {\(p_2\)};
            \node[text=cqtwo, font=\scriptsize] at ($ (y)!0.25!(x2) + (0.1, 0.25) $) {\(q_2\)};

            \node[text=cptwo, font=\scriptsize] at ($ (x2pb)!0.25!(yp) + (-0.05, -0.28) $) {\(p_2\)};
            \node[text=cqtwo, font=\scriptsize] at ($ (yp)!0.25!(x2pb) + (0.1, -0.2) $) {\(q_2\)};

            \node[text=cr, font=\scriptsize] at ($ (x2)!0.25!(x2pb) + (0.2, 0) $) {\(r\)};
            \node[text=cs, font=\scriptsize] at ($ (x2pb)!0.25!(x2) + (0.2, 0) $) {\(s\)};
        \end{scope}
    \end{tikzpicture}
\end{minipage}

\begin{proof}
    By the triangle inequality, for $i=1,2$, we have
    \begin{equation*}
        \mnorm{p_i(s\alpha_y-r\alpha_{y'})} \leq \mnorm{s(p_i\alpha_y-q_i\alpha_{x_i})} + \mnorm{r(p_i\alpha_{y'}-q_i\alpha_{x_i'})} + \mnorm{q_i(s\alpha_{x_i}-r\alpha_{x_i'})} \leq \frac{6CP}{H}.
    \end{equation*}
    Since $p_1\neq p_2$, we conclude that $\norm{s\alpha_y-r\alpha_{y'}} \leq 6C/H$ by \cref{lem:normcalculus}. On the other hand,
    \begin{equation*}
        \abs{q_1(ry-sy')} \leq \abs{r(p_1x_1-q_1y)} + \abs{s(p_1x_1'-q_1y')} + \abs{p_1(rx_1-sx_1')} \leq 6CP^2,
    \end{equation*}
    so that $\abs{ry-sy'} \leq 6CP$. Therefore, $\dist{r,s}{y,y'} \leq 6C$.
\end{proof}

\subsection{Clusters of local relations} To prove the local structure theorem, we will need to show the existence of certain patterns of local relations in $\A$, which we call \emph{clusters}.

\begin{definition}[Cluster]
    \label{def:cluster}
    Let $C>0$. A \emph{cluster of diameter at most $C$} is a subset $S\subset A\times \P$ such that $\dist{r,s}{x,x'} \leq C$ for all $(x, r), (x', s) \in S$.
\end{definition}

A cluster $S$ of diameter $\leq C$ has size ${|S| \leq (2C+1)|\P|}$. This is an immediate consequence of the following basic observation.

\begin{lemma}
    \label{lem:basic}
    Let $C>0$. For any $x\in A$ and $p,q\in \P$, there are at most $2C+1$ elements $y\in A$ such that $\dist{p,q}{x,y}\leq C$.
\end{lemma}
\begin{proof}
    Fix $x,p,q$. If $\dist{p,q}{x,y} \leq C$ for some $y\in A$, then $\abs{px/q-y} \leq C$. Since $A$ is $1$-separated, this interval contains at most $2C + 1$ points of $A$.
\end{proof}

We now give examples of clusters.
\begin{example}
    \label{ex:cluster}
    Let $z\in [PY, 4PY]$ and $\beta \in \R/\Z$. Let $S$ be the set of all pairs $(x,p)\in A\times \P$ such that $\abs{px - z} \leq CP$ and $\norm{p\beta - \alpha_x} \leq C(PH)^{-1}$. Then $S$ is a cluster of diameter at most $4C$. Indeed, for any two pairs $(x,r), (x',s)\in S$, we have
    \begin{equation*}
        \abs{rx - sx'} \leq \abs{rx - z} + \abs{sx' - z} \leq 2CP
    \end{equation*}
    and
    \begin{equation*}
        \norm{r\alpha_{x'} - s\alpha_x} \leq \norm{r(\alpha_{x'} - s\beta)} + \norm{s(r\beta - \alpha_x)} \leq \frac{4C}{H},
    \end{equation*}
    which implies that $\dist{r,s}{x,x'} \leq 4C$.
\end{example}

The next lemma shows that any cluster involving at least two distinct primes is of the form given in \cref{ex:cluster}, up to constants. Moreover, the approximation $p\beta \approx \alpha_x \pmod 1$ can be made exact for one pair $(x,p)$ of our choice in the cluster.

\begin{lemma}[{\cite[Lemmas~2.1~and~2.2]{walsh1}}]
    \label{lem:clusterlift}
    Let $0<C\leq \tfrac{H}{30P}$. Let $S\subset A\times \P$ be a cluster of diameter $\leq C$ and let $(x,p)\in S$. Suppose that $S$ is not entirely contained in $A\times \{p\}$. Then, there exists $\beta \in \R/\Z$ such that
    \begin{enumerate}[label=(\roman*),ref=\roman*]
        \item \label{item:exactlift} $\norm{p\beta - \alpha_x} = 0$, and
        \item \label{item:clusterlift} for all $(x',r)\in S$, we have $\norm{r \beta - \alpha_{x'}} \leq {C}{(Hp)^{-1}}$.
    \end{enumerate}
\end{lemma}

\begin{proof}
    By assumption, there exists $(y,q)\in S$ with $q \neq p$. By definition of a cluster, we have $\norm{q\alpha_x - p \alpha_y} \leq C/H$. Since the linear combinations $qm - pn$ with $m,n\in \Z$ attain all integer values, we can find representatives $\widetilde{\alpha}_x, \widetilde{\alpha}_y \in \R$ of $\alpha_x, \alpha_y \in \R/\Z$ such that $\abs{q \widetilde{\alpha}_x - p \widetilde{\alpha}_y} \leq C/H$.

    We set $\beta := \widetilde{\alpha}_x / p \pmod 1$; this choice clearly verifies \cref{item:exactlift}. Moreover,
    \begin{equation}
        \label{eq:usefulbeta}
        \norm{q\beta - \widetilde{\alpha}_y} = \norm{\frac{q\widetilde{\alpha}_x - p \widetilde{\alpha}_y}{p}} \leq \abs{\frac{q\widetilde{\alpha}_x - p \widetilde{\alpha}_y}{p}} \leq \frac{C}{Hp}.
    \end{equation}

    Now, let $(x', r) \in S$ be arbitrary. On the one hand, applying the cluster property to $(x', r)$ and $(x, p)$, we get
    \begin{equation*}
        \norm{p(r\beta - \alpha_{x'})} = \norm{r \alpha_x - p \alpha_{x'}} \leq \frac{C}{H}.
    \end{equation*}
    On the other hand, using \cref{eq:usefulbeta} and the cluster property applied to $(x', r)$ and $(y, q)$, we have
    \begin{equation*}
        \norm{q(r\beta - \alpha_{x'})} \leq \norm{r(q\beta - \alpha_{y})} + \norm{r\alpha_{y} - q\alpha_{x'}} \leq \frac{3C}{H}.
    \end{equation*}
    By \cref{lem:normcalculus}, we conclude that $\norm{r\beta - \alpha_{x'}} \leq C(Hp)^{-1}$, as desired.
\end{proof}

In \cref{lem:clusterdisjointness}, we show that two clusters can be combined into a larger cluster provided their intersection involves at least two distinct primes.

\begin{definition}
    \label{def:disjointclusters}
    Two sets $S_1,S_2\subset A\times \P$ are called \emph{almost disjoint} if
    \begin{equation*}
        S_1\cap S_2 \subset A\times \{p\}
    \end{equation*}
    for some $p\in \P$. Otherwise, we say that $S_1, S_2$ \emph{strongly intersect}.
\end{definition}

\begin{lemma}
    \label{lem:clusterdisjointness}
    Let $k\in \N$ and $0 < C \leq \tfrac{4^{-k}H}{30P^2}$. Let $S_1, \ldots, S_k$ be clusters of diameter $\leq C$, where $S_i$~and~$S_{i+1}$ strongly intersect for all $1\leq i<k$. Then, $\bigcup_{i=1}^k S_i$ is a cluster of diameter $\leq 4^{k-1}C$.\footnote{The precise constant is not important as we will only use \cref{lem:clusterdisjointness} with bounded values of $k$.}
\end{lemma}
\begin{proof}
    It suffices to prove the case $k=2$, as the general case follows by induction.

    Consider arbitrary elements $(x,r) \in S_1$ and $(x',s) \in S_2$. Since $S_1$ and $S_2$ strongly intersect, $S_1 \cap S_2$ contains two elements $(y_1, q_1)$ and $(y_2, q_2)$ with $q_1 \neq q_2$. By \cref{def:cluster}, we have ${\dist{r,q_i}{x,y_i} \leq C}$ and $\dist{q_i,s}{y_i,x'} \leq C$ for $i=1,2$. By \cref{lem:triangle}, we conclude that $\dist{r,s}{x,x'} \leq 4C$.
\end{proof}

\section{Covering lemmas for clusters}
\label{sec:covering}

In this section, we prove several lemmas on global arrangements of clusters. We begin with a technical lemma that bounds the size of the overlaps for a family of almost disjoint clusters.

\begin{lemma}
    \label{lem:packing}
    Let $C\geq 1$ and $\eta > 0$. Let $(F_i)_{i\in I}$ be a collection of pairwise almost disjoint clusters of diameter $\leq C$ and size~$\geq \eta |\P|$. For each $i$, let $\overline{F_i}$ be a cluster of diameter $\leq 20C$ containing $F_i$. If~$J\subset I\times I$ is the set of all pairs $(i,j)$ such that $\overline{F_i}$ and $\overline{F_j}$ are almost disjoint, then
    \begin{equation*}
        \sum_{(i,j)\in J} \abs{\overline{F_i} \cap \overline{F_j}} \ll C^9 \eta^{-5} |A|.
    \end{equation*}
\end{lemma}

\begin{proof}
    For every cluster $S$ of diameter $\leq C$, the set $\{px : (x,p)\in S\}$ is contained in an interval of length $\leq CP$. For $z\in [PY, 4PY]$, define
    \begin{equation*}
        V_{z} := \big\{(x,p)\in A\times \P \,:\, |px-z| \leq 100CP\big\}
    \end{equation*}
    and consider the multiplicity function
    \begin{equation*}
        m(z) := \sum_{i\in I} \ind{F_i \subset V_z}.
    \end{equation*}

    We first prove a pointwise upper bound for $m(z)$. By \cref{lem:basic}, for every $i\in I$ and $p\in \P$,
    \begin{equation}
        \label{eq:multiplicitiesp}
        \abs{\{x\in A\,:\, (x,p)\in {F_i}\}} \leq \abs{\{x\in A\,:\, (x,p)\in \overline{F_i}\}} \ll C.
    \end{equation}
    In particular, each $F_i$ admits a subset $F_i'$ of size $\gg |F_i|/C$ such that no two elements of $F_i'$ share the same prime. Let $\T_i$ be the collection of all two-element sets $\{u,v\}$ with $u, v\in F_i'$, $u\neq v$. For all $i\neq j$, since $F_i$ and $F_j$ are almost disjoint, we have $|F_i'\cap F_j'|\leq 1$, and thus $\T_i \cap \T_j = \emptyset$.

    Note that $\abs{V_z} \ll C|\P|$ as $A$ is $1$-separated. Since the sets $(\T_i)_{i\in I}$ are pairwise disjoint, we have
    \begin{equation*}
        \sum_{\substack{i\in I\\ F_i \subset V_z}} |\T_i| \leq \binom{|V_z|}{2} \ll C^2|\P|^2.
    \end{equation*}
    Recalling that $|\T_i| \gg |F_i'|^2 \gg C^{-2} |F_i|^2$ and $|F_i| \geq \eta |\P|$, we get
    \begin{equation}
        \label{eq:Tibusiness}
        m(z) \ll C^4\eta^{-2}.
    \end{equation}

    Next, we bound the measure of the support of $m(z)$. For each $i\in I$, fix some ${z_i \in \{px : (x,p)\in F_i\}}$ and let $I'$ be a maximal subset of $I$ such that the points $(z_i)_{i\in I'}$ are $(100CP)$-separated. Then, the clusters $(F_i)_{i\in I'}$ are pairwise disjoint, so that
    \begin{equation*}
        |I'| \eta|\P| \leq \sum_{i\in I'} |F_i| \leq |A||\P|.
    \end{equation*}
    This yields $|I'| \leq \eta^{-1} |A|$. The support of $m(z)$ is clearly contained in the union of the intervals $[z_i - 200CP, z_i + 200CP]$ for $i\in I'$. Hence, the support of $m(z)$ has measure at most
    \begin{equation}
        \label{eq:measuremz}
        \leq 400C P \cdot |I'| \ll C \eta^{-1} |A| P.
    \end{equation}

    For $(i,j)\in J$, since the clusters $\overline{F_i}$ and~$\overline{F_j}$ are almost disjoint, we have $\abs{\overline{F_i} \cap \overline{F_j}} \ll C$ by \cref{eq:multiplicitiesp}. In addition, whenever $\overline{F_i} \cap \overline{F_j} \neq \emptyset$, the set $\{px : (p,x)\in \overline{F_i} \cup \overline{F_j}\}$ is contained in an interval of length~$\leq 40CP$. Therefore, for all $(i,j)\in J$,
    \begin{equation*}
        \abs{\overline{F_i} \cap \overline{F_j}} \ll \frac{1}{P}\int_{PY}^{4PY} \ind{\overline{F_i} \cup \overline{F_j} \subset V_z} \, dz \leq \frac{1}{P} \int_{PY}^{4PY} \ind{{F_i} \subset V_z} \ind{{F_j} \subset V_z} \, dz.
    \end{equation*}
    Summing over all $(i,j)\in J$, we get
    \begin{equation*}
        \sum_{(i,j)\in J} \abs{\overline{F_i} \cap \overline{F_j}} \ll \frac{1}{P} \int_{PY}^{4PY} m(z)^{2} \,dz.
    \end{equation*}
    The claimed estimate now follows from \cref{eq:Tibusiness} and \cref{eq:measuremz}.
\end{proof}

With the previous technical estimate in hand, we can prove a covering result reminiscent of the Vitali covering lemma in metric spaces: \Cref{lem:covering} allows us to replace an arbitrary collection of clusters with a more manageable collection while only losing a constant factor in the diameter.

\begin{lemma}[Efficient covering]
    \label{lem:covering}
    Let $1\leq C \leq |\P|^{1/4}$ and $\eta \geq |\P|^{-1/4}$. Let $\U\subset A\times \P$ be the union of all clusters of diameter~$\leq C$ and size $\geq \eta |\P|$.

    There exists a collection $(K_i)_{i\in I}$ of clusters of diameter $\leq 16C$ with the following properties.
    \begin{enumerate}[label=(\roman*),ref=\roman*]
        \item \label{item:cover1} For every cluster $S$ of diameter $\leq C$ and size $\geq \eta |\P|$, and every $(x,p)\in S$, there is a cluster~$K_i$ such that $(x,p)\in K_i$ and $|K_i| \geq |S|$.
        \item \label{item:cover2} $\sum_{i\in I}|K_i| \ll |A| |\P| + C^9 \eta^{-5} |A|$.
    \end{enumerate}
\end{lemma}

\begin{proof}
    If $\U = \emptyset$, the result is trivial. Otherwise, we proceed as follows.

    We start by constructing an auxiliary collection of clusters $(F_i)_{i\in I}$ using a greedy algorithm. Let~$F_1$ be a cluster of diameter $\leq C$ of maximal size. If $F_1, \ldots, F_{i-1}$ have been chosen, let $F_i$ be a cluster of maximal cardinality among all clusters of diameter $\leq C$ and size $\geq \eta |\P|$ that are almost disjoint from each of $F_1, \ldots, F_{i-1}$. This algorithm clearly terminates with a maximal collection $(F_i)_{i\in I}$ of pairwise almost disjoint clusters of diameter $\leq C$ and size $\geq \eta |\P|$, where $I\subset \N$.

    These clusters may not cover all of $\U$, so we shall expand them. For each $(x,p)\in \U$, let $S_{(x,p)}$ be a cluster of diameter $\leq C$ containing $(x,p)$ of maximal size, and define $i(x,p)$ to be the smallest $i$ such that $S_{(x,p)}$ and $F_i$ strongly intersect. Note that $|S_{(x,p)}| \leq |F_{i(x,p)}|$ by the maximality assumption in the greedy procedure.

    For every $i\in I$, define
    \begin{equation*}
        K_i := F_i \cup \big\{(x,p)\in \U \, :\, i(x,p) = i\big\}.
    \end{equation*}
    We claim that $K_i$ is a cluster of diameter $\leq 16C$. To see this, consider arbitrary elements $(x,r), (x',s)\in K_i$. By definition of $K_i$, there are clusters $S, S'$ of diameter $\leq C$ (possibly equal to $F_i$) such that $(x,r)\in S$, $(x',s)\in S'$, and both $S$ and $S'$ strongly intersect $F_i$. By \cref{lem:clusterdisjointness}, $S\cup F_i \cup S'$ is a cluster of diameter $\leq 16C$, so that $\dist{r,s}{x,x'} \leq 16C$.

    Property \cref{item:cover1} now follows: if $(x,p)\in \U$ lies in a cluster $S$ of diameter $\leq C$, then $(x,p)\in K_{i(x,p)}$ and
    \begin{equation*}
        |S|\leq |S_{(x,p)}| \leq |F_{i(x,p)}|\leq |K_{i(x,p)}|.
    \end{equation*}

    We turn to property \cref{item:cover2}. By definition of $K_i$ and truncated inclusion-exclusion (i.e.~the Bonferroni inequalities), we have
    \begin{equation*}
        \sum_{i\in I} |K_i| \leq \sum_{i\in I} |F_i| + |\U| \ll \bigg\lvert\bigcup_{i\in I}F_i\bigg\rvert + \sum_{\substack{i,j\in I\\ i\neq j}} |F_i\cap F_j|  + |\U|.
    \end{equation*}
    The middle term of the right-hand side is $\ll C^9 \eta^{-5} |A|$ by \cref{lem:packing} (with $\overline{F_i} = F_i$). The other terms are $\ll |A||\P|$, so we obtain \cref{item:cover2}.
\end{proof}

Finally, we establish a result that can be viewed as a strengthening of \cref{lem:clusterdisjointness}. While both results show that a union of clusters is again a cluster of slightly larger diameter, the conditions under which this holds are now significantly relaxed. Instead of requiring the clusters to strongly intersect, we consider families of clusters containing a given pair $(x,p)$, even if those clusters share no other points. The conclusion is not guaranteed to hold for every pair $(x,p)$, but \cref{lem:clusteruniqueness} shows that it does hold for almost all pairs.

\begin{definition}
    \label{def:regularity}
    Let $C\geq 1$ and $\eta > 0$. A pair $(x,p)\in A\times \P$ is called \emph{$(C,\eta)$-regular} if the union of all clusters of diameter $\leq C$ and size $\geq \eta |\P|$ containing $(x, p)$ is either empty, or itself a cluster of diameter $\leq 64C$.
\end{definition}

\begin{lemma}[Regularity]
    \label{lem:clusteruniqueness}
    Let $1\leq C \leq |\P|^{1/4}$ and $\eta \geq |\P|^{-1/4}$. Then, all but $\ll C^9 \eta^{-5} |A|$ pairs $(x,p)\in A\times \P$ are $(C,\eta)$-regular.
\end{lemma}

\begin{proof}
    Let $\U$ be the union of all clusters of diameter $\leq C$ and size $\geq \eta |\P|$. Let $(F_i)_{i\in I}$ be a maximal collection of pairwise almost disjoint clusters of diameter $\leq C$ and size $\geq \eta |\P|$.

    For each $i\in I$, define $\overline{F_i}$ to be the union of $F_i$ and all clusters of diameter $\leq C$ that strongly intersect $F_i$. By \cref{lem:clusterdisjointness}, each $\overline{F_i}$ is a cluster of diameter~${\leq 16C}$ (as in the proof of \cref{lem:covering}).

    Consider a pair $(x,p)\in \U$ that is not $(C,\eta)$-regular. Then, there are clusters $S_1,S_2$ of diameter~$\leq C$ and size $\geq \eta |\P|$, both containing $(x,p)$, such that $S_1\cup S_2$ is not a cluster of diameter $\leq 64C$. By definition of $(F_i)$, there exist indices $j_1,j_2\in I$ such that $S_i$ strongly intersects $F_{j_i}$ for each $i=1,2$. Hence, $S_i \subset \overline{F_{j_i}}$. Observe that $\overline{F_{j_1}}$ and $\overline{F_{j_2}}$ must be almost disjoint, as otherwise their union (and hence also $S_1\cup S_2$) would be a cluster of diameter $\leq 64C$ by \cref{lem:clusterdisjointness}. We have thus shown that for every pair $(x,p)$ that is not $(C,\eta)$-regular, there are $j_1, j_2$ such that $(x,p)\in \overline{F_{j_1}}\cap \overline{F_{j_2}}$, where $\overline{F_{j_1}}$ and $\overline{F_{j_2}}$ are almost disjoint. The conclusion now follows from \cref{lem:packing}.
\end{proof}
\section{Proof of the local structure theorem}
\label{sec:structure}

The goal of this section is to prove \cref{thm:structure}, which describes the structure of configurations having many local relations relative to their size, and constructs lifts preserving a large proportion of these relations.

\subsection{Initial clusters}
In the next lemma, we construct many large disjoint clusters, which will later help us define the first set $A_0$ of the decomposition of $A$ given in \cref{thm:structure}. This lemma is a variant of \cite[Corollary~4.10]{walsh3} that only assumes the existence of many local relations \emph{relative to $|A|$}, and thus also applies to sparse sets $A$.

\begin{lemma}[Large clusters]
    \label{lem:firstset}
    Suppose that ${|\Q(\A)| \geq \delta |A| |\P|^2}$, where $\delta \geq \Cclusters |\P|^{-1/2}$ for some large absolute constant $\Cclusters$. Then, there exists a collection of $\gg \delta |A|$ disjoint clusters, each of diameter~$\leq 1/2$ and size $\gg \delta^4 |\P|$.
\end{lemma}

\begin{proof}
    For $z\in [PY, 4PY]$, let
    \begin{equation*}
        V_{z} := \big\{(x,p)\in A\times \P \,:\, |px-z| \leq \tfrac{1}{4}P\big\}
    \end{equation*}
    and
    \begin{equation*}
        Q_{z} := \big\{((x, p),(y,q))\in V_z^2 \,:\, (x,y,p,q)\in \Q(\A)\big\}.
    \end{equation*}
    We have the simple bound $|Q_{z}| \leq |V_z|^2 \leq |\P|^2$ as $A$ is $1$-separated.

    Suppose that the reverse inequality $|Q_z| \geq \eps |\P|^2$ holds for some $\eps \geq {10}/{|\P|}$. We claim that $V_z$ contains a cluster of diameter~$\leq 1/2$ and size $\gg \eps^{2} |\P|$. To see this, consider the graph $G_z$ with vertex set $V_z$ and edge set $E_z := Q_z \setminus \Delta$ where $\Delta = \{((x,p),(x,p))\,:\, (x,p)\in V_z\}$ is the diagonal. Since $|Q_z|\geq \eps |\P|^2$ and $|\Delta| = |V_z| \leq |\P|$, we have $|E_z| \gg \eps |\P|^2$ (as $\eps \geq {10}/{|\P|}$). By convexity of the function $n \mapsto \binom{n}{2}$, we get
    \begin{equation*}
        \frac{1}{|V_z|} \sum_{\substack{v_1, v_2\in V_z\\ v_1\neq v_2}}\sum_{u\in V_z} \ind{(u,v_1), (u,v_2)\in E_z} = \EE_{u\in V_z} \binom{\deg(u)}{2} \geq \binom{2|E_z|/|V_z|}{2} \gg \eps^2 |\P|^2.
    \end{equation*}
    Hence, there exist two distinct vertices $v_1, v_2\in V_z$ having $\gg \eps^2 |\P|$ common neighbours in $G_z$. Write~$v_1 = (y_1, q_1)$, $v_2 = (y_2, q_2)$. We know that $q_1\neq q_2$ as $v_1\neq v_2$ and $A$ is $1$-separated. We claim that the set $N_z$ of common neighbours of $v_1$ and $v_2$ is a cluster of diameter~${\leq 1/2}$. Indeed, if $(x_1, p_1), (x_2, p_2)\in N_z$, then ${(x_i, y_j, p_i, q_j)\in \Q(\A)}$ for all $i,j\in\{1,2\}$, and hence $\dist{p_1, p_2}{x_1, x_2} \leq 4/10$ by \cref{lem:triangle}.

    To finish the proof, it therefore suffices to find a $P$-separated set $Z\subset [PY, 4PY]$ of size $|Z| \gg \delta |A|$ such that $|Q_z| \gg \delta^2 |\P|^2$ for every $z\in Z$; the separation assumption guarantees that the sets $(V_z)_{z\in Z}$ will be pairwise disjoint.

    To achieve this, it is enough to show that
    \begin{equation}
        \label{eq:probtolb}
        \Pr{|Q_{\z}|\geq c\delta^2 |\P|^2} \gg \frac{\delta |A|}{Y}
    \end{equation}
    where $\z$ is a random variable uniformly distributed in $[PY, 4PY]$, and $c>0$ is an absolute constant. On the one hand, since ${|\Q(\A)| \geq \delta |A| |\P|^2}$ and ${|px-qy|\leq \tfrac{1}{10}P}$ whenever $(x,y,p,q)\in \Q(\A)$, we have
    \begin{equation*}
        \E{|Q_\z|} \geq \sum_{(x,y,p,q)\in \Q(\A)} \Pr{|px-\z|\leq \tfrac{1}{10}P} \gg \frac{\delta |A||\P|^2}{Y}.
    \end{equation*}
    On the other hand, for any $\eps>0$,
    \begin{equation*}
        \E{|Q_{\z}|} \leq |\P|^2 \Pr{|Q_{\z}|\geq \eps |\P|^2} + \E{ \abs{Q_{\z}} \ind{|Q_{\z}|\leq \eps |\P|^2}},
    \end{equation*}
    and
    \begin{equation*}
        \E{ \abs{Q_{\z}} \ind{|Q_{\z}|\leq\eps |\P|^2} } \leq \eps^{1/2}|\P| \E{ \abs{Q_{\z}}^{1/2} }\leq \eps^{1/2}|\P| \E{|V_\z|}\ll \eps^{1/2} \frac{|A| |\P|^2}{Y}.
    \end{equation*}
    Rearranging and choosing $\eps = c\delta^2$ for some small absolute constant $c>0$ yields \cref{eq:probtolb}.
\end{proof}

\newcommand{\tinyparallelogram}{\tikz[baseline=0.0ex, line width=0.08ex, line join=round, line cap=round]{\draw (1ex, 0) -- (0,0) -- (0.3ex, 0.8ex) -- (1.3ex, 0.8ex);

        \path (1.3ex, 0.8ex) -- coordinate[pos=0.5] (mid) (1ex, 0);
        \draw[red, line width=0.06ex]
        ([xshift=-0.15ex, yshift=0.15ex]mid) -- ([xshift=0.15ex, yshift=-0.15ex]mid)
        ([xshift=-0.15ex, yshift=-0.15ex]mid) -- ([xshift=0.15ex, yshift=0.15ex]mid);
    }}

\newcommand{\tinygraphicon}{\tikz[baseline=0.2ex, line width=0.08ex, line join=round, line cap=round]{\coordinate (x)  at (0.2ex, 1.2ex);
        \coordinate (x1) at (0.4ex, 0.0ex);
        \coordinate (x2) at (1.4ex, 0.8ex);
        \coordinate (y)  at (1.4ex, 1.2ex);
        \coordinate (y1) at (1.6ex, 0.0ex);
        \coordinate (y2) at (2.6ex, 0.8ex);

        \draw (x) -- (y);    \draw (x1) -- (y1);  \draw (x2) -- (y2);  \draw (x1) -- (x2);  \draw (x) -- (x1);

        \path (y1) -- coordinate[pos=0.5] (mid) (y2);
        \draw[red, line width=0.06ex]
        ([xshift=-0.15ex, yshift=0.15ex]mid) -- ([xshift=0.15ex, yshift=-0.15ex]mid)
        ([xshift=-0.15ex, yshift=-0.15ex]mid) -- ([xshift=0.15ex, yshift=0.15ex]mid);
    }}

\subsection{Rigidity} Exploiting the properties of $\dist{p,q}{x,y}$, we show that certain patterns of local relations are rare.

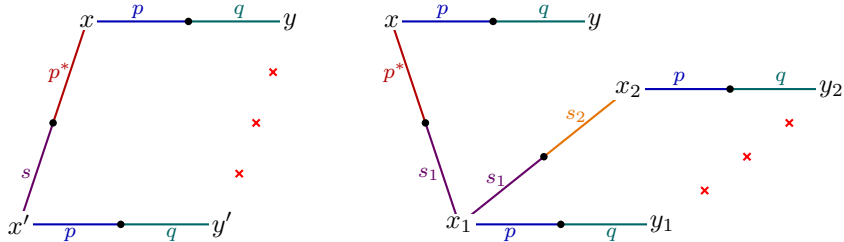
\begin{figure}[H]
    \centering
    \begin{tikzpicture}[scale=0.9, every node/.style={font=\small}]
        \colorlet{cp}{blue!70!black}       \colorlet{cq}{teal!80!black}       \colorlet{cvertA}{red!70!black}    \colorlet{cvertB}{violet!80!black} \colorlet{cvertC}{orange!90!black}

        \begin{scope}[xshift=-6.5cm]
            \coordinate (x)  at (1, 3);
            \coordinate (y)  at (4, 3);
            \coordinate (xp) at (0, 0);
            \coordinate (yp) at (3, 0);

            \draw[thick, cp] (x) -- ($(x)!0.5!(y)$);
            \draw[thick, cq] (y) -- ($(x)!0.5!(y)$);

            \draw[thick, cp] (xp) -- ($(xp)!0.5!(yp)$);
            \draw[thick, cq] (yp) -- ($(xp)!0.5!(yp)$);

            \draw[thick, cvertA] (x)  -- ($(x)!0.5!(xp)$);
            \draw[thick, cvertB] (xp) -- ($(x)!0.5!(xp)$);

            \foreach \i in {0.25, 0.5, 0.75} {
                    \path (y) -- coordinate[pos=\i] (mid) (yp);
                    \draw[red, thick]
                    ([xshift=-1.5pt, yshift=1.5pt]mid) -- ([xshift=1.5pt, yshift=-1.5pt]mid)
                    ([xshift=-1.5pt, yshift=-1.5pt]mid) -- ([xshift=1.5pt, yshift=1.5pt]mid);
                }

            \foreach \A/\B in {x/y, xp/yp, x/xp} {
                    \path let \p1=(\A), \p2=(\B) in
                    node[fill=black, circle, inner sep=1pt] at ($(\p1)!0.5!(\p2)$) {};
                }

            \node[fill=white, inner sep=1pt] at (x)  {\(x\)};
            \node[fill=white, inner sep=1pt] at (y)  {\(y\)};
            \node[fill=white, inner sep=1pt] at (xp) {\(x'\)};
            \node[fill=white, inner sep=1pt] at (yp) {\(y'\)};

            \node[text=cp, font=\scriptsize] at ($ (x)!0.25!(y) + (0, 0.15) $) {\(p\)};
            \node[text=cq, font=\scriptsize] at ($ (y)!0.25!(x) + (0, 0.15) $) {\(q\)};

            \node[text=cp, font=\scriptsize] at ($ (xp)!0.25!(yp) + (0, -0.15) $) {\(p\)};
            \node[text=cq, font=\scriptsize] at ($ (yp)!0.25!(xp) + (0, -0.15) $) {\(q\)};

            \node[text=cvertA, font=\scriptsize] at ($ (x)!0.25!(xp) + (-0.15, 0) $) {\(p^*\)};
            \node[text=cvertB, font=\scriptsize] at ($ (xp)!0.25!(x) + (-0.15, 0) $) {\(s\)};
        \end{scope}

        \begin{scope}
            \coordinate (x1) at (0, 0);
            \coordinate (y1) at (3, 0);
            \coordinate (x2) at (2.5, 2);
            \coordinate (y2) at (5.5, 2);
            \coordinate (x)  at (-1, 3);
            \coordinate (y)  at (2, 3);

            \draw[thick, cp] (x) -- ($(x)!0.5!(y)$);
            \draw[thick, cq] (y) -- ($(x)!0.5!(y)$);

            \draw[thick, cp] (x1) -- ($(x1)!0.5!(y1)$);
            \draw[thick, cq] (y1) -- ($(x1)!0.5!(y1)$);

            \draw[thick, cp] (x2) -- ($(x2)!0.5!(y2)$);
            \draw[thick, cq] (y2) -- ($(x2)!0.5!(y2)$);

            \draw[thick, cvertA] (x)  -- ($(x)!0.5!(x1)$);
            \draw[thick, cvertB] (x1) -- ($(x)!0.5!(x1)$);

            \draw[thick, cvertB] (x1) -- ($(x1)!0.5!(x2)$);
            \draw[thick, cvertC] (x2) -- ($(x1)!0.5!(x2)$);

            \foreach \i in {0.25, 0.5, 0.75} {
                    \path (y1) -- coordinate[pos=\i] (mid) (y2);
                    \draw[red, thick]
                    ([xshift=-1.5pt, yshift=1.5pt]mid) -- ([xshift=1.5pt, yshift=-1.5pt]mid)
                    ([xshift=-1.5pt, yshift=-1.5pt]mid) -- ([xshift=1.5pt, yshift=1.5pt]mid);
                }

            \foreach \A/\B in {x/y, x1/y1, x2/y2, x/x1, x1/x2} {
                    \path let \p1=(\A), \p2=(\B) in
                    node[fill=black, circle, inner sep=1pt] at ($(\p1)!0.5!(\p2)$) {};
                }

            \node[fill=white, inner sep=1pt] at (x1) {\(x_1\)};
            \node[fill=white, inner sep=1pt] at (y1) {\(y_1\)};
            \node[fill=white, inner sep=1pt] at (x2) {\(x_2\)};
            \node[fill=white, inner sep=1pt] at (y2) {\(y_2\)};
            \node[fill=white, inner sep=1pt] at (x)  {\(x\)};
            \node[fill=white, inner sep=1pt] at (y)  {\(y\)};

            \node[text=cp, font=\scriptsize] at ($ (x)!0.25!(y) + (0, 0.15) $) {\(p\)};
            \node[text=cq, font=\scriptsize] at ($ (y)!0.25!(x) + (0, 0.15) $) {\(q\)};

            \node[text=cp, font=\scriptsize] at ($ (x1)!0.25!(y1) + (0, -0.15) $) {\(p\)};
            \node[text=cq, font=\scriptsize] at ($ (y1)!0.25!(x1) + (0, -0.15) $) {\(q\)};

            \node[text=cp, font=\scriptsize] at ($ (x2)!0.25!(y2) + (0, 0.15) $) {\(p\)};
            \node[text=cq, font=\scriptsize] at ($ (y2)!0.25!(x2) + (0, 0.15) $) {\(q\)};

            \node[text=cvertA, font=\scriptsize] at ($ (x)!0.25!(x1) + (-0.2, 0) $) {\(p^*\)};
            \node[text=cvertB, font=\scriptsize] at ($ (x1)!0.25!(x) + (-0.2, 0) $) {\(s_1\)};

            \node[text=cvertB, font=\scriptsize] at ($ (x1)!0.28!(x2) + (-0.15, 0.1) $) {\(s_1\)};
            \node[text=cvertC, font=\scriptsize] at ($ (x2)!0.25!(x1) + (-0.15, 0.1) $) {\(s_2\)};

        \end{scope}
    \end{tikzpicture}
    \caption{Patterns considered in \cref{def:incomplete} (where $x,y,p,q$ and $p^*$ are fixed).}
    \label{fig:incomplete}
\end{figure}

\begin{definition}
    \label{def:incomplete}
    Let $(x,y,p,q)\in \Q(\A)$ and $p^*\in \P$. For $C>0$, define $\None(x,y,p,q,p^*;C)$ to be the number of triples $(x',y',s)\in A^2\times \P$ such that
    \begin{equation*}
        \dist{p,q}{x',y'}, \dist{p^*\!,s}{x,x'} \leq C\quad \text{and}\quad \dist{p^*\!,s}{y,y'} > 6C.
    \end{equation*}
    Also, define $\Ntwo(x,p,q,p^*;C)$ to be the number of tuples ${(x_1, x_2, y_1, y_2,s_1,s_2)\in A^{4}\times \P^{2}}$ such that
    \begin{equation*}
        \dist{p,q}{x_1,y_1}, \dist{p,q}{x_2,y_2}, \dist{p^*\!,s_1}{x,x_1}, \dist{s_1,s_2}{x_1,x_2} \leq C\quad\text{and}\quad \dist{s_1,s_2}{y_1,y_2}> 6C.
    \end{equation*}
\end{definition}

We have the trivial bound $\None(x,y,p,q,p^*;C) \ll_C |\P|$, as fixing $s$ leaves only $O_C(1)$ possibilities for $x'$ and $y'$ (by \cref{lem:basic}). Similarly, we have the trivial bound $\Ntwo(x,p,q,p^*;C) \ll_C |\P|^2$. The next lemma shows that, \emph{on average} over $(x,y,p,q)\in \Q(\A)$, we can gain a factor of $|\P|$ in each case.

\begin{lemma}
    \label{lem:badquad}
    Let $1\leq C \leq c_1 |\P|^{c_1}$ for some small absolute constant $c_1>0$. Fix $p^*\in \P$. Then
    \begin{equation*}
        \sum_{(x,y,p,q)\in \Q(\A)} \None(x,y,p,q,p^*;C) \ll C^3|A||\P|^2
    \end{equation*}
    and
    \begin{equation*}
        \sum_{(x,y,p,q)\in \Q(\A)} \Ntwo(x,p,q,p^*;C) \ll C^5 |A||\P|^3.
    \end{equation*}
\end{lemma}
\begin{proof}
    The first part is a relative version of \cite[Lemma~4.8]{walsh3}. To shorten notation, write
    \begin{equation*}
        \ind{\tinyparallelogram}(x,x',y,y';p,q,r,s) := \ind{\{\dist{p,q}{x,y},\dist{p,q}{x',y'},\dist{r,s}{x,x'}\leq C\}} \ind{\dist{r,s}{y,y'}>6C},
    \end{equation*}
    i.e.~the indicator function of the pattern represented in \cref{fig:incomplete} (left side). Swapping the order of summation, we have
    \begin{equation*}
        \sum_{(x,y,p,q)\in \Q(\A)} \None(x,y,p,q,p^*;C) \leq \sum_{(y,y',s)\in A^2\times \P} \sum_{(x,x',p,q)\in A^2\times \P^2} \ind{\tinyparallelogram}(x,x',y,y';p,q,p^*,s).
    \end{equation*}
    We may restrict the outer sum to the triples $(y,y',s)$ satisfying $|p^*y-sy'| \leq 6C$, as otherwise $\ind{\tinyparallelogram}(x,x',y,y',p,q,p^*,s)=0$ (regardless of $x,x',p,q$). Thus, there are $\ll C|A||\P|$ choices for $y, y'$ and~$s$. Now, the key observation is that, by \cref{lem:parallelogram}, there is at most one prime $p = p(y,y',s)\in \P$ giving a non-zero contribution to the inner sum. There are $|\P|$ choices for $q$. Once $y,y',s,p$ and $q$ are fixed, there are $\ll C$ choices for each of $x$ and $x'$, by \cref{lem:basic}. Hence,
    \begin{equation*}
        \sum_{(x,y,p,q)\in \Q(\A)} \None(x,y,p,q,p^*;C) \ll C|A||\P|\cdot |\P| \cdot C^2 = C^3|A||\P|^2
    \end{equation*}
    as claimed.

    Similarly, the $\Ntwo$ sum is bounded by
    \begin{equation*}
        \leq \sum_{(y_1,y_2,s_1,s_2)\in A^2\times \P^2} \sum_{(x_1,x_2,p,q)\in A^2\times \P^2} \ind{\tinyparallelogram}(x_1,x_2,y_1,y_2;p,q,s_1,s_2) \sum_{x,y\in A} \ind{\{\dist{p^*\!,s_1}{x,x_1},\dist{p,q}{x,y}\leq C\}}.
    \end{equation*}
    Here, the inner sum over $x,y$ is always $\ll C^2$ by \cref{lem:basic}. Thus,
    \begin{equation*}
        \sum_{(x,y,p,q)\in \Q(\A)} \Ntwo(x,p,q,p^*;C) \ll C^2 \! \sum_{(y_1,y_2,s_1,s_2)\in A^2\times \P^2} \sum_{(x_1,x_2,p,q)\in A^2\times \P^2} \!\!\!\ind{\tinyparallelogram}(x_1,x_2,y_1,y_2;p,q,s_1,s_2).
    \end{equation*}
    Once $s_1$ is fixed, we are left with the same expression as in the previous case, after relabelling. Since there are $|\P|$ choices for $s_1\in \P$, we obtain a final bound of $\ll C^2 |\P| \cdot C^3|A||\P|^2$, as required.
\end{proof}

\subsection{Replication of clusters}

The following lemma provides a mechanism to generate new clusters from existing ones, at the cost of multiplying the diameter by a fixed constant. This operation is possible under certain technical conditions, including bounds on the functions $\None$ and $\Ntwo$ introduced in \cref{def:incomplete}.

\begin{lemma}
    \label{lem:clustercopy}
    Let $1\leq C \leq c_1 |\P|^{c_1}$ for some small absolute constant $c_1>0$. Let $(x,y,p,q)\in \Q(\A)$ and $p^*\in \P$. Let $S$ be a cluster of diameter $\leq C$ and size $\geq |\P|^{1/2}$ containing~$(x, p^*)$.

    Suppose that, for all $(x', r)\in S$, there exists $y'\in A$ such that $\dist{p,q}{x',y'}\leq 1/10$. Furthermore, assume that
    \begin{equation}
        \label{eq:assumpN1N2}
        \None(x,y,p,q,p^*;C)\leq \tfrac{1}{100}|S| \quad\text{and}\quad \Ntwo(x,p,q,p^*;C)\leq \tfrac{1}{100}|S|^2.
    \end{equation}

    Then, there exists a cluster $S'$ of diameter $\leq 24C$ and size $\geq |S|/2$ containing $(y,p^*)$. Moreover, for any $(y',r)\in S'$, there exists $(x',r)\in S$ such that $\dist{p,q}{x',y'}\leq C$.
\end{lemma}

\begin{proof}
    Let $V$ be the set of all triples $(x',y',r)\in A^2\times \P$ such that $(x',r)\in S$ and $\dist{p,q}{x',y'}\leq 1/10$. Note that $|V|=|S|$ as, by assumption, every $(x',r)\in S$ corresponds to some unique $(x',y',r)\in V$.

    Let $G$ be the graph with vertex set $V$ where two distinct vertices $(x_1,y_1,r_1),(x_2,y_2,r_2)\in V$ are joined by an edge whenever $\dist{r_1,r_2}{y_1,y_2} \leq 6C$.

    We restrict to a highly connected subset of $V$ containing $(x,y,p^*)$ (which lies indeed in $V$) as follows. By the first bound in \cref{eq:assumpN1N2}, there are $\leq \tfrac{1}{100}|V|$ vertices $(x',y',s)\in V$ which are not neighbours of $(x,y,p^*)$ in $G$. In addition, by the second bound in \cref{eq:assumpN1N2}, there are $\leq \tfrac{1}{100}|V|^2$ pairs of vertices of $G$ that are not connected by an edge. Defining $V'$ to be the set of all vertices $(x',y',r)\in V$ which have $\geq \tfrac{3}{4}|V|$ neighbours in $G$, we see that $(x,y,p^*)\in V'$ and $|V'|\geq |V|/2$.

    We can now construct the required cluster. Let
    \begin{equation*}
        S' := \big\{(y',r)\in A\times \P \, :\, \exists x'\in A,(x',y',r)\in V'\big\}.
    \end{equation*}
    Observe that $(y,p^*)\in S'$ and $|S'| = |V'| \geq |S|/2$, as required. Moreover, for every $(y',r)\in S'$, there is $x'\in A$ such that $(x',r)\in S$ and $\dist{p,q}{x',y'}\leq 1/10$. Thus, it only remains to prove that $S'$ is a cluster of diameter $\leq 24C$.

    Let $(y_1, r_1), (y_2, r_2)\in S'$ and let $(x_1, y_1, r_1), (x_2, y_2, r_2)\in V'$ be the corresponding vertices of $G$. These two vertices each have $\geq \tfrac{3}{4}|V|$ neighbours in $G$. Hence, they have $\geq \tfrac{1}{2} |V|$ common neighbours.  We can find two of these common neighbours, say $(x_1', y_1', r_1')$ and $(x_2', y_2', r_2')$, such that ${r_1' \neq r_2'}$ (otherwise $S$ would contain $\geq \tfrac12 |\P|^{1/2}$ pairs with the same prime, contradicting \cref{lem:basic}). By definition of $G$ we have
    \begin{equation*}
        \dist{r_i,r_j'}{y_i, y_j} \leq 6C
    \end{equation*}
    for all $i,j\in \{1,2\}$. By \cref{lem:triangle}, this implies $\dist{r_1,r_2}{y_1,y_2} \leq 24C$, which concludes the proof.
\end{proof}

\subsection{Proof of the local structure theorem} We are now ready to prove the main result of this section, which we recall for convenience.

\structhm*

\begin{proof}[Proof of \cref{thm:structure}]
    The assumptions $\eps^{-\Clift L}\leq P$ and $\eps \delta^{-5} L \leq 1$ imply that $\delta \geq \eps \geq P^{-1/\Clift}$.

    We will use an iterative argument, working with clusters of increasing diameter. To this end, we set up a hierarchy of scales
    \begin{equation*}
        \CC=\{500^\ell : 0\leq \ell\leq L+1\}. \end{equation*}
    For each $C\in \CC$, we apply \cref{lem:covering} with the parameters $C$ and $\eta = |\P|^{-1/100}$ to obtain a collection $(K_{i;C})_{i\in I_C}$ of clusters of diameter $\leq 16C$ satisfying the two properties in that lemma. These auxiliary clusters will be used in the construction of the sets $A_\ell$.

    By \cref{lem:firstset}, there exist $\asymp \delta |A|$ disjoint clusters of diameter~$\leq 1/2$ and size $\gg \delta^4 |\P|$. By an averaging argument, there is a set $\P^* \subset \P$ of size $\gg \delta^5|\P|$ such that every $p\in \P^*$ appears in $\gg \delta^5 |A|$ of these clusters. That is, for each $p\in \P^*$, there are $\gg \delta^5 |A|$ elements $x\in A$ such that $(x,p)$ belongs to a cluster of diameter $\leq 1/2$ and size $\gg \delta^4|\P|$.

    We now select a prime $\mathbf{p^*}$ uniformly at random from $\P^*$. We will show that \cref{thm:structure} holds with positive probability with this prime as the lifting prime.

    To construct the required decomposition of $A$, we will define a subset $\RR \subset \Q(\A)$ of quadruples satisfying various technical properties, which we now describe.

    \begin{itemize}
        \item \textit{Regularity.} Let $A_{\mathrm{reg}}$ be the random set of $x\in A$ such that $(x,\mathbf{p^*})$ is $(C, |\P|^{-1/100})$-regular for all $C \in \CC$. By \cref{lem:clusteruniqueness},
              \begin{equation*}
                  \EE_{\mathbf{p^*}\in \P^*} \abs{A \setminus A_{\mathrm{reg}}} \ll \frac{1}{|\P^*|} \sum_{C\in \CC} C^9 |\P|^{1/20} |A| \ll \frac{e^{O(L)}}{\delta^5 |\P|^{1-1/20}} |A|.
              \end{equation*}
              By Markov's inequality and our choice of parameters (choosing $\Clift$ to be sufficiently large), with probability at least $99\%$, we have
              \begin{equation}
                  \label{eq:boundAreg}
                  \abs{A \setminus A_{\mathrm{reg}}} \ll \frac{|A|}{|\P|^{1-1/10}}.
              \end{equation}

        \item \textit{Rigidity.} By \cref{lem:badquad}, for every $p_0\in \P$, the number of quadruples $(x,y,p,q)\in \Q(\A)$ satisfying ${\None(x,y,p,q,{p_0};C) > |\P|^{19/20}}$~or~$\Ntwo(x,p,q,{p_0};C) > |\P|^{39/20}$ for some $C \in \CC$ is bounded by
              \begin{equation}
                  \label{eq:badquadbound}
                  \ll \sum_{C\in \CC} \bigg( \frac{C^3 |A| |\P|^2}{|\P|^{19/20}} + \frac{C^5 |A| |\P|^3}{|\P|^{39/20}} \bigg) \ll e^{O(L)} |A| |\P|^{1+1/20} \ll |A| |\P|^{1+1/10}.
              \end{equation}

        \item \textit{Connectivity.} Let $\Q_{\mathrm{bad}}$ be the random set of quadruples $(x,y,p,q)\in \Q(\A)$ such that one of the previously defined clusters $K_{i;C}$ has the following properties: $K_{i;C}$ contains $(x,\mathbf{p^*})$ and there are fewer than $\eps^2 |K_{i;C}|$ elements $(x',r)\in K_{i;C}$ for which $\dist{p,q}{x',y'}\leq 1/10$ for some $y'\in A$.

              By the union bound, we have
              \begin{equation*}
                  \EE_{\mathbf{p^*}\in \P^*} \abs{\Q_{\mathrm{bad}}} \leq \frac{1}{|\P^*|} \sum_{C\in \CC} \sum_{i\in I_C} \sum_{p_0\in \P^*} \sum_{\substack{(x,y,p,q)\in \Q(\A) \\ (x,p_0)\in K_{i;C}}} I(p,q;i,C)
              \end{equation*}
              where $I(p,q;i,C)$ is the indicator function that the number of triples $(x',y',r)\in A^2\times \P$ with ${(x',y',p,q)\in \Q(\A)}$ and $(x',r)\in K_{i;C}$ is less than $\eps^2 |K_{i;C}|$. For any $p,q\in \P$, by definition of $I(p,q;i,C)$, we have
              \begin{equation*}
                  I(p,q;i,C) \sum_{\substack{x,y\in A\\ (x,y,p,q)\in \Q(\A)}}\sum_{\substack{p_0\in \P \\ (x,p_0)\in K_{i;C}}} 1 \leq \eps^2 |K_{i;C}|.
              \end{equation*}
              Inserting this bound into the previous expression, we obtain
              \begin{equation*}
                  \EE_{\mathbf{p^*}\in \P^*} \abs{\Q_{\mathrm{bad}}} \leq \frac{1}{|\P^*|} \sum_{C\in \CC} \sum_{i\in I_C} \sum_{p,q\in \P} \eps^2|K_{i;C}|  \ll \eps^2 \delta^{-5} |\P| \sum_{C\in \CC} \sum_{i\in I_C} |K_{i;C}|.
              \end{equation*}
              Since $(K_{i;C})_{i\in I_C}$ satisfies the second property in \cref{lem:covering} with $\eta = |\P|^{-1/100}$, we have
              \begin{equation*}
                  \sum_{i\in I_C} |K_{i;C}| \ll |A| |\P| + C^9 |\P|^{1/20} |A| .
              \end{equation*}
              We conclude that
              \begin{equation*}
                  \EE_{\mathbf{p^*}\in \P^*} \abs{\Q_{\mathrm{bad}}} \ll \eps^2 \delta^{-5} \big( L |A| |\P|^2 +  e^{O(L)} |A| |\P|^{1+1/20} \big).
              \end{equation*}
              By Markov's inequality and our choice of parameters (recall that $\eps \delta^{-5} L \leq 1$), with probability at least $99\%$, we have
              \begin{equation}
                  \label{eq:boundQbad}
                  \abs{\Q_{\mathrm{bad}}} \ll \eps |A| |\P|^2.
              \end{equation}
    \end{itemize}

    We can now define the random set $\RR$ to be the collection of all $(x,y,p,q)\in  \Q(\A)$ satisfying the following three properties:
    \begin{enumerate}[label=(\alph*),ref=\alph*]
        \item \label{item:robust1} $x,y\in A_{\mathrm{reg}}$.
        \item \label{item:robust2} For all $C\in \CC$,
              \begin{equation*}
                  \begin{cases}
                      \None(x,y,p,q,\mathbf{p^*};C) \leq |\P|^{19/20} \\
                      \Ntwo(x,p,q,\mathbf{p^*};C) \leq |\P|^{39/20}
                  \end{cases}
                  \text{and}\quad
                  \begin{cases}
                      \None(y,x,q,p,\mathbf{p^*};C) \leq |\P|^{19/20} \\
                      \Ntwo(y,q,p,\mathbf{p^*};C) \leq |\P|^{39/20}.
                  \end{cases}
              \end{equation*}
        \item \label{item:robust3} Neither $(x,y,p,q)$ nor $(y,x,q,p)$ belongs to $\Q_{\mathrm{bad}}$.
    \end{enumerate}
    By \cref{eq:boundAreg,eq:badquadbound,eq:boundQbad}, with positive probability, we have $\abs{A \setminus A_{\mathrm{reg}}} \ll {|A|}{|\P|^{-1+1/10}}$ and
    \begin{equation}
        \label{eq:boundRRR}
        \abs{\Q(\A) \setminus \RR} \ll \eps |A| |\P|^2.
    \end{equation}
    For the rest of this proof, we fix a realisation $p^*\in \P^*$ of $\mathbf{p^*}$ for which this holds. By a slight abuse of notation, we will henceforth use $A_{\mathrm{reg}}$, $\Q_{\mathrm{bad}}$ and $\RR$ to refer to the deterministic sets obtained by evaluating these random sets at the fixed prime $p^*$.

    We turn to the construction of the decomposition $A = A_0 \sqcup \ldots \sqcup A_L \sqcup A_{L+1}$.

    By definition of $\P^*$ and \cref{eq:boundAreg}, there is a subset $A_0\subset A_{\mathrm{reg}}$ of size $\gg \delta^5 |A|$ such that, for each $x\in A_0$, the pair $(x,p^*)$ is contained in a cluster of diameter $\leq 1/2$ and size $\gg \delta^4|\P|$. In particular, property~\cref{item:struct1} of~\cref{thm:structure} holds.

    Suppose that $A_0, \ldots, A_{\ell-1}$ have been constructed, for some $1\leq \ell \leq L$. We define
    \begin{equation}
        \label{eq:definitionAell}
        A_{\ell} := \Big\{ y\in A_{\mathrm{reg}}\setminus \bigsqcup_{k<\ell}A_k\ :\ \exists (x,y,p,q)\in \RR\text{ with }x\in A_{\ell-1} \Big\}.
    \end{equation}Finally, we define $A_{L+1} := A\setminus \bigsqcup_{\ell\leq L} A_\ell$. Property \cref{item:struct2a} follows directly from the definition of the sets $A_{\ell}$, the symmetry of $\RR$ (i.e.~the fact that $(x,y,p,q)\in \RR$ if and only if $(y,x,q,p)\in \RR$) and \cref{eq:boundRRR}.

    We shall now iterate \cref{lem:clustercopy} to construct, for each $x\in A\setminus A_{L+1}$, a cluster containing $(x,p^*)$.

    \textbf{Claim 1.} \textit{Let $(x,y,p,q)\in \RR$, where $x\in A_k$ for some $k\leq L$. Suppose that $(x,p^*)$ is contained in a cluster~$S$ of diameter $\leq 500^{k}$ and size $\geq |\P|^{99/100}$. Then, there is a cluster $S'$ of diameter~$\leq 500^{k+1}$ and size~$\geq \tfrac{1}{2}\eps^2 |S|$ containing $(y,p^*)$.
    }

    \textit{Proof of Claim 1.} Let $C := 500^{k}$. By definition of the collection $(K_{i;C})_{i\in I_C}$, there exists $i\in I_C$ such that $(x,p^*)\in K_{i;C}$ and $|K_{i;C}|\geq |S|$ (see \cref{lem:covering}). Let $S_1$ be the set of all $(x',r)\in K_{i;C}$ such that $\dist{p,q}{x',y'}\leq 1/10$ for some $y'\in A$. By property~\cref{item:robust3} in the definition of $\RR$, we have $(x,y,p,q)\notin \Q_{\mathrm{bad}}$, which means that $|S_1| \geq \eps^2 |K_{i;C}| \geq \eps^2 |S|$.

    We are now in a position to apply \cref{lem:clustercopy} to $S_1$ (which is a cluster of diameter $\leq 16C$ as it is contained in $K_{i;C}$). The assumption \cref{eq:assumpN1N2} in \cref{lem:clustercopy} holds by property~\cref{item:robust2} in the definition of $\RR$ and our choice of parameters. Thus, \cref{lem:clustercopy} yields a cluster $S'$ of diameter $\leq 24\cdot 16 C \leq 500^{k+1}$ and size $\geq |S_1|/2 \geq \tfrac12 \eps^2 |S|$ containing $(y,p^*)$, as required. \hfill$\triangleleft$

    \textbf{Claim 2.} \textit{For every $k \leq L$ and $x\in A_{k}$, the pair $(x,p^*)$ is contained in a cluster $S$ of diameter~$\leq 500^{k}$ and size $\geq \max(|\P|^{99/100}, \,c_0(\tfrac{1}{2}\eps^2)^{k} \delta^4 |\P|)$ where $c_0>0$ is an absolute constant.}

    \textit{Proof of Claim 2.} This follows from Claim 1 by induction on $k$. For $k=0$, this is true by definition of $A_0$. For $1\leq k \leq L$, we know that each $y\in A_{k}$ is connected to some $x\in A_{k-1}$ by a quadruple $(x,y,p,q)\in \RR$. By the induction hypothesis, $(x,p^*)$ is contained in a cluster of diameter $\leq 500^{k-1}$ and size $\geq \,c_0(\tfrac{1}{2}\eps^2)^{k-1} \delta^4 |\P|$. This size is at least $|\P|^{99/100}$ by our assumption that $\eps^{-\Clift L}\leq P$ and $\delta^{-\Clift}\leq P$ for a sufficiently large constant $\Clift$. Applying Claim 1 completes the proof.\hfill$\triangleleft$

    We remark that any cluster $S$ as in Claim 2 contains at least one pair $(x',r)$ with $r \neq p^*$, by \cref{lem:basic} and the lower bound $|S| \geq |\P|^{99/100}$. This will be useful in the next step.

    Now, we define the frequencies $\beta_x$ for $x\in A\setminus A_{L+1}$. Let $k \leq L$ and $x\in A_{k}$. Let $S(x)$ be the union of all clusters of diameter $\leq 500^{k}$ and size $\geq |\P|^{99/100}$ containing $(x,p^*)$; note that $S(x)$ is non-empty by Claim 2. By construction, $A_{k} \subset A_{\mathrm{reg}}$, which implies that~$S(x)$ is a cluster of diameter $\leq 64\cdot 500^{k}$. We define $\beta_x$ to be the real number given by applying \cref{lem:clusterlift} to this cluster $S(x)$ and the pair $(x, p^*)\in S(x)$. The first property in \cref{lem:clusterlift} ensures that $\B := (A\setminus A_{L+1}, \beta_{\bullet}, Hp^*)$ is a lift of $\A$.

    It only remains to prove \cref{item:struct2b} of \cref{thm:structure}. In view of our bound \cref{eq:boundRRR} on $\abs{\Q(\A) \setminus \RR}$, it suffices to show that, for every $(x,y,p,q)\in \RR$ with $x\in A_k$ for some $k \leq L$, we have $(x,y,p,q)\in \Q(\B)$. Consider such a quadruple, and let $S$ be a cluster of diameter $\leq 500^{k}$ and size $\geq |\P|^{99/100}$ containing $(x,p^*)$ (which exists by Claim 2). Let $K_{i;C}$, $S_1$ and $S'$ be the clusters constructed in the proof of Claim 1. In particular, $(y,p^*)\in S'$. Let $(y',r)\in S'$ be any pair such that $r\neq p^*$. Recall that $S'$ was obtained from $S_1$ by applying \cref{lem:clustercopy}. By the last property in \cref{lem:clustercopy}, there exists $(x',r)\in S_1$ such that $\dist{p,q}{x',y'}\leq 500^{k}$, so that $\norm{q\alpha_{x'} - p\alpha_{y'}} \leq 500^k/H$. Therefore, by the triangle inequality,
    \begin{equation*}
        \norm{r(q\beta_x - p\beta_y)} \leq \norm{q(r\beta_x - \alpha_{x'})} + \norm{q\alpha_{x'}-p\alpha_{y'}} + \norm{p(r\beta_y - \alpha_{y'})}  \ll \frac{500^{L}}{H},
    \end{equation*}
    using the second property in \cref{lem:clusterlift}. Moreover,
    \begin{equation*}
        \norm{p^*(q\beta_x - p\beta_y)} = \norm{q\alpha_x - p\alpha_y} \leq \frac{1}{10H}
    \end{equation*}
    as $(x,y,p,q)\in \Q(\A)$. Since $p^*\neq r$, these two estimates imply that $\norm{q\beta_x - p\beta_y} \leq (10Hp^*)^{-1}$ by \cref{lem:normcalculus}. Thus, $(x,y,p,q)\in \Q(\B)$, which completes the proof of \cref{thm:structure}.
\end{proof}

\part{Appendices}
\label{part:appendices}

\appendix

\begin{appendices}

    \section{Recovering Walsh's conditional theorem}
    \label{sec:GRHcase}
    In this appendix, we show how Walsh's Fourier uniformity result under GRH~\cite{walsh3} can be derived from our global structure theorem, yielding explicit quantitative savings (see \cref{cor:GRHfourieruniformity}).

    \subsection{Conditional expansion estimate}
    Assuming the Generalised Riemann Hypothesis for Dirichlet $L$-functions, we can replace \cref{prop:GRHsubstitute} with a much stronger expansion estimate.

    \begin{lemma}[GRH expansion estimate]
        \label{lem:GRHexpansion}
        Assume GRH. Suppose that~${P \geq (\log Y)^{C}}$ for some sufficiently large absolute constant $C>0$.

        Let $A, B\subset \big[\tfrac{1}{10}Y, 10Y\big]$ be multisets with at most $M$ elements in any unit interval.

        Let $1\leq q_0 \leq Y^{O(1)}$ be an integer, let $\P' \subset \P$ consist of all primes not dividing $q_0$, and let~$(a_x)_{x\in A}$ be a sequence of integers coprime to $q_0$.\footnote{We assume that $C$ is sufficiently large in terms of the implied constant in the upper bound for $q_0$.}

        Let $1/Y\leq \eps \leq 1$. Then, the number of quadruples $(x,y,p,q)\in A\times B \times (\P')^2$ such that
        \begin{equation*}
            \abs{px - qy} \leq \eps P\quad \text{and}\quad q a_x \equiv p a_y \pmod{q_0}
        \end{equation*}
        is
        \begin{equation*}
            \ll M^2 \bigg( |A|^{1/2} |B|^{1/2} P^{2-c} + \frac{\eps |A| |B| |\P|^2}{Y \phi(q_0)} \bigg),
        \end{equation*}
        where $c>0$ is an absolute constant.
    \end{lemma}

    \begin{proof}
        Repeating the arguments in the proof of \cref{prop:GRHsubstitute}, we see that the number of quadruples under consideration is bounded by
        \begin{equation*}
            \ll M^2 \Bigg((w|\P|)^2|A|^{1/2}|B|^{1/2} + \frac{\eps |A||B|}{Y\phi(q_0)} \sum_{\chi \spmod{q_0}} \int_{-Y/\eps}^{Y/\eps} \Bigg\lvert\!\sum_{p\in \P} \chi(p) p^{it}\Bigg\rvert^2 \ind{\{\lvert\sum_{p\in \P} \overline{\chi(p)} p^{it} \rvert > w |\P|\}} dt\Bigg),
        \end{equation*}
        for any $w>0$. Setting $w := P^{-c}$ for some small absolute constant $c>0$ and bounding $\sum_{p\in \P} \chi(p) p^{it}$ using GRH (see \cite[Section~2.2]{walsh3} for details), we obtain the claimed bound.
    \end{proof}

    We remark that the Riemann Hypothesis is sufficient to obtain \cref{lem:GRHexpansion} in the case $q_0 = 1$. Specialising to $M = q_0 = \eps = 1$ yields the following corollary.

    \begin{corollary}
        \label{cor:GRHexpansion}
        Assume RH. Suppose that~${P \geq (\log Y)^{C}}$ for some sufficiently large constant $C>0$.

        Then, any configuration $\A = (A, \alpha_{\bullet}, H)$ with $|\Q(\A)| \geq \delta |A| |\P|^2$ for some $\delta \geq (\log Y)^{-100}$ satisfies
        \begin{equation*}
            |A| \gg \delta Y.
        \end{equation*}
    \end{corollary}

    The conclusion of \cref{cor:GRHexpansion} is essentially as strong as what the random graph heuristic would predict.

    \subsection{Construction of a tower of configurations}
    For $P$ as small as in Walsh's conditional setting, we cannot directly apply \cref{prop:tower} to obtain a tower of configurations. To construct a suitable tower, we instead combine our local structure theorem, \cref{thm:structure}, with the strong expansion estimate given by \cref{lem:GRHexpansion}.

    We start by proving a variant of \cref{lem:liftorsplit}. The main observation is that cases \cref{item:GRHincrement1} and \cref{item:GRHincrement2} of \cref{lem:alternativeforGRHcase} pass to a subset of $A$ that retains more local relations than would be expected from a random subset of the same size. Unlike in the unconditional setting, losing the constant factor $\frac{99}{100}$ in case \cref{item:GRHincrement2} is acceptable here since the RH expansion estimate is strong enough to handle a loss of this magnitude.

    \begin{lemma}
        \label{lem:alternativeforGRHcase}
        Assume RH. Let $\A = (A, \alpha_{\bullet}, H)$ be a configuration such that $|\Q(\A)| \geq \delta |A| |\P|^2$, and suppose that~${P \geq (\log Y)^{C\delta^{-1}}}$ for some sufficiently large constant $C>0$.

        Then, one of the following is true.
        \begin{enumerate}[label=(\roman*),ref=\roman*]
            \item \label{item:GRHlift} There exists a lift $\B = (B, \beta_{\bullet}, Hp^*)$ of $\A$ such that
                  \begin{equation*}
                      |\Q(\B)| \geq \Big(1 - O\Big( \frac{1}{(\log Y)^{10}} \Big) \Big) \delta |B| |\P|^2.
                  \end{equation*}
            \item \label{item:GRHincrement1} There is a proper subset $B\subset A$ such that $|B| \geq \tfrac12 |A|$ and $|\Q(\A|_{B})| \geq \delta |B||\P|^2$.
            \item \label{item:GRHincrement2} There is a subset $B \subset A$ such that $|B| \leq \tfrac12 |A|$ and $|\Q(\A|_{B})| \geq \tfrac{99}{100} \delta |B||\P|^2$.
        \end{enumerate}
    \end{lemma}

    \begin{proof}
        Let $\eps := (\log Y)^{-20}$ and let $L$ be the largest integer satisfying the conditions $\eps^{-\Clift L} \leq P$ and $\eps \delta^{-5} L \leq 1$ of \cref{thm:structure}. Note that $\delta \geq 1/\log Y$ as $Y\geq P^3 \geq (\log Y)^{C\delta^{-1}}$. Thus, the parameter~$L$ is only limited by the condition $\eps^{-\Clift L} \leq P$, so we have $L \asymp \log P / \log \log Y$.

        Let $A = A_0 \sqcup A_1 \sqcup \ldots \sqcup A_{L+1}$ and ${\B = (A\setminus A_{L+1}, \beta_{\bullet}, Hp^*)}$ be the decomposition and the lift given by \cref{thm:structure}.

        \textbf{Case 1.} If $|A_{L+1}| \leq |A|/(\log Y)^{11}$, then by part \cref{item:struct2b} of \cref{thm:structure}, we have
        \begin{equation*}
            |\Q(\B)| \geq \delta |A| |\P|^2 - O\! \left(\eps |A||\P|^2+|A_{L+1}||\P|^2\right) \geq \delta |B| |\P|^2 - O\!\left(\frac{|A| |\P|^2}{(\log Y)^{11}} \right),
        \end{equation*}
        so that conclusion \cref{item:GRHlift} holds. We may therefore assume that $|A_{L+1}| \geq |A|/(\log Y)^{11}$.

        \textbf{Case 2.} Suppose that there exists $1\leq i \leq L$ such that $|A_i| \leq |A|/(\log Y)^{13}$. Let $B_1 := \bigsqcup_{j=0}^{i} A_j$ and $B_2 := \bigsqcup_{j=i+1}^{L+1} A_j$. Then, by part \cref{item:struct2a} of \cref{thm:structure},
        \begin{equation}
            \label{eq:duringGRHSplit}
            |\Q(\A|_{B_1})| + |\Q(\A|_{B_2})| \geq \delta |A| |\P|^2 - O\! \left(\eps |A||\P|^2+|A_i||\P|^2\right) \geq \delta |A| |\P|^2 - O\!\left(\frac{|A| |\P|^2}{(\log Y)^{13}} \right).
        \end{equation}
        Since $|A_0| \gg |A|/(\log Y)^5$ by part \cref{item:struct1} of \cref{thm:structure}, and $|A_{L+1}| \geq |A|/(\log Y)^{11}$, we have $\min(|B_1|, |B_2|) \gg |A|/(\log Y)^{11}$. Hence, \cref{eq:duringGRHSplit} can be rewritten as
        \begin{equation*}
            |\Q(\A|_{B_1})| + |\Q(\A|_{B_2})| \geq \delta |B_1| |\P|^2 + \delta |B_2| |\P|^2 - O\!\left(\frac{\delta \min(|B_1|, |B_2|) |\P|^2}{\log Y} \right).
        \end{equation*}
        For $Y$ sufficiently large, this estimate implies that one of conclusions \cref{item:GRHincrement1} or \cref{item:GRHincrement2} holds for some choice of $B\in \{B_1, B_2\}$: if $|B_1| \geq |A|/2$, then either conclusion \cref{item:GRHincrement1} holds for $B = B_1$, or conclusion \cref{item:GRHincrement2} holds for $B = B_2$; the case $|B_2| \geq |A|/2$ is treated symmetrically.

        \textbf{Case 3.} It only remains to treat the case where $|A_i| \gg |A|/(\log Y)^{13}$ for all $0\leq i \leq L+1$. By averaging, there exists $1<i<L$ such that $|A_i|+|A_{i+1}| \ll |A|/L$.

        Let $B_1 := \bigsqcup_{j=0}^{i} A_j$ and $B_2 := \bigsqcup_{j=i+1}^{L+1} A_j$. By the conditional expansion estimate, \cref{lem:GRHexpansion}, the number of quadruples $(x,y,p,q)\in \Q(\A)$ with $x\in A_i$ and $y\in A_{i+1}$ is bounded by
        \begin{equation*}
            \ll |A_i|^{1/2} |A_{i+1}|^{1/2} P^{2-c} + \frac{|A_i||A_{i+1}| |\P|^2}{Y} \ll \frac{|A_i||A_{i+1}| |\P|^2}{Y},
        \end{equation*}
        using that $P \geq (\log Y)^C$ for some sufficiently large constant $C>0$ in the last inequality. Therefore, by part \cref{item:struct2a} of \cref{thm:structure}, we get
        \begin{equation*}
            |\Q(\A|_{B_1})| + |\Q(\A|_{B_2})| \geq \delta |A| |\P|^2 - O\! \left(\eps |A||\P|^2+ \frac{|A_i||A_{i+1}| |\P|^2}{Y}\right).
        \end{equation*}
        Since $|A|/(\log Y)^{13} \ll |A_i|, |A_{i+1}| \ll |A|/L$ and $\eps = (\log Y)^{-20}$, the error term can be bounded by $O(\min(|B_1|, |B_2|) |\P|^2 / L)$, so that
        \begin{equation*}
            |\Q(\A|_{B_1})| + |\Q(\A|_{B_2})| \geq \delta |B_1| |\P|^2 + \delta |B_2| |\P|^2 - O\!\left(\frac{\delta \min(|B_1|, |B_2|) |\P|^2}{\delta L} \right).
        \end{equation*}
        As before, this estimate implies that one of conclusions \cref{item:GRHincrement1} or \cref{item:GRHincrement2} holds for some $B\in \{B_1, B_2\}$, provided that $\delta L$ is sufficiently large. Since
        \begin{equation*}
            L \asymp \frac{\log P}{\log \log Y} \geq C \delta^{-1},
        \end{equation*}
        this is indeed the case if $C$ is chosen to be sufficiently large.
    \end{proof}

    Iterating the preceding lemma, we can construct a tower of configurations with a large relative density of local relations at the top level.

    \begin{lemma}
        \label{lem:GRHtower}
        Assume RH. Let $\A = (A, \alpha_{\bullet}, H)$ be a configuration such that $|\Q(\A)| = \delta |A| |\P|^2$, and suppose that~${P \geq (\log Y)^{C\delta^{-2}}}$ for some sufficiently large constant $C>0$.

        Then, there exists a tower of configurations $(\A_i)_{0\leq i \leq k}$ of height $k\gg (\log Y)^{10}$, such that $\A_0 = \A$ and $|\Q(\A_k)| \gg \delta^2 |A_k| |\P|^2$ (where $A_k$ is the set of points of $\A_k$).
    \end{lemma}

    \begin{proof}
        We imitate the proof of \cref{prop:tower}, using \cref{lem:alternativeforGRHcase} instead of \cref{lem:liftorsplit}.

        That is, by induction, we construct a sequence of triples $(\B_i, B_i, \delta_i)$ where $\B_i$ is a configuration with point set $B_i$ satisfying $|\Q(\B_i)| = \delta_i |B_i| |\P|^2$. We start with $(\B_0, B_0, \delta_0) := (\A, A, \delta)$ and construct the subsequent triples by repeatedly applying \cref{lem:alternativeforGRHcase}. Thus, for each $i\geq 1$, one of the following alternatives holds:
        \begin{itemize}
            \item \textit{Lifting step:} $\B_i$ is a lift of $\B_{i-1}$ and $\delta_i \geq \left(1 - O\!\left((\log Y)^{-10} \right) \right)  \delta_{i-1}$.
            \item \textit{Lossless step:} $\B_i = \B_{i-1}|_{B_i}$ for some proper subset $B_i\subset B_{i-1}$ and $\delta_i \geq \delta_{i-1}$.
            \item \textit{Halving step.} $\B_i= \B_{i-1}|_{B_i}$ for some $B_i\subset B_{i-1}$ with $|B_i| \leq \tfrac12 |B_{i-1}|$, and $\delta_i \geq \tfrac{99}{100} \delta_{i-1}$.
        \end{itemize}
        We run this iterative procedure until either $\delta_i< c\delta^2$ or we have encountered $\lfloor c (\log Y)^{10}\rfloor$ lifting steps, for some small parameter $c>0$ to be chosen later.

        Suppose that, for some $i_1 \geq 1$, we have $\delta_{i_1} < c\delta^2$, but there have been fewer than $\lfloor c (\log Y)^{10}\rfloor$ lifting steps up to time $i_1$. Then, if $c>0$ is sufficiently small, we have $\delta_{i_1} \asymp (99/100)^m \delta$, where $m$ is the number of halving steps up to time $i_1$. This is because lossless steps do not decrease $\delta_i$, and lifting steps decrease $\delta_i$ by at most a factor $1 - O((\log Y)^{-10})$. Thus, we see that
        \begin{equation*}
            \delta_{i_1} \asymp \big( \tfrac{99}{100} \big)^m\delta \asymp c\delta^2 \quad\text{and}\quad |B_{i_1}| \leq 2^{-m} |A|.
        \end{equation*}
        However, since $|\Q(\B_{i_1})| = \delta_{i_1} |B_{i_1}| |\P|^2$, the expansion estimate, \cref{cor:GRHexpansion}, implies that
        \begin{equation*}
            |B_{i_1}| \gg \delta_{i_1} Y.
        \end{equation*}
        Combining these estimates, we get
        \begin{equation*}
            c\delta^2 \ll 2^{-m} \leq \big( \tfrac{99}{100} \big)^{2m} \asymp c^2\delta^2,
        \end{equation*}
        which is a contradiction if $c>0$ is a sufficiently small absolute constant.

        As a result, the algorithm must terminate after $\gg (\log Y)^{10}$ lifting steps have been performed. Defining $\A_i$ to be the configuration obtained after the $i$-th lifting step, we obtain the desired tower of configurations.
    \end{proof}

    \subsection{Derivation of Walsh's conditional theorem} We now prove a stronger version of \cref{lem:largeHcase} assuming GRH, and then deduce Walsh's global structure theorem.

    \begin{lemma}
        \label{lem:GRHlargeHcase}
        Assume GRH. There is an absolute constant $C>0$ such that the following holds.

        Suppose that $P \geq (\log Y)^{C}$. Let $\A = (A, \alpha_{\bullet}, H)$ be a configuration such that $|\Q(\A)| \geq \delta |A| |\P|^2$, where $H \geq P^{3\log Y}$ and $\delta \geq C/\log |\P|$.

        Then, there exist
        \begin{itemize}
            \item a subset $A'\subset A$ with $|\Q(\A|_{A'})| \gg \delta |A'| |\P|^2$,
            \item an integer $q_0 \ll \delta^{-2}$, and
            \item a real number $t$ with $|t| \ll  \delta^{-1} (\log Y)^3 \frac{Y^2}{HP}$,
        \end{itemize}
        such that, for every $x\in A'$, there is an integer $a_x$ coprime to $q_0$ such that
        \begin{equation}
            \alpha_x = \frac{a_x}{q_0} + \frac{t}{x} + O\bigg(\frac{ \delta^{-1} (\log Y)^4}{HP}\bigg) \pmod 1.
        \end{equation}
    \end{lemma}

    \begin{proof}
        The proof is the same as that of \cref{lem:largeHcase}, with the GRH expansion lemma \cref{lem:GRHexpansion} replacing \cref{prop:GRHsubstitute}. We quickly reproduce the main steps for the reader's convenience.

        After applying \cref{lem:getfirstformula}, we obtain an initial bound on $q_0$ of the form
        \begin{equation*}
            q_0 \leq \exp\!\left( O\!\left(\frac{(\log Y)\log \log Y}{\log |\P|} + \log |\P|\right) \right)  \leq Y^{O(1)}.
        \end{equation*}
        We also obtain a subset $A_2 \subset A$ such that, for all $x\in A_2$,
        \begin{equation*}
            \alpha_x = \frac{a_x}{q_0} + \beta_x \pmod 1,
        \end{equation*}
        where $(a_x, q_0)=1$ and $|\beta_x| \leq P^{-2\log Y}$. Moreover, for each $x\in A_2$, there are $\gg \delta |\P|^2$ quadruples $(x,y,p,q)\in \Q(\A|_{A_2})$ such that $(pq, q_0)=1$, $p\neq q$, ${qa_x \equiv p a_y \pmod{q_0}}$ and $|q\beta_x - p\beta_y| \leq 1/H$.

        The GRH expansion estimate, \cref{lem:GRHexpansion}, applied with $A=B=A_2$ and $M = \eps=1$, implies that
        \begin{equation*}
            q_0 \ll \phi(q_0)^{1+o(1)} \ll \delta^{-1+o(1)}.
        \end{equation*}

        Applying \cref{lem:betax} to the configuration $\B := (A_2, \beta_\bullet, H)$ then produces a subset $B' \subset A_2$ with $|\Q(\B|_{B'})| \gg \delta |B'| |\P|^2$ and an element $x_0 \in B'$ such that, for every $x\in B'$,
        \begin{equation*}
            \alpha_x = \frac{a_x}{q_0} + \frac{x_0\beta_{x_0}}{x} + O \!\left(\frac{ |\beta_{x_0}| \log Y}{Y} + \frac{\log Y}{HP}\right) \pmod 1.
        \end{equation*}
        As in the proof of \cref{lem:largeHcase}, we bound $|\beta_{x_0}|$ by applying an expansion estimate to the multiset~$V$ given in \cref{lem:betax}. By \cref{lem:GRHexpansion}, the number of quadruples $(v_1, v_2, p,q)\in V^2 \times \P^2$ such that $|pv_1 - qv_2| \leq \eps P$ is
        \begin{equation*}
            \ll (\log Y)^2 \bigg( |V| P^{2-c} + \frac{\eps |V|^2 |\P|^2}{Y} \bigg)
        \end{equation*}
        for any $1/Y \leq \eps \leq 1$. Comparing this with the final conclusion of \cref{lem:betax}, we deduce that
        \begin{equation*}
            |\beta_{x_0}| \ll \delta^{-1} (\log Y)^3 \frac{Y}{HP},
        \end{equation*}
        which gives the claimed result.
    \end{proof}

    Finally, we can prove Walsh's global structure theorem under GRH.

    \begin{theorem}
        \label{thm:GRHglobalstructuretheorem}
        Assume GRH. There is an absolute constant $C>0$ such that the following holds.

        Let $\A = (A, \alpha_{\bullet}, H)$ be a configuration such that $|\Q(\A)| \geq \delta |A| |\P|^2$, where $\delta \geq C(\log P)^{-1/2}$.

        Suppose that $P \geq (\log Y)^{C\delta^{-2}}$. Then, there is a subset $A'\subset A$ of size
        \begin{equation*}
            |A'| \gg \delta^6 |A|,
        \end{equation*}
        such that, for all $x\in A'$, we have the approximate formula
        \begin{equation*}
            \alpha_x = \frac{a_0}{q_0} + \frac{T_0}{x} + O\bigg(\frac{\delta^{-2} (\log Y)^4}{HP}\bigg) \pmod 1,
        \end{equation*}
        where $a_0, q_0$ are coprime integers with $1\leq q_0 \ll \delta^{-4}$ and $T_0\in \R$ satisfies $|T_0| \leq \delta^{-2} (\log Y)^3 \frac{Y^2}{HP}$.
    \end{theorem}

    \begin{proof}
        The proof of \cref{thm:GRHglobalstructuretheorem} closely follows that of \cref{thm:globalstructuretheorem}, with \cref{lem:GRHtower} and \cref{lem:GRHlargeHcase} replacing \cref{prop:tower} and \cref{lem:largeHcase}, respectively.

        By \cref{lem:GRHtower}, we can construct a tower of configurations $(\A_i)_{0\leq i \leq k}$ of height $k \gg (\log Y)^{10}$, such that $\A_0 = \A$ and $|\Q(\A_k)| \gg \delta^2 |A_k| |\P|^2$, where $\A_k = (A_k, \widetilde{\alpha}_{\bullet}, H_k)$.

        Applying \cref{lem:GRHlargeHcase} to $\A_k$, we obtain a subset $A' \subset A_k$ with ${|\Q(\A_k|_{A'})| \gg \delta^2 |A'| |\P|^2}$, along with an integer $q_0 \ll \delta^{-4}$ and a real number $|t| \ll  \delta^{-2} (\log Y)^3 \frac{Y^2}{H_k P}$, such that, for every $x\in A'$,
        \begin{equation*}
            \widetilde{\alpha}_x = \frac{a_x}{q_0} + \frac{t}{x} + O\bigg(\frac{ \delta^{-2} (\log Y)^4}{H_k P}\bigg) \pmod 1
        \end{equation*}
        for some integer $a_x$ coprime to $q_0$.

        By \cref{cor:GRHexpansion}, we have the lower bound $|A'| \gg \delta^2 Y$. By pigeonholing on the values of $a_x \spmod{q_0}$, there exists a subset $A''\subset A'$ with $|A''| \gg \delta^6 Y$ such that $a_x$ is constant for $x\in A''$. The conclusion then follows from the definition of a lift, as in the proof of \cref{thm:globalstructuretheorem}.
    \end{proof}

    Combining \cref{thm:GRHglobalstructuretheorem} with \cref{prop:setup} and \cref{thm:MR} (as in the proof of \cref{thm:maingeneral}), we recover a quantitative form of Walsh's conditional Fourier uniformity estimate~\cite{walsh3}.

    \begin{corollary}
        \label{cor:GRHfourieruniformity}
        Assume GRH. Let $H = H(X) = (\log X)^{\psi(X)}$, where $\psi(X)\to \infty$ as $X\to \infty$, and ${\psi(X) \leq (\log X)^{1-c}}$ for some fixed $c>0$.

        Then
        \begin{equation*}
            \sum_{X\leq x < 2X} \sup_{\alpha\in \R/\Z}  \bigg\lvert\!\sum_{x \leq n < x+H} \lambda(n) e(n\alpha)\bigg\rvert \ll \frac{HX}{\sqrt{\log \psi(X)}}.
        \end{equation*}
    \end{corollary}

    \begin{proof}
        Let $X$ be sufficiently large, and suppose that the conclusion fails. Then, there exists an $H$-separated subset $S \subseteq [X, 2X)$ of size $\gg \delta X/H$, such that for each $x \in S$, there exists $\theta_x \in \R/\Z$ satisfying
        \begin{equation*}
            \bigg\lvert\!\sum_{x \leq n < x+H} \lambda(n) e(n\theta_x)\bigg\rvert \gg \delta H,
        \end{equation*}
        where
        \begin{equation}
            \label{eq:deltaGRH}
            \delta := \frac{C_1}{\sqrt{\log \psi(X)}}
        \end{equation}
        for some large constant $C_1>0$ to be chosen later.

        By \cref{prop:setup} (the assumption on $\delta$ holds by \cref{eq:deltaGRH} if $C_1$ is sufficiently large), there exists a scale $P$ with
        \begin{equation*}
            H^{e^{-O(\delta^{-2})}} \leq P \leq H^{1/10}
        \end{equation*}
        such that, writing $\P$ for the set of primes in $[P, 2P)$, the number of quadruples ${(x, y, p, q) \in S^2 \times \P^2}$ such that $|px-qy|\leq \tfrac{1}{10} PH$ and $\|q \theta_x - p \theta_y\| \leq \tfrac{P}{10H}$ is $\gg \delta^{7} |S| |\P|^2$. Defining $Y := X/H$, $A := \frac{1}{H} S\subset [Y, 2Y)$ and $\alpha_x := \theta_{Hx}$ for $x\in A$, we see that $\A := (A, \alpha_{\bullet}, H/P)$ is a configuration satisfying
        \begin{equation*}
            \abs{\Q(\A)} \gg \delta^{7} |A| |\P|^2.
        \end{equation*}

        We now apply \cref{thm:GRHglobalstructuretheorem} to $\A$. Again, note that the required lower bound on $P$ is satisfied, using~\cref{eq:deltaGRH} and choosing $C_1$ sufficiently large. Thus, by \cref{thm:GRHglobalstructuretheorem}, there exists a subset $A' \subseteq A$ of size $|A'| \gg \delta^{O(1)} |A|$, such that, for each $x \in A'$, we have the approximate formula
        \begin{equation*}
            \alpha_x = \frac{a_0}{q_0} + \frac{T_0}{x} + O\bigg(\frac{\delta^{O(1)}(\log X)^4}{H}\bigg) \pmod 1
        \end{equation*}
        where $(a_0, q_0)=1$, $q_0 \ll \delta^{-O(1)}$ and $|T_0| \ll \delta^{-O(1)} (\log X)^3 Y^2/H$.

        We conclude by applying \cref{thm:MR}. In the present setting, the condition \cref{eq:paramassumpt} takes the form
        \begin{equation*}
            (\log X)^{\delta^{-C}} \leq H \leq X^{\delta^{C}}
        \end{equation*}
        for some large constant $C>0$, which holds by definition of $H$ and $\delta$, for $X$ sufficiently large. Hence, \cref{thm:MR} yields
        \begin{equation*}
            \sum_{p\leq X} \frac{1+\Re(\chi(p)p^{it})}{p} \ll \log (\delta^{-1})
        \end{equation*}
        for some Dirichlet character $\chi$ of conductor $\ll \delta^{-O(1)} \ll (\log \log X)^{O(1)}$ and some real number $t$ with $|t| \ll (\delta^{-1} \log X)^{O(1)} X^2/H^2 \ll X^2$. This contradicts classical estimates for twisted character sums under GRH (see e.g.~\cite[Chapter~13]{MV}) when $X$ is sufficiently large.
    \end{proof}
    \section{Existence of local relations between the frequencies}
    \label{sec:appendixlocalrelations}
    In this appendix, we reproduce the arguments in \cite[Section~3]{walsh1} to show that the frequencies $\theta_x$ associated to large Fourier coefficients of a multiplicative function in short intervals must satisfy certain local relations, making the quantitative dependence on $\delta$ explicit.

    \begin{lemma}
        \label{lem:Elliott}
        Let $g : I \to \C$ be a $1$-bounded function, where $I \subset \Z$ is an interval of size $|I| = H$. For every $\tau > 0$, there is an exceptional set of primes $\mathcal{E}$ with
        \begin{equation*}
            \sum_{p\in \mathcal{E}} \frac{1}{p} \ll \tau^{-2}
        \end{equation*}
        such that, for every prime $p\leq H$ not in $\mathcal{E}$, we have
        \begin{equation*}
            \bigg\lvert\frac{1}{H} \sum_{n\in I} g(n) - \frac{p}{H} \sum_{\substack{n\\ pn\in I}} g(pn) \bigg\rvert \leq \tau.
        \end{equation*}
    \end{lemma}

    \begin{proof}
        This is (essentially) Elliott's inequality~\cite[Theorem~3.13,~p.462]{tenenbaum}. We reproduce the proof for the interested reader.

        It is enough to show the second moment bound
        \begin{equation*}
            \sum_{p\leq H} \frac{1}{p} \bigg\lvert\frac{1}{H} \sum_{n\in I} g(n) - \frac{p}{H} \sum_{\substack{n\\ pn\in I}} g(pn) \bigg\rvert^2 \ll 1.
        \end{equation*}
        Writing $F(p) := {\sum_{n\in I}  \big(\ind{p\mid n} - \frac{1}{p}\big) p^{1/2} g(n)}$, this is equivalent to
        \begin{equation}
            \label{eq:elliotttoprove}
            \sum_{p \leq H} \abs{F(p)}^2 \ll H^2.
        \end{equation}
        By Cauchy-Schwarz, we have
        \begin{equation*}
            \sum_{p \leq H} \abs{F(p)}^2 = \sum_{n\in I} g(n)  \sum_{p\leq H}  \big(\ind{p\mid n} - \tfrac{1}{p}\big) p^{1/2} \overline{F(p)} \leq H^{1/2}\bigg( \sum_{n\in I} \bigg\lvert\sum_{p\leq H}  \big(\ind{p\mid n} - \tfrac{1}{p}\big) p^{1/2}  \overline{F(p)}\,\bigg\rvert^2 \bigg)^{1/2}.
        \end{equation*}
        Expanding the square and swapping the order of summation, we get
        \begin{equation}
            \label{eq:elliottinterm}
            \sum_{n\in I} \bigg\lvert\sum_{p\leq H}  \big(\ind{p\mid n} - \tfrac{1}{p}\big) p^{1/2}  \overline{F(p)}\,\bigg\rvert^2 \ll H \sum_{p\leq H} |F(p)|^2 + \sum_{\substack{p_1, p_2\leq H\\p_1\neq p_2}} p_1^{1/2}p_2^{1/2} \abs{F(p_1)}\abs{F(p_2)},
        \end{equation}
        using that $\sum_{n\in I} \big(\ind{p_1\mid n} - \frac{1}{p_1}\big)\big(\ind{p_2\mid n} - \frac{1}{p_2}\big) \ll 1$ for $p_1\neq p_2$. By Cauchy-Schwarz, the off-diagonal contribution in \cref{eq:elliottinterm} is negligible. Simplifying, we obtain \cref{eq:elliotttoprove}.
    \end{proof}

    \begin{lemma}
        \label{lem:parseval}
        Let $g : I \to \C$ be a $1$-bounded function, where $I \subset \Z$ is an interval of size $|I| = H$. Suppose that $S \subset \R/\Z$ is a set of $1/H$-separated frequencies such that
        \begin{equation*}
            \bigg\lvert\frac{1}{H} \sum_{n\in I} g(n)e(n\alpha) \bigg\rvert \geq \tau
        \end{equation*}
        for all $\alpha\in S$. Then $|S| \ll \tau^{-3}$.
    \end{lemma}

    \begin{proof}
        By the triangle inequality, the lower bound $\abs{\frac{1}{H} \sum_{n\in I} g(n)e(n\alpha)} \gg \tau$ persists for all $\alpha$ in a \mbox{$(c\tau/H)$-neighbourhood} of $S$, for some small constant $c>0$. Since $S$ is~$1/H$-separated, we have
        \begin{equation*}
            |S| \frac{\tau}{H} \tau^2 \ll \int_0^1 \bigg\lvert\frac{1}{H} \sum_{n\in I} g(n)e(n\alpha) \bigg\rvert^2\! d\alpha  \leq \frac{1}{H},
        \end{equation*}
        which rearranges to $|S| \ll \tau^{-3}$.
    \end{proof}

    \begin{proposition}
        \label{prop:setup}
        Let $f:\N \to \C$ be a $1$-bounded multiplicative function. Let $C_0 \leq H\leq X$ and $\delta \geq C_0 (\log \log H)^{-1/2}$ for some sufficiently large absolute constant $C_0$.

        Let $A\subset [X, 2X]$ be an $H$-separated set. Suppose that, for each $x\in A$, there exists a frequency $\theta_x \in \R/\Z$ satisfying
        \begin{equation*}
            \bigg\lvert{\sum_{x\leq n<x+H} f(n) e(n\theta_x)}\bigg\rvert \geq \delta H.
        \end{equation*}
        Then, there exists a scale $P$ with $H^{e^{-O(\delta^{-2})}} \leq P \leq H^{1/10}$ such that, writing $\P$ for the set of primes in $[P, 2P]$, there are
        \begin{equation*}
            \gg \delta^{6} \frac{|A|^2}{X/H} |\P|^2
        \end{equation*}
        quadruples $(x,y,p,q)\in A^2\times \P^2$ with $|qx-py| \leq \tfrac{1}{10}PH$ and ${\norm{p\theta_x-q\theta_y} \leq \frac{P}{10H}}$.
    \end{proposition}

    \begin{proof}
        Let $W$ be the set of all pairs $(x, p)\in A\times \{p \text{ prime}: p\leq H\}$ such that
        \begin{equation}
            \label{eq:defW2}
            \bigg\lvert{\frac{p}{H} \sum_{\substack{n\\ x\leq pn<x+H}} f(pn) e(pn\theta_x)}\bigg\rvert \geq \frac{\delta}{2}.
        \end{equation}
        The set $W$ can be shown to be large using Elliott's inequality. Indeed, for each $x\in A$, applying \cref{lem:Elliott} with $g(n) := f(n)e(n\theta_x)$, $I := [x, x+H)\cap \Z$ and $\tau := \delta/2$, we have
        \begin{equation*}
            \sum_{p\leq H} \frac{1-\ind{(x,p)\in W}}{p} \ll \delta^{-2}.
        \end{equation*}
        Let $P_1 := H^{1/10}$ and $P_0 := e^{-k}P_1$ for some parameter $k\in \N$ such that $10\leq P_0 \leq P_1^{1/2}$. Dropping the primes $p\notin [P_0, P_1)$ and summing over $x\in A$, we get
        \begin{equation*}
            \sum_{i=0}^{k-1} \frac{1}{e^{-i}P_1} \sum_{e^{-i-1}P_1\leq p<e^{-i}P_1}\,  \sum_{x\in A} \big(1-\ind{(x,p)\in W}\big) \ll \delta^{-2}  |A|.
        \end{equation*}
        By averaging, there exists $0\leq i \leq k$ such that
        \begin{equation*}
            \frac{\log (e^{-i}P_1)}{e^{-i}P_1} \sum_{e^{-i-1}P_1\leq p<e^{-i}P_1}\,  \sum_{x\in A} \big(1-\ind{(x,p)\in W}\big) \ll \frac{\delta^{-2}}{\sum_{i=0}^{k-1} \frac{1}{\log(e^{-i}P_1)}}|A| \ll \frac{\delta^{-2}}{\log \frac{\log P_1}{\log P_0}} |A|.
        \end{equation*}
        Let $C$ be a large constant to be chosen later. We assume throughout that the constant $C_0$ in the statement of the proposition is sufficiently large with respect to $C$. Select $k$ so that $P_0 \asymp H^{e^{-C\delta^{-2}}}$. By the prime number theorem, for $P' := e^{-i-1}P_1 \in [P_0, P_1]$, we have
        \begin{equation*}
            \frac{1}{|\P'|}\sum_{p\in \P'}\sum_{x\in A} \ind{(x,p)\in W} \gg |A|
        \end{equation*}
        provided $C$ is sufficiently large, where $\P'$ denotes the set of primes in $[P', eP')$. It will be convenient to restrict to a slightly shorter interval of primes. By the pigeonhole principle, there exists $P\in [P', eP')$ such that, writing $\P$ for the set of primes in $[P, (1+\delta/C)P]$, we have
        \begin{equation}
            \label{eq:goodlbW}
            \sum_{x\in A} \sum_{p\in \P}\ind{(x,p)\in W} \gg \frac{\delta}{C} |A| |\P'|.
        \end{equation}

        For every $(x,p)\in W \cap (A\times \P)$, since $f$ is multiplicative (but not necessarily completely multiplicative), the estimate \cref{eq:defW2} implies that
        \begin{equation}
            \label{eq:usemult2}
            \bigg\lvert{\frac{p}{H} \sum_{\substack{n\\ x\leq pn<x+H}} f(n) e(n p\alpha_x)}\bigg\rvert \geq \frac{\delta}{2} - O(P^{-1}) \geq \frac{\delta}{4},
        \end{equation}
        where the last inequality holds provided $C_0$ is large enough in terms of $C$, using the assumption $\delta \geq C_0 (\log \log H)^{-1/2}$ and the definition of $P_0$.

        Let $Z$ be a maximal $C^{-1}\delta H/P$-separated subset of $[\frac{X}{2P}, \frac{2X}{P}]$. For each $z\in Z$, let $S_z$ be a maximal $\frac{P}{20H}$-separated set of frequencies $\beta \in \R/\Z$ with the property that
        \begin{equation*}
            \bigg\lvert{\frac{P}{H} \sum_{\substack{z\leq n<z+\frac{H}{P}}} f(n) e(n \beta)}\bigg\rvert \geq \frac{\delta}{8}.
        \end{equation*}
        By \cref{lem:parseval}, we know that $\abs{S_z} \ll \delta^{-3}$.

        For each pair $(x,p)\in W\cap (A\times \P)$, there is at least one $z\in Z$ such that $|x/p - z| \leq C^{-1}\delta H/P$; by~\cref{eq:usemult2}, we then have
        \begin{equation*}
            \bigg\lvert{\frac{P}{H} \sum_{\substack{z\leq n<z+\frac{H}{P}}} f(n) e(n p\alpha_x)}\bigg\rvert \geq \bigg\lvert{\frac{p}{H} \sum_{\substack{ \frac{x}{p}\leq n<\frac{x+H}{p}}} f(n) e(n p\alpha_x)}\bigg\rvert - O(C^{-1}\delta) \geq \frac{\delta}{8}
        \end{equation*}
        for $C$ large enough, using that $|p-P| \leq C^{-1}\delta P$. Therefore, $\norm{p\alpha_x - \beta} \leq \frac{P}{20H}$ for some $\beta\in S_z$. By Cauchy-Schwarz, this implies that the number of tuples $(x,y,p,q,z,\beta)\in A^2\times \P^2 \times Z\times \R/\Z$ such that $(x,p), (y,q)\in W$, $\beta\in S_z$,
        \begin{equation*}
            |x/p-z|, |y/q-z| \leq C^{-1}\delta H/P\quad\text{and}\quad \norm{p\alpha_x - \beta}, \norm{q\alpha_y - \beta} \leq \tfrac{P}{20H}
        \end{equation*}
        is at least
        \begin{equation*}
            \frac{(\sum_{x\in A}\sum_{p\in \P} \ind{(x,p)\in W})^2}{\sum_{z\in Z} |S_z|} \gg \frac{(C^{-1}\delta|A| |\P'|)^2}{C\delta^{-4}X/H} = C^{-3}\delta^{6} \frac{|A|^2}{X/H} |\P'|^2,
        \end{equation*}
        using \cref{eq:goodlbW} and the bound $|S_z| \ll \delta^{-3}$. Hence, by the triangle inequality, the number of quadruples $(x,y,p,q)\in A^2\times \P^2$ with $|qx-py| \ll C^{-1}\delta PH$ and $\norm{p\alpha_x - q\alpha_y} \leq \frac{P}{10H}$ is at least
        \begin{equation*}
            \gg C^{-3}\delta^{6} \frac{|A|^2}{X/H} |\P'|^2.
        \end{equation*}
        The conclusion follows upon choosing $C$ sufficiently large.
    \end{proof} \section{A variant of the Matomäki-Radziwiłł theorem}
    \label{sec:appendix_mr}
    In this appendix, we prove \cref{thm:MR}, a structural result characterising multiplicative functions that correlate with linear phases $e(n\alpha_x)$ on many short intervals, where $x\mapsto \alpha_x$ is a smoothly varying function of the form
    \begin{equation*}
        \alpha_x\approx \frac{a}{q}+\frac{T}{x} \pmod 1.
    \end{equation*}
    While the underlying arguments are well known to experts, they are scattered across \cite{MR,MRII,MRTTZ}, and no result of the required generality is proved in the literature. We fill in this gap by providing a detailed proof.

    \subsection{Preliminary lemmas and technical setup}

    The following lemma is a restatement of \cite[Lemma~8.1]{MRII}, whose proof is not explicitly provided in \cite{MRII}. Instead, the reader is referred to the proof of \cite[Lemma~14]{MR} for justification. Since \cite[Lemma~8.1]{MRII} differs slightly from \cite[Lemma~14]{MR}, we include a proof for completeness.

    \begin{lemma}
        \label{lem:MRParseval}
        Let $X\geq y\geq 1$. Let $f:\R\to \C$ be in $L^1\cap L^2$. Then
        \begin{equation}
            \label{eq:MRparseval}
            \frac{1}{X} \int_X^{2X} \abs{\int_{\R} f(t) \frac{(x+y)^{1+it} - x^{1+it}}{1+it} dt }^2 dx \ll \max_{T\geq X/y} \frac{Xy}{T} \int_{-T}^T \abs{f(t)}^2 dt.
        \end{equation}
    \end{lemma}
    \begin{proof}
        If $f$ is supported on $[-X/y, X/y]$, using the trivial bound
        \begin{equation}
            \label{eq:trivboundit}
            \abs{\frac{(x+y)^{1+it} - x^{1+it}}{1+it}} = \abs{\int_x^{x+y}u^{it}du}\leq y,
        \end{equation}
        we immediately obtain that the left-hand side of \cref{eq:MRparseval} is $\leq y^2 \int_{-X/y}^{X/y} \abs{f(t)}^2 dt$, which is acceptable.

        Hence, it suffices to treat the case where the support of $f$ is contained in the complement of~$[-X/y, X/y]$; the general case then follows by the triangle inequality.

        It turns out to be enough to prove the inequality
        \begin{equation}
            \label{eq:parsevaltoprove}
            \frac{1}{X} \int_X^{2X} \abs{\int_{\R} f(t) x^{1+it} \frac{(1+u)^{1+it} - 1}{1+it} dt }^2 dx \ll \max_{T\geq X/y} \frac{Xy}{T} \int_{-T}^T \abs{f(t)}^2 dt
        \end{equation}
        where $u \ll y/X \leq 1$ is a \emph{fixed} real number (this is the analogue of \cref{eq:MRparseval} where the $x$-dependent quantity $y/x$ has been replaced by $u$). This reduction from \cref{eq:MRparseval} to \cref{eq:parsevaltoprove} follows from a trick, which consists of applying the identity
        \begin{equation*}
            \frac{(x+y)^{s} - x^{s}}{s} = \frac{1}{2y} \left(\int_{y}^{3y} \frac{(x+w)^s-x^s}{s} dw - \int_{y}^{3y} \frac{(x+w)^s-(x+y)^s}{s} dw\right)
        \end{equation*}
        and pulling the integral over $w$ outside via Cauchy-Schwarz (see \cite[Proof of Lemma~14]{MR} for further details).

        Introducing a smooth weight function $\ind{[1,2]} \leq W \leq \ind{[1/2,5/2]}$ and changing the order of summation, the left-hand side of \cref{eq:parsevaltoprove} is
        \begin{equation}
            \label{eq:parsevalinterm}
            \ll \int_{\R} \int_{\R} \abs{\frac{f(t_1)f(t_2)}{t_1 t_2}} \abs{\frac{1}{X} \int_{\R} W\! \left(\frac{x}{X}\right) x^{2+i(t_1-t_2)} dx }dt_1 dt_2.
        \end{equation}
        A standard Mellin transform computation shows that
        \begin{equation*}
            \frac{1}{X} \int_{\R} W\! \left(\frac{x}{X}\right) x^{2+i(t_1-t_2)} dx \ll \frac{X^2}{1+|t_1-t_2|^2}.
        \end{equation*}
        Using $|f(t_1)f(t_2)|\leq |f(t_1)|^2+|f(t_2)|^2$, the expression \cref{eq:parsevalinterm} simplifies to
        \begin{equation*}
            \ll X^2 \int_{\R} \abs{ \frac{f(t)}{t} }^2 dt.
        \end{equation*}
        By a dyadic decomposition, and recalling that $f$ is zero on $[-X/y, X/y]$, we obtain \cref{eq:parsevaltoprove}. This proves \cref{lem:MRParseval}.
    \end{proof}

    Following Matomäki and Radziwiłł \cite{MR,MRII}, we shall work with a sequence $(a_n)$ supported on integers $n$ having prime factors in prescribed ranges, and obeying a suitable factorability assumption.

    \begin{notation}
        \label{def:factorable}
        Let $X\geq 10^{100}$. Let $10\leq P_1<Q_1\leq P_2<Q_2\leq P_3<Q_3\leq X^{1/3}$. Let $(a_n)_{n\geq 1}$ be a sequence of $1$-bounded complex numbers with the following properties.
        \begin{enumerate}
            \item The support of $(a_n)$ is contained in the set of $X/4 < n\leq 8X$ having at least one prime factor in each of the three intervals $(P_i, Q_i]$, and no repeated prime factor in those intervals.
            \item There are $1$-bounded sequences $(b_m)$ and $(c_p)$ such that, whenever $n=mp_1p_2p_3\in (X/4, 8X]$ for some integer $m$ and primes $p_i\in (P_i, Q_i]$ not dividing $m$, we have $a_n = b_m c_{p_1} c_{p_2} c_{p_3}$.
        \end{enumerate}
    \end{notation}

    \begin{notation}
        \label{not:WoneWtwo}
        Let $W_1,W_2$ be smooth functions such that $\ind{[1,\,5/2]}\leq W_1\leq \ind{[1/2,\,4]} \leq W_2 \leq \ind{[1/4, \,8]}$, and whose Mellin transforms ${\widetilde{W_j}(s) := \int_0^{\infty} W_j(x)x^{s-1}dx}$ satisfy the decay estimate
        \begin{equation}
            \label{eq:MRmellindecay}
            \widetilde{W_j}(it) \ll \exp(-c |t|^{1/2})
        \end{equation}
        for $j=1,2$ (use e.g.~the construction in~\cite{ingham}). We define
        \begin{equation}
            \label{eq:deffofs}
            F(s) := \sum_{n\geq 1}  \frac{a_n}{n^s}W_1 \!\left(\frac{n}{X}\right).
        \end{equation}
    \end{notation}

    \cref{lem:expression1forF,lem:expression2forF} give alternative expressions for $F(s)$ via Mellin inversion, with \cref{lem:expression2forF} also relying on the multiplicativity properties of the sequence $(a_n)$.

    \begin{lemma}
        \label{lem:expression1forF}
        For $s\in \C$ with $\Re(s)\geq 1$,
        \begin{equation*}
            F(s) = \frac{1}{2\pi} \int_{-(\log X)^3}^{(\log X)^3} \bigg(\sum_{n\geq 1} \frac{a_n}{n^{s+iu}}W_2 \Big(\frac{n}{X}\Big)\bigg) \widetilde{W_1}(iu) X^{iu} du + O(X^{-10}).
        \end{equation*}
    \end{lemma}
    \begin{proof}
        By Mellin inversion of $W_1$, for any $s\in \C$, we have
        \begin{equation*}
            F(s) = \sum_{n\geq 1}  \frac{a_n}{n^s}W_2 \!\left(\frac{n}{X}\right) W_1 \!\left(\frac{n}{X}\right) = \frac{1}{2\pi} \int_{-\infty}^{\infty} \sum_{n\geq 1} \frac{a_n}{n^{s+iu}} W_2 \!\left(\frac{n}{X}\right) \widetilde{W_1}(iu) X^{iu} du.
        \end{equation*}
        The integral can then be truncated using the fast decay of the Mellin transform \cref{eq:MRmellindecay}.
    \end{proof}

    \begin{lemma}
        \label{lem:expression2forF}
        For $s\in \C$ with $\Re(s)\geq 1$, we have
        \begin{equation*}
            F(s) = \sum_{A,B,C} \frac{1}{2\pi} \int_{-(\log X)^3}^{(\log X)^3} Q_{2,A}(s+iu) Q_{3,B}(s+iu)R_C(s+iu) \widetilde{W_1}(iu) X^{iu} du + F_2(s) + O(X^{-10}),
        \end{equation*}
        where $A,B$ and $C$ range over the powers of two such that $A\in (\tfrac12 P_2, Q_2]$, $B\in (\tfrac12 P_3, Q_3]$ and ${ABC\in [\tfrac{1}{16}X, 4X]}$, and where
        \begin{empheq}[left=\empheqlbrace]{align*}
            Q_{j, D}(s) &:= \sum_{p \in (D, 2D]\cap (P_j, Q_j]} \frac{c_p}{p^s},\\
            R_D(s) &:= \sum_{D<m\leq 2D} \frac{r_m}{m^s},\\
            F_2(s) &:= \sum_{\substack{X/2< n\leq 4X}} \frac{e_n}{n^s}
        \end{empheq}
        for some complex coefficients $r_m, e_n$ satisfying $|r_m|\leq 1$ and $\sum_n |e_n|^2 \ll X(\log X)^4/P_2$.
    \end{lemma}

    \begin{proof}
        Write $\omega_{(P_j, Q_j]}(n) := \sum_{p\in (P_j, Q_j]} \ind{p\mid n}$. The assumptions on $(a_n)$ (see \cref{def:factorable}) imply that
        \begin{equation*}
            \sum_{n\geq 1} \frac{a_n}{n^{s}} W_1 \!\left(\frac{n}{X}\right) =   \sum_{p\in (P_2, Q_2]}  \sum_{q\in (P_3, Q_3]} \sum_{\substack{m\geq 1 \\  (m, pq)=1}} \frac{b_m' c_p c_q }{(mpq)^s \omega_{(P_2, Q_2]}(mpq) \omega_{(P_3, Q_3]}(mpq)} W_1 \!\left(\frac{mpq}{X}\right)
        \end{equation*}
        for some $1$-bounded complex sequence $(b_m')$. Thus, letting
        \begin{equation*}
            r_m := \frac{b_m'}{(\omega_{(P_2, Q_2]}(m)+1)( \omega_{(P_3, Q_3]}(m)+1)},
        \end{equation*}
        we have
        \begin{equation*}
            \sum_{n\geq 1} \frac{a_n}{n^{s}}W_1 \!\left(\frac{n}{X}\right) =   \sum_{p\in (P_2, Q_2]}  \sum_{q\in (P_3, Q_3]} \sum_{\substack{m\geq 1 \\  (m,pq)=1}} \frac{r_m c_p c_q}{(mpq)^s} W_1 \!\left(\frac{mpq}{X}\right).
        \end{equation*}
        The condition $(m,pq)=1$ can be dropped, at the expense of an extra term of the form
        \begin{equation*}
            \sum_{n\geq 1} \frac{e_n}{n^s}
        \end{equation*}
        for some coefficients $|e_n| \ll (\log X)^2$ supported on the integers $n\in (\tfrac12 X, 4X]$ having a repeated prime factor in $(P_2, Q_2] \cup (P_3, Q_3]$. Note that the number of such integers $n$ is
        \begin{equation*}
            \ll \sum_{P_2<p\leq Q_3} \sum_{X/2<n\leq 4X} \ind{p^2\mid n} \ll X \sum_{P_2<p\leq Q_3} \frac{1}{p^2} \ll \frac{X}{P_2}.
        \end{equation*}

        Therefore, by dyadic partitioning,
        \begin{equation*}
            \sum_{n\geq 1} \frac{a_n}{n^{s}}W_1 \!\left(\frac{n}{X}\right) = \sum_{A,B,C}  \sum_{\substack{p\in (P_2, Q_2]\\ p\in (A, 2A]}}  \sum_{\substack{q\in (P_3, Q_3]\\ q\in (B, 2B]}} \sum_{\substack{m\in (C, 2C]}} \frac{r_m c_p c_q}{(mpq)^s} W_1 \!\left(\frac{mpq}{X}\right) +  \sum_{n\geq 1} \frac{e_n}{n^s}
        \end{equation*}
        where $A,B,C$ range over the set of powers of two such that $ABC\in [\tfrac{1}{16}X, 4X]$. The conclusion follows by Mellin inversion and truncation, as in the proof of \cref{lem:expression1forF}.
    \end{proof}

    \subsection{Main part of the Matomäki-Radziwiłł proof}

    \Cref{prop:MRpart1} below is the version of \mbox{\cite[Proposition~8.3]{MRII}} that we require. As noted in \cref{rem:MRpart1differences}, treating \cite[Proposition~8.3]{MRII} as a black box would be insufficient for our purposes (even though it is more general than \cref{prop:MRpart1} in other respects). By extracting from its proof only those ideas relevant to our purposes, we avoid some technical complications.

    \begin{proposition}
        \label{prop:MRpart1}
        Let $X, P_j, Q_j, (a_n), (b_m)$ and $(c_p)$ be as in \cref{def:factorable}.

        Let $1\leq y\leq X^{1/2}$. Let $\U \subset [-X, X]$ be any measurable set and $\nu > 0$. Then, there exists a measurable function $g : [X, 2X] \to \C$ such that
        \begin{equation*}
            \frac{1}{X} \int_{X}^{2X} \bigg\lvert{g(x)-\sum_{x<n\leq x+y} a_n} \bigg\rvert ^2 dx \ll E_1 + E_2 + E_3
        \end{equation*}
        where
        \begin{empheq}[left=\empheqlbrace]{align*}
            E_1 &:= (\log X)^2 + \frac{y^2 (\log X)^5}{P_2},\\
            E_2 &:=  \max_{X/y\leq T\leq X} \frac{Xy}{T} \int_{\substack{[-T, T]\qquad\  \\ \odist(t,\U) \geq X^{1/10} }} \bigg\lvert\sum_{n\geq 1} \frac{a_n}{n^{1+it}} W_2 \!\left(\frac{n}{X}\right)\! \bigg\rvert^2 dt,\\
            E_3 &:= \frac{y^2(\log X)^{11}}{X^{2\nu}} \bigg(1+ \frac{ Q_2 \cdot \abs{\{t\in [-X, X] \,:\, \odist(t, \U) \leq 4X^{1/10}\}}}{(Xy)^{1/2}}\bigg),
        \end{empheq}
        and $g$ satisfies
        \begin{equation}
            \label{eq:boundforg}
            \norm{g}_{\infty} \ll y \sum_{A,B} \int_{\V} \abs{Q_{2,A}(1+it) Q_{3,B}(1+it)} dt \end{equation}
        with
        \begin{equation}
            \label{eq:defV}
            \V := \Big\{t\in [-X, X]\, :\, \max_A \abs{Q_{2,A}(1+it)} \geq X^{-\nu} \Big\}.
        \end{equation}
        In \cref{eq:boundforg,eq:defV}, the variables $A$ and $B$ range over the sets of powers of two in $(\tfrac12 P_2, Q_2]$ and $(\tfrac12 P_3, Q_3]$, respectively. The Dirichlet polynomials $Q_{2, A}$ and $Q_{3,B}$ are those defined in \cref{lem:expression2forF}.
    \end{proposition}

    \begin{remark}
        \label{rem:MRpart1}
        Note that $E_2$ and $E_3$ depend on $\U$, while $E_3$ and $\norm{g}_{\infty}$ depend on $\nu$. We will later choose $\U$ and $\nu$ (as well as the intervals $(P_j, Q_j]$) to ensure that all of these quantities are suitably small.
    \end{remark}

    \begin{remark}
        \label{rem:MRpart1differences}
        A key difference between \cref{prop:MRpart1} and \cite[Proposition~8.3]{MRII} is that we do not apply a Hal\'asz-Montgomery type estimate to bound $\norm{g}_{\infty}$ at this stage. This allows us to exploit an additional averaging over Dirichlet characters later on.
    \end{remark}

    The following lemma is the first step towards \cref{prop:MRpart1}.

    \begin{lemma}
        \label{lem:MRstep1}
        Let $X, P_j, Q_j, (a_n), (b_m),(c_p),y$ and $\U$ be as in \cref{prop:MRpart1}.

        Let $J \subset [-X, X]$ be any measurable set. We have
        \begin{equation*}
            \frac{1}{X} \int_{X}^{2X} \bigg\lvert{g(x)-\sum_{x<n\leq x+y} a_n} \bigg\rvert ^2 dx \ll E_1 + E_2 + E_3^*
        \end{equation*}
        with $E_1, E_2$ as in \cref{prop:MRpart1} and
        \begin{equation*}
            E_3^* :=  \max_{\substack{A,B,C\\ X/y\leq T\leq X}} \frac{X(\log X)^{10}y}{T} \int_{\substack{[-T, T]\setminus J\quad\ \ \\ \odist(t,\U)\leq 4X^{1/10}}} \abs{Q_{2,A}(1+it) Q_{3,B}(1+it) R_C(1+it)}^2 dt,
        \end{equation*}
        for some measurable function $g$ satisfying the uniform bound
        \begin{equation*}
            \norm{g}_{\infty} \ll y \sum_{A,B}  \int_{J} \abs{Q_{2,A}(1+it) Q_{3,B}(1+it)} dt .
        \end{equation*}
        In these expressions, the variables $A,B,C$ and the Dirichlet polynomials $Q_{2,A}, Q_{3,B}$ and $R_C$ are defined to be as in \cref{lem:expression2forF}.
    \end{lemma}

    \begin{proof}
        By Perron's formula (e.g.~take $\alpha = 0$, $c = 1$, $T=X/2$ and $s\to 0^+$ in \cite[Lemma 1.1 (p.11)]{harman}), \begin{equation}
            \label{eq:perron}
            \sum_{x<n\leq x+y} a_n = \sum_{x<n\leq x+y} a_n W_1 \!\left(\frac{n}{X}\right)  = \frac{1}{2\pi} \int_{-X/2}^{X/2} F(1+it) \frac{(x+y)^{1+it} - x^{1+it}}{1+it} dt + O (\log X).
        \end{equation}

        Our goal is to write
        \begin{equation*}
            \sum_{x<n\leq x+y} a_n = g(x)+h(x)
        \end{equation*}
        with $g$ small in $L^\infty([X, 2X])$ and $h$ small in $L^2([X, 2X])$ (precisely, $\frac{1}{X} \int_{X}^{2X} \abs{ h(x) }^2 dx \ll E_1+E_2+E_3^*$). We will establish this decomposition by successively extracting terms from the formula \cref{eq:perron} and demonstrating that their contribution in $L^2([X, 2X])$ is acceptable. These terms will be absorbed into $h(x)$, while what remains after this process will define $g(x)$.

        First of all, the error term $O (\log X)$ from \cref{eq:perron} can clearly be absorbed into $h(x)$.

        Let $I := \big\{t\in [-X/2, X/2] \, :\, \odist(t, \U) \geq 2 X^{1/10}\big\}$. By \cref{lem:MRParseval} applied to the function ${f(t) := F(1+it) \ind{I}(t)}$, we have
        \begin{equation*}
            \frac{1}{X} \int_{X}^{2X} \abs{ \int_{I} F(1+it) \frac{(x+y)^{1+it} - x^{1+it}}{1+it} dt}^2 dx \ll  \max_{X/y\leq T\leq X/2} \frac{Xy}{T} \int_{I \cap [-T, T]} \abs{F(1+it)}^2 dt.
        \end{equation*}
        By \cref{lem:expression1forF}, Cauchy-Schwarz and \cref{eq:MRmellindecay}, for $X/y\leq T\leq X/2$ we have
        \begin{equation*}
            \int_{I \cap [-T, T]} \abs{F(1+it)}^2 dt \ll \int_{I \cap [-T, T]}  \int_{-(\log X)^3}^{(\log X)^3} \bigg\lvert\sum_{n\geq 1}\frac{a_n}{n^{1+it+iu}} W_2 \!\left(\frac{n}{X}\right)\!\bigg\rvert^2 \abs{\widetilde{W_1}(iu)} du\, dt    + \frac{1}{X^{19}}.
        \end{equation*}
        Swapping the order of integration, performing the change of variables $r:=t+u$ and integrating over $u$, this is
        \begin{equation*}
            \ll \int_{\substack{[-2T, 2T]\quad \  \\ \odist(r,\U) \geq X^{1/10} }} \bigg\lvert\sum_{n\geq 1}\frac{a_n}{n^{1+ir}}  W_2 \!\left(\frac{n}{X}\right) \!\bigg\rvert^2 dr  + \frac{1}{X^{19}}
        \end{equation*}
        where we used the definition of $I$ and the fact that $(\log X)^3 < \min( X^{1/10}, T)$. This is an acceptable contribution (see the definition of $E_2$ in \cref{prop:MRpart1}). In other words, the contribution corresponding to integration over $I$ in \cref{eq:perron} can be absorbed into $h(x)$.

        We now use the formula for $F(1+it)$ given in \cref{lem:expression2forF}. The error term $O(X^{-10})$ of that formula can trivially be absorbed into $h(x)$.

        Let us consider the term $F_2(1+it)$ appearing in \cref{lem:expression2forF}. By \cref{lem:MRParseval},
        \begin{equation*}
            \frac{1}{X} \int_{X}^{2X} \abs{\frac{1}{2\pi}\int_{[-X/2, X/2]\setminus I}F_2(1+it) \frac{(x+y)^{1+it} - x^{1+it}}{1+it} dt}^2 dx  \ll \max_{T\geq X/y} \frac{Xy}{T} \int_{-T}^{T} \abs{F_2(1+it)}^2 dt.
        \end{equation*}
        By the mean-value theorem \cite[Theorem~9.1]{IK}, this is
        \begin{equation*}
            \ll \max_{T\geq X/y} \frac{Xy}{T} (T+X) \sum_{n} \abs{\frac{e_n}{X}}^2 \log T \ll \max_{T\geq X/y} \frac{Xy}{T} (T+X) \frac{(\log X)^5}{XP_2} \ll \frac{y^2 (\log X)^5}{P_2}.
        \end{equation*}
        Thus, the contribution of $F_2(1+it)$ can be absorbed into $h(x)$, by definition of $E_1$.

        Hence, up to some terms that are suitably bounded in $L^2([X, 2X])$, the sum $\sum_{x<n\leq x+y} a_n$ equals
        \begin{equation}
            \label{eq:remainingint}
            \frac{1}{2\pi} \sum_{A,B,C} \int_{-(\log X)^3}^{(\log X)^3} \int_{-X/2}^{X/2}  G_{A,B,C}(t,u) \frac{(x+y)^{1+it} - x^{1+it}}{1+it} dt\, du
        \end{equation}
        where the range of $A,B,C$ is the same as in \cref{lem:expression2forF} and
        \begin{equation*}
            G_{A,B,C}(t,u) :=  \ind{\{\odist(t,\U) < 2X^{1/10}\}}Q_{2,A}(1+it+iu) Q_{3,B}(1+it+iu)R_C(1+it+iu) \widetilde{W_1}(iu) X^{iu}.
        \end{equation*}

        We split \cref{eq:remainingint} based on whether $t+u\in J$ or not (recall that $J$ is the set given in the statement of \cref{lem:MRstep1}).

        The contribution of $t+u\notin J$ can be absorbed into $h(x)$. Indeed, by Cauchy-Schwarz and \cref{lem:MRParseval} we have
        \begin{align*}
            \frac{1}{X} \int_{X}^{2X} & \Bigg\lvert \!\sum_{A,B,C} \int_{-(\log X)^3}^{(\log X)^3} \int_{-X/2}^{X/2}  \ind{t+u\notin J} G_{A,B,C}(t,u) \frac{(x+y)^{1+it} - x^{1+it}}{1+it} dt\, du \Bigg\rvert^2 dx \\
                                      & \ll (\log X)^{5}  \sum_{A,B,C} \int_{-(\log X)^3}^{(\log X)^3} \max_{X/y\leq T\leq X/2} \frac{Xy}{T}  \int_{-T}^{T}   \abs{ \ind{t+u\notin J} G_{A,B,C}(t,u) }^2 dt \, du.
        \end{align*}
        Using the trivial bound $\widetilde{W_1}(iu) X^{iu} \ll 1$ and changing variables $r:=t+u$, this becomes
        \begin{equation*}
            \ll (\log X)^{10} \max_{\substack{A,B,C\\ X/y\leq T\leq X/2}} \frac{Xy}{T} \int_{-2T}^{2T}  \ind{\{\odist(r,\U)\leq 4X^{1/10}\}} \ind{r\notin J} \abs{Q_{2,A}(1+ir) Q_{3,B}(1+ir)R_C(1+ir) }^2 dr,
        \end{equation*}
        which is an acceptable contribution of the form $E_3^*$.

        We can finally define $g(x)$ to be the remaining expression
        \begin{equation*}
            g(x) := \frac{1}{2\pi} \sum_{A,B,C} \int_{-(\log X)^3}^{(\log X)^3} \int_{-X/2}^{X/2}  \ind{t+u\notin J} G_{A,B,C}(t,u) \frac{(x+y)^{1+it} - x^{1+it}}{1+it} dt\, du.
        \end{equation*}
        By the triangle inequality, along with the trivial bounds $R_C(1+it+iu) \ll 1$ and~\cref{eq:trivboundit}, we have
        \begin{equation*}
            \norm{g}_{\infty} \ll y  \sum_{A,B} \int_{-(\log X)^3}^{(\log X)^3}  \int_{-X/2}^{X/2} \ind{t+u\in J}\abs{Q_{2,A}(1+it+iu) Q_{3,B}(1+it+iu) \widetilde{W_1}(iu)} dt\,  du.
        \end{equation*}
        Changing variables $r:=t+u$ and integrating over $u$ yields
        \begin{equation*}
            \norm{g}_{\infty} \ll y \sum_{A,B} \int_{-X}^{X} \ind{r\in J} \abs{Q_{2,A}(1+ir) Q_{3,B}(1+ir)} dr,
        \end{equation*}
        which is the claimed bound for $\norm{g}_{\infty}$.
    \end{proof}

    We can now prove \cref{prop:MRpart1} by choosing an appropriate set $J$ in \cref{lem:MRstep1} (namely, the set of all $t\in [-X, X]$ where one of the Dirichlet polynomials $Q_{2, A}(1+it)$ is large).

    \begin{proof}[Proof of \cref{prop:MRpart1}]
        We apply \cref{lem:MRstep1} with $J$ being the set $\V$ defined in \cref{eq:defV}. The announced bound for $\norm{g}_{\infty}$ is immediate.

        It remains to treat the term $E_3^*$. Using the bound for $Q_{2,A}$ in the definition of $\V$, we have
        \begin{equation}
            \label{eq:theboundforA2}
            E_3^* \ll \frac{X(\log X)^{10}y}{X^{2\nu}} \max_{\substack{A,B,C\\ X/y\leq T\leq X}} \frac{1}{T}  \int_{\substack{[-T, T]\quad\quad\ \ \\ \odist(t,\U)\leq 4X^{1/10}}} \abs{Q_{3,B}(1+it) R_C(1+it)}^2 dt.
        \end{equation}
        For any $1$-separated set $S\subset [-T, T]$, the Hal\'asz-Montgomery inequality \cite[Theorem~9.6]{IK} gives
        \begin{equation*}
            \sum_{t\in S} \abs{Q_{3,B}(1+it) R_C(1+it)}^2 \ll \frac{1}{BC} \big(BC + |S| T^{1/2}\big) \log T.
        \end{equation*}
        Applying this to bound the integral in \cref{eq:theboundforA2} (after discretising), we deduce that
        \begin{equation*}
            E_3^* \ll \frac{X(\log X)^{11}y}{X^{2\nu}} \max_{\substack{A,B,C\\ X/y\leq T\leq X}} \frac{1}{T} \bigg(1 + \frac{\abs{\{t\in [-T, T] \,:\, \odist(t, \U) \leq 4X^{1/10}\}}\cdot T^{1/2}}{BC}\bigg).
        \end{equation*}
        Observing that $(BC)^{-1} \asymp A/X \ll Q_2/X$ and treating $T$ trivially, we conclude that $E_3^*\ll E_3$.
    \end{proof}

    Next, we choose the set $\U$ in \cref{prop:MRpart1} to obtain a good bound for the quantities $E_2$ and $E_3$. This step corresponds to \cite[Theorem~9.2]{MRII}, but is technically much simpler in our context.

    \begin{proposition}
        \label{prop:MRpart2}
        Let $\nu > 0$ be sufficiently small.

        Let $X\geq 10^{100}$ and $10\leq P_1<Q_1\leq P_2<Q_2\leq P_3<Q_3\leq X^{1/3}$. Let $(a_n), (b_m)$ and $(c_p)$ be as in \cref{def:factorable}.

        Let $1\leq y\leq X^{1/2}$. Suppose that $P_1 \geq (\log X)^{1/\nu}$. Then
        \begin{equation*}
            \frac{1}{X} \int_{X}^{2X} \bigg\lvert{g(x)-\sum_{x<n\leq x+y} a_n} \bigg\rvert ^2 dx \ll \frac{y(y+Q_1)(\log X)^{11}}{P_1^{2\nu }} + (\log X)^2
        \end{equation*}
        for some measurable function $g$ obeying the uniform bound \cref{eq:boundforg}.
    \end{proposition}

    \begin{proof}
        We apply \cref{prop:MRpart1}. Recall that
        \begin{equation}
            \label{eq:recalldefE2}
            E_2 := \max_{X/y\leq T\leq X} \frac{Xy}{T} \int_{\substack{[-T, T]\qquad\  \\ \odist(t,\U) \geq X^{1/10} }} \bigg\lvert\sum_{n\geq 1} \frac{a_n}{n^{1+it}} W_2 \!\left(\frac{n}{X}\right)\!\bigg\rvert^2 dt.
        \end{equation}
        Taking out one prime from the interval $(P_1, Q_1]$ exactly as in the proof of \cref{lem:expression2forF}, we obtain the identity
        \begin{equation*}
            \sum_{n\geq 1}\frac{a_n}{n^{s}} W_2 \!\left(\frac{n}{X}\right) = \sum_{D,E} \frac{1}{2\pi} \int_{-(\log X)^3}^{(\log X)^3} Q_{1,D}(s+iu) R_E'(s+iu) \widetilde{W_2}(iu) X^{iu} du + F_2'(s) + O(X^{-10}),
        \end{equation*}
        where $D,E$ range over the powers of two such that $\tfrac{1}{16}X\leq DE \leq 8X$ and $D\in (\tfrac12 P_1, Q_1]$, and where
        \begin{empheq}[left=\empheqlbrace]{align*}
            Q_{1, D}(s) &:= \sum_{p \in (D, 2D]\cap (P_1, Q_1]} \frac{c_p}{p^s},\\
            R_E'(s) &:= \sum_{E<m\leq 2E} \frac{r_m'}{m^s},\\
            F_2'(s) &:= \sum_{\substack{X/4< n\leq 8X}} \frac{e_n'}{n^s}
        \end{empheq}
        for some complex coefficients $r_m', e_n'$ satisfying $|r_m'|\leq 1$ and $\sum_n |e_n'|^2 \ll X(\log X)^2/P_1$.

        As in the proof of \cref{lem:MRstep1}, the error term $O(X^{-10})$ is negligible, and the contribution of $F_2'(1+it)$ to \cref{eq:recalldefE2} is
        \begin{equation*}
            \ll \max_{T\geq X/y} \frac{Xy}{T} (T+X) \sum_{n} \abs{\frac{e_n'}{X}}^2 \log T \ll \max_{T\geq X/y} \frac{Xy}{T} (T+X) \frac{(\log X)^3}{XP_1} \ll \frac{y^2 (\log X)^3}{P_1}.
        \end{equation*}
        by the mean-value theorem \cite[Theorem~9.1]{IK}.

        For the main term of the identity, we use Cauchy-Schwarz and change variables as in the proof of \cref{lem:MRstep1}. This gives
        \begin{equation}
            \label{eq:interestforE2}
            E_2 \ll  \max_{\substack{X/y\leq T\leq X\\ D,E}} \frac{X(\log X)^8y}{T}  \int_{[-2T, 2T]\setminus \U}  \abs{Q_{1,D}(1+ir) R_{E}'(1+ir)}^2 dr + \frac{y^2 (\log X)^3}{P_1}.
        \end{equation}

        Defining
        \begin{equation}
            \label{eq:defU}
            \U := \Big\{t\in [-2X, 2X]\, :\, \max_D \abs{Q_{1,D}(1+it)} \geq P_1^{-\nu} \Big\},
        \end{equation}
        we have
        \begin{equation*}
            \int_{[-2T, 2T]\setminus \U}  \abs{Q_{1,D}(1+ir) R_{E}'(1+ir)}^2 dr \leq P_1^{-2\nu}  \int_{[-2T, 2T]}  \abs{R_{E}'(1+ir)}^2 dr.
        \end{equation*}
        By the mean-value theorem and the bound $E\asymp X/D \gg X/Q_1$,
        \begin{equation*}
            \int_{[-2T, 2T]}  \abs{R_{E}'(1+ir)}^2 dr \ll (T+E) \sum_m \abs{\frac{r_m'}{E}}^2\log T \ll \left(\frac{T}{E}+1\right) \log X \ll  \left(\frac{TQ_1}{X}+1\right) \log X.
        \end{equation*}
        Hence, the estimate \cref{eq:interestforE2} becomes
        \begin{equation*}
            E_2 \ll \max_{X/y\leq T\leq X} \frac{X(\log X)^9y}{P_1^{2\nu }} \left(\frac{Q_1}{X}+\frac{1}{T}\right) + \frac{y^2 (\log X)^3}{P_1} \ll \frac{y (Q_1+y)(\log X)^9}{P_1^{2\nu }},
        \end{equation*}
        which is acceptable.

        It remains to give a suitable bound for $E_3$, and for this it is enough to show that
        \begin{equation}
            \label{eq:Uissmall}
            \frac{ Q_2 \cdot \abs{\{t\in [-X, X] \,:\, \odist(t, \U) \leq 4X^{1/10}\}}}{(Xy)^{1/2}} \ll 1.
        \end{equation}
        Let $S$ be an arbitrary $1$-separated subset of $\U$. We bound the size of $S$ using the large value estimate \cite[Lemma~8]{MR} (which is a simple consequence of the mean-value theorem, see also \cite[Eq.~(9.30),~p.236]{IK}). We get
        \begin{equation*}
            \abs{S} \ll \sum_D \exp \! \left(2\nu \frac{(\log X)(\log P_1)}{\log D} + 2\nu \log P_1 + 2 \frac{\log X}{\log D} \log \log X\right),
        \end{equation*}
        where $D$ ranges over the powers of two in $(\tfrac12 P_1, Q_1]$. Using our assumption that $\log X \leq P_1^{\nu}$, we have $\log \log X \ll \nu \log D$ for any such $D$. Therefore, we have the simple bound
        \begin{equation*}
            \abs{S} \ll X^{O(\nu)} \log X.
        \end{equation*}
        Choosing $S$ to be a maximal $1$-separated subset of $\U$, we conclude that
        \begin{equation*}
            \big\lvert{\{t\in [-X, X] \,:\, \odist(t, \U) \leq 4X^{1/10}\}}\big\rvert \ll X^{1/10 + O(\nu)} \log X.
        \end{equation*}
        Since $Q_2 \leq X^{1/3}$, this establishes \cref{eq:Uissmall} (in a strong form) if $\nu$ is sufficiently small, which concludes the proof of \cref{prop:MRpart2}.
    \end{proof}

    \subsection{Average over Dirichlet characters}
    As mentioned in \cref{rem:MRpart1differences}, our application of \cref{prop:MRpart2} involves an extra average over Dirichlet characters. The following estimate will then be used to control the $g(x)$ term on average (compare with \cref{eq:boundforg}).

    \begin{lemma}
        \label{lem:boundgonaverage}
        Let $X$ be sufficiently large, and let $(\log X)^{-1} \leq \nu \leq 10^{-8}$.

        Let $X^{\nu^{1/8}} \leq P_2<Q_2<P_3<Q_3\leq X^{1/3}$. Let $q$ be a positive integer such that $q\leq X^{\nu}$.

        For a Dirichlet character $\chi\spmod q$, $D>0$ and $j\in \{2,3\}$, define
        \begin{equation*}
            Q_{j,D}(s,\chi) := \sum_{p\in (D, 2D]\cap (P_j, Q_j]} \frac{\chi(p) c_p}{p^s},
        \end{equation*}
        where $(c_p)$ is a sequence of $1$-bounded complex coefficients supported on primes.

        Then
        \begin{equation}
            \label{eq:boundinggonavg}
            \sum_{\chi \spmod{q}} \sum_{A,B} \int_{\V(\chi)} \abs{Q_{2,A}(1+it, \chi) Q_{3,B}(1+it, \chi)} dt \ll \left(1+\log \frac{\log Q_3}{\log P_2}\right)^2,
        \end{equation}
        where
        \begin{equation*}
            \V(\chi) := \Big\{t\in [-X, X]\, :\, \max_A \abs{Q_{2,A}(1+it,\chi)} \geq X^{-\nu} \Big\}
        \end{equation*}
        and, in these expressions, $A$ and $B$ range over the powers of two in $(\tfrac12 P_2, Q_2]$ and $(\tfrac12 P_3, Q_3]$, respectively.
    \end{lemma}

    The proof of \cref{lem:boundgonaverage} follows \cite[p.92]{MRTTZ}, which itself generalises some ideas from the proof of \mbox{\cite[Proposition~8.3]{MRII}} to the case $q>1$.

    \begin{proof}
        It suffices to prove that
        \begin{equation}
            \label{eq:discreteschi}
            \sum_{\chi \spmod{q}} \sum_{A,B} \sum_{t\in S_{\chi}} \abs{Q_{2,A}(1+it, \chi) Q_{3,B}(1+it, \chi)} \ll \left(1+\log \frac{\log Q_3}{\log P_2}\right)^2
        \end{equation}
        where, for every $\chi \spmod q$, $S_{\chi}$ is a $1$-separated subset of $\V(\chi)$. Define
        \begin{equation*}
            M_{j,D} := \sum_{\chi \spmod q} \sum_{t\in S_{\chi}} \abs{Q_{j,D}(1+it, \chi)}^2
        \end{equation*}
        for $j=2,3$ and $D>10$. By Cauchy-Scwharz, the left-hand side of \cref{eq:discreteschi} is
        \begin{equation*}
            \leq \bigg(\sum_A \frac{1}{\log A}\bigg)^{1/2} \bigg(\sum_A  (\log A) M_{2, A} \bigg)^{1/2} \bigg(\sum_B \frac{1}{\log B}\bigg)^{1/2} \bigg(\sum_B (\log B) M_{3,B}\bigg)^{1/2}.
        \end{equation*}
        We will show that each of these terms is $\ll (1+\log \frac{\log Q_3}{\log P_2})^{1/2}$. This is clear for the first and third terms.

        To bound the other terms, we use \cite[Lemma~6.6]{MRTTZ}\footnote{The inequality \cite[Lemma~6.6]{MRTTZ} corresponds to the case $k=1$ of \cref{lem:largevalues}, except that it accounts for the sparsity of the primes. It can be obtained by inserting a linear sieve upper bound in the proof of \cref{lem:largevalues}, as in \cite[Proof~of~Lemma~4.4]{MRII}.} which implies that, uniformly for $\eta\in (0, 1/2)$,
        \begin{equation*}
            M_{j,D}  \ll \frac{1}{(\log D)^2} + \frac{q^{\eta} D^{-\eta/2}}{\log D}\left(X^{5\eta^{3/2}} \log X + \eta^{-1}\right) \sum_{\chi \spmod q} |S_{\chi}|
        \end{equation*}
        for $j=2, 3$ and $D\geq 10$. In particular,
        \begin{align*}
            \max \bigg(\sum_{A} (\log A) & M_{2, A},\  \sum_B (\log B) M_{3, B} \bigg)                                                                                         \\
                                         & \ll 1+ \log \frac{\log Q_3}{\log P_2} + q^{\eta} P_2^{-\eta/2} X^{5\eta^{3/2}} (\log X)^2 \eta^{-1}\sum_{\chi \spmod q} |S_{\chi}|.
        \end{align*}
        On the other hand, by definition of $\V(\chi)$, we have
        \begin{equation*}
            \sum_{A} (\log A) M_{2, A}\geq \sum_{A} M_{2,A} \geq  X^{-2\nu} \sum_{\chi \spmod q} |S_{\chi}|.
        \end{equation*}
        We choose $\eta := \nu^{1/2}$. Since $(\log X)^{-1} \leq \nu \leq 10^{-8}$, this ensures that
        \begin{equation*}
            X^{-2\nu} > q^{\eta} P_2^{-\eta/2} X^{5\eta^{3/2}} (\log X)^3 \eta^{-1}
        \end{equation*}
        if $X$ is sufficiently large, using the bounds for $P_2$ and $q$ in the statement (the term $P_2^{\eta}\geq X^{\nu^{5/8}}$ dominates). Thus, with this choice of $\eta$, we must have
        \begin{equation*}
            \max \bigg(\sum_{A} (\log A) M_{2, A},\  \sum_B (\log B) M_{3, B} \bigg) \ll 1+ \log \frac{\log Q_3}{\log P_2}
        \end{equation*}
        as desired.
    \end{proof}

    \subsection{Rewriting correlations in terms of characters} Suppose that a multiplicative function $f(n)$ correlates with a linear phase $e(n\alpha_x)$ over many short intervals $[x, x+H)$, where $\alpha_x \approx \frac{a}{q} + \frac{T}{x}$. We show that this forces $f$ to correlate with $\chi(n)n^{2\pi i T}$ for many Dirichlet characters $\chi\spmod{q}$, over many slightly shorter intervals. This is done in two steps, \cref{lem:taylor,lem:coprimality}, which are heavily inspired by \cite[p.40]{MRTTZ} and \cite[p.90-91]{MRTTZ} respectively.

    \begin{lemma}
        \label{lem:taylor}
        Let $1\leq H'\leq H\leq X$, $\delta>0$ and $\kappa \geq 1$. Assume that $H'\leq c_0 \delta H/\kappa$ for some sufficiently small absolute constant $c_0>0$.

        Let $(b_n)$ be a $1$-bounded sequence of complex numbers. Suppose that there exists an $H$-separated set $S \subset [X, 2X]$ such that, for every $x\in S$,
        \begin{equation}
            \label{eq:boundE11}
            \bigg\lvert \sum_{x\leq n < x+H} b_n e(n\beta_x) \bigg\rvert \geq \delta H,
        \end{equation}
        where the frequencies $\beta_x$ satisfy $\norm{\beta_x - \frac{T}{x}} \leq \frac{\kappa}{H}$ for some $T\in \R$ with $|T|\leq \kappa X^2/H^2$.

        Then, there is an $H'$-separated set $S' \subset [X, 2X]$ of size $\gg \delta \frac{H}{H'} |S|$ such that, for every $x\in S'$,
        \begin{equation*}
            \bigg\lvert \sum_{x\leq n < x+H'} b_n n^{2\pi i T} \bigg\rvert \gg \delta H'.
        \end{equation*}
    \end{lemma}

    \begin{proof}
        Fix $x\in S$. For any integer $0\le h < H'$, we have
        \begin{equation*}
            \sum_{x\le n < x+H} b_{n+h} e((n+h)\beta_x)
            = \sum_{x\le n < x+H} b_n e(n\beta_x) + O(H').
        \end{equation*}
        Averaging over $h$ and applying the triangle inequality gives
        \begin{equation*}
            \frac{1}{H'} \sum_{x\le n < x+H}
            \bigg|\sum_{0\le h < H'} b_{n+h} e((n+h)\beta_x)\bigg|
            \ge \delta H - O(H') \ge \tfrac12 \delta H,
        \end{equation*}
        provided $c_0$ is chosen sufficiently small. It follows that there are $\gg \delta H$ integers $n\in[x,x+H)$ such that
        \begin{equation}
            \label{eq:xinsprime}
            \bigg|\sum_{0\le h < H'} b_{n+h} e((n+h)\beta_x)\bigg|
            \ge \tfrac14 \delta H'.
        \end{equation}
        From these integers, one may select an $H'$-separated subset of $[x,x+H)$ of size $\gg \delta H/H'$ on which the above inequality holds. Since this is true for all $x\in S$ and $S$ is $H$-separated, we deduce that there exists an $H'$-separated set $S' \subset [X, 2X]$ of size $\gg \delta \frac{H}{H'} |S|$ such that \cref{eq:xinsprime} holds for every $n\in S'$.

        Note that for $0 \le h < H'$ and $n\in [x, x+H)$, we have
        \begin{equation*}
            e(h\beta_x)
            = e\Big(h\Big(\beta_x - \frac{T}{x}\Big)\Big)
            e\Big(h\Big(\frac{T}{x}-\frac{T}{n}\Big)\Big)
            e\Big(\frac{hT}{n}\Big)
            = \Big(1+O\Big(\frac{\kappa H'}{H}\Big)\Big)
            e\Big(\frac{hT}{n}\Big)
        \end{equation*}
        using that $|T| \leq \kappa X^2/H^2$. Since $\kappa H' \le c_0 \delta H$ for sufficiently small $c_0>0$,
        we deduce from \cref{eq:xinsprime} that
        \begin{equation*}
            \bigg|\sum_{0\le h < H'} b_{n+h} e\Big(\frac{hT}{n}\Big)\bigg|
            \ge \tfrac18 \delta H'
        \end{equation*}
        for every $n \in S'$. By Taylor expansion,
        \begin{equation*}
            (n+h)^{2\pi iT}
            = n^{2\pi iT}\Big(1+\frac{h}{n}\Big)^{2\pi iT}
            = n^{2\pi iT} e\Big(\frac{hT}{n}\Big)
            + O\Big(\frac{\kappa (H')^2}{H^2}\Big),
        \end{equation*}
        and since $\kappa (H')^2/H^2 \ll \delta$, it follows that
        \begin{equation*}
            \bigg|\sum_{0\le h < H'} b_{n+h} (n+h)^{2\pi iT}\bigg|
            \gg \delta H'
        \end{equation*}
        for every $n \in S'$, as required.
    \end{proof}

    \begin{lemma}
        \label{lem:coprimality}
        Let $1\leq H\leq X$. Let $q\in \N$ and $a\in \Z$ be such that $(a,q)=1$. Let $f:\N \to \C$ be a $1$-bounded function such that $f(mn)=f(m)f(n)$ whenever $(m,n)=1$ and $\rad(m)\mid q$.

        Let $\eps, \delta>0$ and assume that $C_0^{\delta^{-1}}q^2\leq H$ for some large enough constant $C_0=C_0(\eps)$. Suppose that
        \begin{equation}
            \label{eq:assumpE12}
            \bigg\lvert \sum_{x\leq n<x+H} f(n) e\Big(\frac{na}{q}\Big) \bigg\rvert \geq \delta H
        \end{equation}
        for all $x$ in an $H$-separated set $S \subset [X, 2X]$.

        Then, there exists an integer $1\leq d\ll_{\eps} \delta^{-2}q^{\eps}$ and an $H/d$-separated set $S' \subset [X/d, 2X/d]$ of size $|S'| \gg_{\eps} \delta |S|$ such that
        \begin{equation*}
            \sum_{\chi \spmod q}\bigg\lvert \sum_{\substack{x\leq n<x+H/d}} f(n) \chi(n) \bigg\rvert \gg_{\eps} \delta^3 q^{1/2-\eps} \frac{H}{d}
        \end{equation*}
        for all $x\in S'$.
    \end{lemma}

    \begin{proof}
        We begin by excluding integers with unusually large prime power divisors coming from primes dividing $q$.
        Let $A := C(\eps)\delta^{-1}$, where $C(\eps)\ge 1$ is a sufficiently large constant to be chosen later.

        Let $E$ be the set of integers divisible by $p^k$ for some prime $p\mid q$ and exponent $k\ge 2$ such that
        $p^k \ge A^2$. For any $x\in[X,2X]$, we have
        \begin{equation*}
            |E\cap[x,x+H)|
            \le \sum_{p\mid q}\;\min_{\substack{k\ge 2\\ p^k\ge A^2}}
            \left(\frac{H}{p^k}+1\right)
            \ll q + H\left(\sum_{p\le A}\frac{1}{A^2} + \sum_{p>A}\frac{1}{p^2}\right)
            \ll q+\frac{H}{A}.
        \end{equation*}
        If $C(\eps)$ is chosen sufficiently large, then for every $x\in S$, the hypothesis
        \cref{eq:assumpE12} implies
        \begin{equation}
            \label{eq:E1E2cond}
            \Bigg|\sum_{\substack{x\le n<x+H\\ n\notin E}} f(n)e\Big(\frac{na}{q}\Big)\Bigg|
            \gg \delta H.
        \end{equation}

        Every $n\ge 1$ can be written uniquely as $n=dm$ with $(m,q)=1$ and $\rad(d)\mid q$.
        Using the multiplicativity assumption on $f$, we may rewrite
        \cref{eq:E1E2cond} as
        \begin{equation}
            \label{eq:mdseparated}
            \Bigg|
            \sum_{\substack{d\notin E\\ \rad(d)\mid q}}
            f(d)\!\!\!\sum_{\substack{x/d\le m<(x+H)/d\\ (m,q)=1}}
            f(m)e\Big(\frac{mad}{q}\Big)
            \Bigg|
            \gg \delta H
        \end{equation}
        for every $x\in S$.

        We first bound the contribution of large $d$. By definition of $E$, for every $d\in \N\setminus E$ with $\rad(d)\mid q$, we have the upper bound
        \begin{equation*}
            d \leq q \prod_{p\leq A} A^2 \leq q e^{O(A)},
        \end{equation*}
        and thus $d \leq H$ by the assumption $C_0^{\delta^{-1}}q^2\leq H$ in the statement. Hence, the contribution of those $d$ with $d \geq A^2q^{\eps}$ to the left-hand side of \cref{eq:mdseparated} is trivially bounded by
        \begin{equation*}
            \sum_{\substack{A^2q^{\eps}\leq d\leq H\\ \rad(d)\mid q}} \frac{H}{d} \leq \frac{H}{Aq^{\eps/2}} \sum_{\substack{d\geq 1\\ \rad(d) \mid q}} \frac{1}{d^{1/2}} = \frac{H}{Aq^{\eps/2}} \prod_{p\mid q}\frac{1}{1-p^{-1/2}}\ll_{\eps} \frac{H}{A}.
        \end{equation*}
        This is negligible by our choice of $A$ if $C(\eps)$ is sufficiently large.

        Therefore, for every $x\in S$,
        \begin{equation*}
            \Bigg\lvert \sum_{\substack{1\leq d< A^2 q^{\eps}\\ \rad(d)\mid q\\ d\notin E}} f(d) \sum_{\substack{x/d\leq m< (x+H)/d\\ (m,q)=1}} f(m) e\Big(\frac{mad}{q}\Big) \Bigg\rvert \gg \delta H.
        \end{equation*}
        Since the inner sum is trivially $\ll \frac{\phi(q)}{q} \frac{H}{d}$ and
        \begin{equation*}
            \sum_{\substack{d\geq 1\\ \rad(d)\mid q}} \frac{1}{d} = \prod_{p\mid q} \frac{p}{p-1} = \frac{q}{\phi(q)},
        \end{equation*}
        an averaging argument shows that there exists an integer $d$ (with $1\leq d<A^2q^{\eps}$) such that
        \begin{equation}
            \label{eq:dnolonger}
            \Bigg\lvert \sum_{\substack{x/d\leq m< (x+H)/d\\ (m,q)=1}} f(m) e\Big(\frac{mad}{q}\Big) \Bigg\rvert \gg \delta \frac{\phi(q)}{q} \frac{H}{d}
        \end{equation}
        for $\gg \delta |S|$ values of $x\in S$. Since $S$ is $H$-separated, the corresponding values of $x/d$ form an $H/d$-separated set $S'\subset[X/d, 2X/d]$ with $|S'|\gg \delta |S|$.

        Expanding the sum into congruence classes modulo $q$ and then into Dirichlet characters, we have
        \begin{equation*}
            \sum_{\substack{y\leq m< y+H/d                                                                       \\ (m,q)=1}} f(m) e\Big(\frac{mad}{q}\Big) = \frac{1}{\phi(q)}\sum_{\chi \spmod q} c_{\overline{\chi}}(ad) \sum_{y\leq m< y+H/d}f(m)\chi(m),
        \end{equation*}
        where
        \begin{equation*}
            c_{{\chi}}(x) := \sum_{u\in (\Z/q\Z)^{\times}}{\chi(u)} e\Big(\frac{ux}{q}\Big).
        \end{equation*}
        By \cite[Theorem~9.12]{MV}, writing $q' := q/(q,d)$, we have
        \begin{equation*}
            \abs{c_{\overline{\chi}}(ad)} \leq \frac{\phi(q)\sqrt{q'}}{\phi(q')}.
        \end{equation*}
        Combining this with \cref{eq:dnolonger}, we obtain, for all $y\in S'$,
        \begin{equation*}
            \sum_{\chi \spmod q}\bigg\lvert \sum_{y\leq m< y+H/d}f(m)\chi(m) \bigg\rvert \gg \delta \frac{\phi(q)}{q}\frac{\phi(q')}{\sqrt{q'}} \frac{H}{d}.
        \end{equation*}
        Recalling that $d\ll_{\eps} \delta^{-2} q^{\eps}$, we see that $q'\gg_{\eps} \delta^2 q^{1-\eps}$. Since $\phi(n)\gg n/\log \log 5n$ for $n\geq1$, we deduce that
        \begin{equation*}
            \sum_{\chi \spmod q}\bigg\lvert \sum_{y\leq m< y+H/d}f(m)\chi(m) \bigg\rvert \gg_{\eps} \delta^3 \frac{q^{(1-\eps)/2}}{(\log \log 5q)^2} \frac{H}{d}
        \end{equation*}
        for all $y\in S'$, giving the result.
    \end{proof}

    \subsection{From correlations in short intervals to pretentiousness}
    Finally, we bring together the results from the previous sections to prove \cref{thm:MR}. We emphasize that all proof ideas originate from \cite{MRII} and \cite{MRTTZ}; our objective here is to provide a comprehensive derivation, explicitly stating the quantitative aspects that were previously scattered or left to the reader to adapt.

    \begin{theorem}
        \label{thm:MR}
        Let $f : \N \to \C$ be a $1$-bounded multiplicative function. Let $10\leq H\leq X$. Let~$a, q\in \N$ be coprime integers. Let $T\in \R$ with $|T|\leq \kappa X^2/H^2$ for some $\kappa\geq 1$.

        Let $0<\delta<1/2$. Suppose that
        \begin{equation*}
            \bigg\lvert \sum_{x\leq n < x+H} f(n) e(n\alpha_x) \bigg\rvert \geq \delta H
        \end{equation*}
        for all $x$ in an $H$-separated set $S \subset [X, 2X]$, where $\alpha_x\in \R/\Z$ satisfy the approximate formula
        \begin{equation*}
            \alpha_x = \frac{a}{q} + \frac{T}{x} + O \!\left(\frac{\kappa}{H}\right) \pmod 1.
        \end{equation*}

        Assume that
        \begin{equation}
            \label{eq:paramassumpt}
            \kappa^{C}\left(q \frac{|S|}{X/H}\log X\right)^{\delta^{-C}} \leq H \leq X^{\delta^{C}}
        \end{equation}
        where $C$ is a sufficiently large absolute constant.

        Then, $q\ll \delta^{-10}$ and $f$ is pretentious in the sense that
        \begin{equation*}
            \sum_{p\leq X} \frac{1-\Re(f(p)\chi(p)p^{it})}{p} \ll \log (\delta^{-1})
        \end{equation*}
        for some Dirichlet character $\chi \spmod q$ and some real number $t = O(\kappa X^2/H^2)$.
    \end{theorem}

    \Cref{thm:MR} mainly relies on \cref{prop:MRpart2}, which requires restricting to integers $n$ with prime factors in certain ranges. To do this, we use the following result from \cite{MRII}.

    \begin{lemma}
        \label{lem:sieve}
        Let $X>0$ and $10\leq P<Q\leq X^{3/4}$. Let $2\leq H\leq \tfrac12 X^{1/6}$. Let $S \subset [X, 2X]$ be an $H$-separated set. Then, for all but
        \begin{equation*}
            \ll \frac{\log Q}{\log P} \cdot \frac{X(\log H)^2}{H^2}
        \end{equation*}
        elements $x\in S$, the number of integers $n\in [x, x+H)$ without any prime factor in $(P, Q]$ is
        \begin{equation*}
            \ll \frac{\log P}{\log Q} H.
        \end{equation*}
    \end{lemma}

    \begin{proof}
        Let $E\subset \N$ be the set of integers without any prime factor in $(P, Q]$.

        We apply \cite[Proposition~10.4]{MRII} with interval length $h := 2H$, taking $f$ to be the constant function $1$ and $\Delta := \log P/\log Q$. Note that $h_1 = H(f; X) = 1$ (see \cite[p.7]{MRII} for the definition of $H(f;X)$). We obtain that, for all but
        \begin{equation*}
            \ll \frac{X(\log H)^2}{\Delta H}
        \end{equation*}
        integers $y\in [X, 2X]$, it holds that $\abs{[y, y+2H)\cap E} \ll \Delta H$.

        Noting that $[x, x+H) \cap E \subset [y, y+2H)\cap E$ for any $x\in S$ and $y\in [x-H, x]$, and recalling that $S$ is $H$-separated, the conclusion follows.
    \end{proof}

    \begin{proof}[Proof of \cref{thm:MR}]
        We first restrict to integers $n$ with suitably sized prime factors. Let $c>0$ be a small absolute constant to be chosen later. The constant $C$ in the statement of the theorem is assumed to be sufficiently large in terms of $c$. Define $Q_1 = H^{1/2}$, $Q_2 = X^{c\delta/3}$, $Q_3 = X^{1/3}$ and, for ${j\in \{1,2,3\}}$, let $P_j := Q_j^{c\delta}$. Let $\NN$ be the set of positive integers $n$ having at least one prime factor in $(P_j, Q_j]$ for each $j\in \{1,2,3\}$, and no repeated prime factor in these ranges.

        By \cref{lem:sieve}, for each $j\in \{1,2,3\}$, there are $\ll X(\log X)^3/H^2$ elements $x\in S$ such that
        \begin{equation*}
            \sum_{x\leq n < x+H} \ind{p\mid n\Rightarrow p\notin (P_j, Q_j]} \geq \tfrac18 \delta H,
        \end{equation*}
        provided that $c$ is sufficiently small.

        Furthermore, the number of integer $n\in [X, 3X]$ divisible by $p^2$ for some $P_1<p\leq Q_3$ is $\ll X/P_1$. Hence, at most $O(\delta^{-1} P_1^{-1} X/H)$ elements $x\in S$ have the property that $[x, x+H)$ contains $\geq \tfrac18 \delta H$ such integers $n$, since $S$ is $H$-separated.

        We have thus shown the existence of a subset $S_1 \subset S$ of size
        \begin{equation*}
            |S_1| \geq |S| - O\bigg(\frac{X(\log X)^3}{H^2} + \frac{X}{\delta P_1 H}\bigg) \gg |S|,
        \end{equation*}
        (using \cref{eq:paramassumpt}) such that all $x\in S_1$ satisfy
        \begin{equation*}
            \Bigg\lvert \sum_{x\leq n < x+H} \ind{n\in \NN}f(n) e(n\alpha_x) \Bigg\rvert \geq \tfrac12 \delta H.
        \end{equation*}

        Let $H' := \lfloor c\delta \kappa^{-1} H \rfloor$. By \cref{lem:taylor} applied with $b_n := \ind{n\in \NN}f(n)e(na/q)$ and $\beta_x := \alpha_x - \frac{a}{q}$, there is a subset $S_2 \subset S_1$ of size $|S_2| \gg \delta \frac{H}{H'} |S_1|$ such that, for every $x\in S_2$,
        \begin{equation*}
            \bigg\lvert \sum_{x\leq n < x+H'} \ind{n\in \NN}f(n) e\Big(\frac{na}{q}\Big) n^{2\pi i T} \bigg\rvert \gg \delta H'.
        \end{equation*}

        We apply \cref{lem:coprimality} with $f(n)$ replaced by $\ind{n\in \NN}f(n)n^{2\pi i T}$, $H$ replaced by $H'$ and $\eps := 1/6$. Note that the weak multiplicativity assumption of \cref{lem:coprimality} holds since $q< P_1$ (by \cref{eq:paramassumpt}) and the definition of $\NN$ only involves primes larger than $P_1$. Thus, there is an integer $1\leq d\ll \delta^{-2}q^{1/6}$ and an $H'/d$-separated set $S_3 \subset [X/d, 2X/d]$ of size $|S_3| \gg \delta |S_2|$ such that
        \begin{equation}
            \label{eq:outcomeexpchi}
            \sum_{\chi \spmod q}\Bigg\lvert \sum_{\substack{x\leq n<x+H'/d}} \ind{n\in \NN} f(n) \chi(n)n^{2\pi i T} \Bigg\rvert \gg \delta^3 q^{1/3} \frac{H'}{d}
        \end{equation}
        for all $x\in S_3$.

        Let $\widetilde{X} := X/d$ and $\widetilde{H} := H'/d$. Define $\nu := (c\delta/3)^{16}$. For every $\chi \spmod{q}$, the sequence
        \begin{equation*}
            (a_n) := \big(\ind{n\in \NN} f(n) \chi(n) n^{2\pi i T}\big)_{\widetilde{X}/4 < n \leq 8\widetilde{X}},
        \end{equation*}
        satisfies the properties in \cref{def:factorable}, by definition of $\NN$ and multiplicativity of $f(n) \chi(n) n^{2\pi i T}$ (with $c_p=f(p) \chi(p) p^{2\pi i T}$). We can thus apply \cref{prop:MRpart2,lem:boundgonaverage} to obtain a family of functions $g_{\chi} : [\widetilde{X}, 2\widetilde{X}]\to \C$ such that
        \begin{equation}
            \label{eq:outcomeMRchi}
            \frac{1}{\widetilde{X}} \int_{\widetilde{X}}^{2\widetilde{X}} \Bigg\lvert{g_{\chi}(x)-\sum_{x<n\leq x+\widetilde{H}}\ind{n\in \NN} f(n) \chi(n)n^{2\pi i T}} \Bigg\rvert ^2 dx \ll \frac{\widetilde{H}(\widetilde{H}+Q_1)(\log X)^{11}}{P_1^{2\nu }} + (\log X)^2
        \end{equation}
        and
        \begin{equation}
            \label{eq:sumchibd}
            \sum_{\chi \spmod{q}} \norm{g_{\chi}}_{\infty} \ll \left(1+\log \frac{\log Q_3}{\log P_2}\right)^2 \widetilde{H}  \ll \big(\!\log \delta^{-1}\big)^{2} \widetilde{H}.
        \end{equation}
        Observe that the conditions $P_1 \geq (\log \widetilde{X})^{1/\nu}$, $P_2 \geq \widetilde{X}^{\nu^{1/8}}$, $q \leq \widetilde{X}^{\nu}$ and $\nu \geq 1/\log \widetilde{X}$ of \cref{prop:MRpart2,lem:boundgonaverage} are all satisfied by the assumption \cref{eq:paramassumpt} for our choice $\nu = (c\delta/3)^{16}$, if $C$ is sufficiently large in terms of $c$.

        Noting that $\widetilde{H} \gg c\delta^3q^{-1/6}\kappa^{-1}H \gg Q_1$ by \cref{eq:paramassumpt}, the right-hand side of \cref{eq:outcomeMRchi} is bounded by $O({\widetilde{H}}^2 (\log X)^{11}/P_1^{2\nu})$. By Cauchy-Schwarz, this implies that
        \begin{equation*}
            \frac{1}{\widetilde{X}} \int_{\widetilde{X}}^{2\widetilde{X}} \sum_{\chi \spmod q}\Bigg\lvert{g_{\chi}(x) - \sum_{x<n\leq x+\widetilde{H}}\ind{n\in \NN} f(n) \chi(n)n^{2\pi i T}} \Bigg\rvert dx \ll q \frac{\widetilde{H}(\log X)^{11/2}}{P_1^{\nu}} \ll \frac{\widetilde{H}}{P_1^{\nu/2}}.
        \end{equation*}
        Consequently, we have
        \begin{equation*}
            \sum_{\chi \spmod q}\Bigg\lvert{g_{\chi}(x) - \sum_{x<n\leq x+\widetilde{H}}\ind{n\in \NN} f(n) \chi(n)n^{2\pi i T}} \Bigg\rvert \leq \widetilde{H}
        \end{equation*}
        for all but $\ll P_1^{-\nu/2} \widetilde{X}$ integers $x\in [\widetilde{X}, 2\widetilde{X}]$. By \cref{eq:sumchibd}, we conclude that
        \begin{equation}
            \label{eq:conclusionremoveg}
            \sum_{\chi \spmod q}\Bigg\lvert{\sum_{x<n\leq x+\widetilde{H}}\ind{n\in \NN} f(n) \chi(n)n^{2\pi i T}} \Bigg\rvert \ll \big(\!\log \delta^{-1}\big)^{2} \widetilde{H}
        \end{equation}
        for all but $\ll P_1^{-\nu/2} \widetilde{X}$ integers $x\in [\widetilde{X}, 2\widetilde{X}]$.

        On the other hand, by the triangle inequality, the lower bound \cref{eq:outcomeexpchi} continues to hold whenever $x$ is at distance at most $c\delta^3 q^{-2/3} \widetilde{H}$ from an element of $S_3$, if $c$ is sufficiently small. Since $S_3$ is $\widetilde{H}$-separated, we conclude that~\cref{eq:outcomeexpchi} holds for
        \begin{equation*}
            \gg |S_3| \delta^3 q^{-2/3} \widetilde{H} \gg \delta^5 q^{-2/3} \frac{|S|}{X/H} \widetilde{X}
        \end{equation*}
        integers $x\in [\widetilde{X}, 2\widetilde{X}]$. By our assumption \cref{eq:paramassumpt}, this contradicts \cref{eq:conclusionremoveg} if $C$ is sufficiently large, unless $\delta^3 q^{1/3} \ll (\log \delta^{-1})^{2}$. We have thus proved the key estimate $q\ll \delta^{-10}$.

        Applying \cref{lem:taylor} and \cref{lem:coprimality} as in the beginning of this proof without first restricting to~$\NN$, we obtain the following analogue of~\cref{eq:outcomeexpchi}: for some integer $1\leq d'\ll \delta^{-2}q^{1/6} \ll \delta^{-4}$, we have
        \begin{equation*}
            \sum_{\chi \spmod q}\Bigg\lvert \sum_{\substack{x\leq n<x+H'/d'}} f(n) \chi(n)n^{2\pi i T} \Bigg\rvert \gg \delta^3 q^{1/3} \frac{H'}{d'}
        \end{equation*}
        for all $x$ in an $H'/d'$-separated set $S' \subset [X/d', 2X/d']$ of size $|S'| \gg \delta^2 \frac{H}{H'} |S|$. By the pigeonhole principle, there is a character $\chi \spmod{q}$ such that
        \begin{equation}
            \label{eq:shortcorl}
            \frac{1}{H'/d'}\Bigg\lvert \sum_{\substack{x\leq n<x+H'/d'}}  f(n) \chi(n)n^{2\pi i T} \Bigg\rvert \gg \frac{\delta^3q^{1/3}}{q} \gg\delta^{10}
        \end{equation}
        for $\gg |S'|/q\gg \delta^{12} \frac{H}{H'} |S|$ elements $x\in S'$. By the triangle inequality and the $H'/d'$-separation of $S'$, the estimate \cref{eq:shortcorl} holds for
        \begin{equation*}
            \gg \delta^{10}\frac{H'}{d'} \cdot  \delta^{12} \frac{H}{H'} |S| \gg \delta^{22} \frac{|S|}{X/H} \frac{X}{d'}
        \end{equation*}
        integers $x\in [X/d', 2X/d']$.

        We can now conclude the proof using the complex-valued Matomäki-Radziwiłł theorem with power savings and Halasz's theorem. By the Matomäki-Radziwiłł theorem \cite[Theorem~1.7]{MRII}, there is an absolute constant $C_1\geq 1$ and a real number $t_0\in [-X/d', X/d']$ such that the following holds: for any $0<\delta_1<1/1000$, we have
        \begin{equation}
            \label{eq:lblongsum}
            \frac{1}{X/d'}\Bigg\lvert\sum_{X/d'< n\leq 2X/d'} f(n) \chi(n)n^{2\pi i (T-t_0)} \Bigg\rvert \gg \delta^{10} - O\!\left(\delta_1 + \frac{\log \log (H'/d')}{\log (H'/d')} + \frac{1}{(\log (X/d'))^{3/1000}}\right)
        \end{equation}
        provided that
        \begin{equation}
            \label{eq:condbigDelta}
            C_1 \bigg(\frac{1}{(H'/d')^{\delta_1/15}} + \frac{1}{(X/d')^{\delta_1^4/10^{16}}}\bigg) < \delta^{22} \frac{|S|}{X/H}.
        \end{equation}
        Using the assumption \cref{eq:paramassumpt}, we see that the condition \cref{eq:condbigDelta} is satisfied if $\delta_1$ is chosen to be a sufficiently small multiple of $\delta^{10}$, and \cref{eq:lblongsum} simplifies to
        \begin{equation}
            \label{eq:lblongsumsimplified}
            \frac{1}{X/d'}\Bigg\lvert\sum_{X/d'< n\leq 2X/d'} f(n) \chi(n)n^{2\pi i (T-t_0)} \Bigg\rvert \gg \delta^{10}.
        \end{equation}

        For $A,B>0$, define
        \begin{equation*}
            m(A,B) := \min_{|\tau| \leq B} \sum_{p\leq A} \frac{1-\Re(f(p)\chi(p)p^{2\pi i (T-t_0)}p^{-i\tau})}{p}.
        \end{equation*}
        By Halasz's theorem \cite[Corollary~4.12,~p.494]{tenenbaum}, for every $B\geq 2$ we have
        \begin{equation}
            \label{eq:uplongsum}
            \frac{1}{X/d'}\Bigg\lvert\sum_{X/d'< n\leq 2X/d'} f(n) \chi(n)n^{2\pi i (T-t_0)} \Bigg\rvert \ll \frac{1+m(X/d', B)}{e^{m(X/d', B)}}+\frac{1}{B}.
        \end{equation}
        Setting $B := X$, we deduce from \cref{eq:lblongsumsimplified,eq:uplongsum} that $m(X/d',X) \ll \log (\delta^{-1})$. Since
        \begin{equation*}
            \sum_{X/d'\leq p\leq X} \frac{1}{p} \ll 1,
        \end{equation*}
        we get $m(X,X) \ll \log (\delta^{-1})$, which concludes the proof of \cref{thm:MR}.
    \end{proof}
\end{appendices}

\providecommand{\bysame}{\leavevmode\hbox to3em{\hrulefill}\thinspace}
\providecommand{\MR}{\relax\ifhmode\unskip\space\fi MR }
\providecommand{\MRhref}[2]{\href{http://www.ams.org/mathscinet-getitem?mr=#1}{#2}
}
\providecommand{\href}[2]{#2}

\bibliographystyle{amsplain}

\end{document}